\documentclass{amsart}

\usepackage[margin=1.5in]{geometry}

\usepackage{amssymb,amsmath}
\usepackage{stmaryrd}
\usepackage{yhmath}
\usepackage{latexsym}
\usepackage{amsthm}
\usepackage{amscd}
\usepackage{mathrsfs}
\usepackage[all]{xy}
\usepackage{mathtools}
\usepackage{arydshln}
\usepackage{bbm}
\usepackage{bm}
\usepackage{url}
\usepackage{color}
\usepackage{enumitem}
\usepackage{colonequals}
\usepackage[colorlinks=false,hidelinks]{hyperref}
\usepackage{tikz-cd}
\usepackage{varwidth}

\DeclarePairedDelimiter\floor{\lfloor}{\rfloor}

\newcommand{\sym}{\mathrm{sym}}
\newcommand{\ur}{\mathrm{ur}}
\newcommand{\ram}{\mathrm{ram}}

\newcommand{\red}{\mathrm{red}}
\newcommand{\ad}{\mathrm{ad}}

\newcommand{\reg}{\mathrm{reg}}
\newcommand{\nreg}{\mathrm{nreg}}

\newcommand{\ol}{\overline}

\newcommand{\vreg}{\mathrm{vreg}}
\newcommand{\nvreg}{\mathrm{nvreg}}

\newcommand{\Jac}[2]{\begin{pmatrix}\frac{#1}{#2}\end{pmatrix}}

\newcommand{\cO}{\mathcal{O}}
\newcommand{\mcO}{\mathcal{O}}
\newcommand{\mfp}{\mathfrak{p}}

\newcommand{\bfT}{\mathbf{T}}

\newcommand{\x}{\mathbf{x}}
\newcommand{\y}{\mathbf{y}}

\newcommand{\bU}{\mathbb{U}}
\newcommand{\bbU}{\mathbb{U}}
\newcommand{\bZ}{\mathbb{Z}}

\newcommand{\F}{\mathbb{F}}
\newcommand{\FF}{\mathbb{F}}

\newcommand{\cA}{\mathcal{A}}
\newcommand{\R}{\mathbb{R}}
\newcommand{\bbR}{\mathbb{R}}
\newcommand{\Q}{\mathbb{Q}}
\newcommand{\bbQ}{\mathbb{Q}}
\newcommand{\QQ}{\mathbb{Q}}
\newcommand{\C}{\mathbb{C}}
\newcommand{\CC}{\mathbb{C}}
\newcommand{\bbC}{\mathbb{C}}

\newcommand{\bbS}{\mathbb{S}}
\newcommand{\bS}{\mathbb{S}}
\newcommand{\bG}{\mathbb{G}}
\newcommand{\bbG}{\mathbb{G}}
\newcommand{\bfG}{\mathbf{G}}

\newcommand{\J}{\mathbf{J}}
\newcommand{\bfS}{\mathbf{S}}

\newcommand{\bfL}{\mathbf{L}}
\newcommand{\bL}{\mathbb{L}}
\newcommand{\bfZ}{\mathbf{Z}}

\newcommand{\cB}{\mathcal{B}}
\newcommand{\mcB}{\mathcal{B}}

\newcommand{\mfg}{\mathfrak{g}}

\newcommand{\bfU}{\mathbf{U}}

\newcommand{\jhat}{\widehat{\jmath\,}}
\newcommand{\der}{\mathrm{der}}

\newcommand{\cc}{{}^{\circ}\!}

\DeclareMathOperator{\Nr}{Nr}
\DeclareMathOperator{\Tr}{Tr}

\DeclareMathOperator{\depth}{depth}
\DeclareMathOperator{\sgn}{sgn}

\DeclareMathOperator{\ord}{ord}

\DeclareMathOperator{\rank}{rank}

\DeclareMathOperator{\SL}{SL}

\DeclareMathOperator{\cInd}{c-Ind}
\DeclareMathOperator{\Ind}{Ind}

\DeclareMathOperator{\Hom}{Hom}

\DeclareMathOperator{\Res}{Res}
\DeclareMathOperator{\bfGal}{Gal}

\DeclareMathOperator{\Stab}{Stab}

\DeclareMathOperator{\GL}{GL}

\DeclareMathOperator{\geom}{geom}
\DeclareMathOperator{\Gal}{Gal}

\DeclareMathOperator{\meas}{meas}

\newcommand\from{\colon}

\theoremstyle{plain}
\newtheorem{thm}{Theorem}[section]
\newtheorem{theorem}[thm]{Theorem}
\newtheorem*{thm*}{Theorem}
\newtheorem{prop}[thm]{Proposition}
\newtheorem{proposition}[thm]{Proposition}
\newtheorem{lem}[thm]{Lemma}
\newtheorem{lemma}[thm]{Lemma}
\newtheorem{cor}[thm]{Corollary}
\newtheorem{conj}[thm]{Conjecture}

\theoremstyle{definition}
\newtheorem{defn}[thm]{Definition}
\newtheorem{definition}[thm]{Definition}

\newtheorem{claim}{Claim}

\theoremstyle{remark}
\newtheorem{rem}[thm]{Remark}
\newtheorem*{claim*}{Claim}
\newtheorem{remark}[thm]{Remark}

\theoremstyle{theorem}
\newtheorem{displaytheorem}{Theorem}

\SelectTips{cm}{11}

\title{Geometric $L$-packets of Howe-unramified toral supercuspidal representations}

\author{Charlotte Chan}
\address{Department of Mathematics, Massachusetts Institute of Technology, 77 Massachusetts Ave, Cambridge, MA 02139, USA.}
\email{charchan@mit.edu}

\author{Masao Oi}
\address{Department of Mathematics (Hakubi center), Kyoto University, Kitashirakawa, Oiwake-cho, Sakyo-ku, Kyoto 606-8502, Japan.}
\email{masaooi@math.kyoto-u.ac.jp}

\begin{document}

\begin{abstract}
We show that $L$-packets of toral supercuspidal representations arising from unramified maximal tori of $p$-adic groups are realized by Deligne--Lusztig varieties for parahoric subgroups. We prove this by exhibiting a direct comparison between the cohomology of these varieties and algebraic constructions of supercuspidal representations. Our approach is to establish that toral irreducible representations are uniquely determined by the values of their characters on a domain of sufficiently regular elements. This is an analogue of Harish-Chandra's characterization of real discrete series representations by their character on regular elements of compact maximal tori, a characterization which Langlands relied on in his construction of $L$-packets of these representations. In parallel to the real case, we characterize the members of Kaletha's toral $L$-packets by their character on sufficiently regular elements of elliptic maximal tori.
\end{abstract}


\maketitle

\section{Introduction}\label{sec:intro}

This paper has several objectives, all of which are connected by the core motif that it is of significant interest to be able to recognize irreducible representations by the values of their characters on some domain.
For real groups, it is the remarkable work of Harish-Chandra \cite{MR219665} that real discrete series representations are determined by their character on the regular elements of compact maximal tori, a domain on which the character formula is very simple. This characterization was later used by Langlands \cite{MR1011897} to package these representations into $L$-packets. 
For $p$-adic groups, Kaletha \cite{MR4013740} recently proposed a construction of regular supercuspidal $L$-packets by reparametrizing (part of) Yu's construction of supercuspidal representations \cite{MR1824988} in terms of characters of certain elliptic maximal tori.
As Kaletha mentions, ideally one should be able to characterize the members of these $L$-packets by their character on some nice domain as in the real case; this problem was also mentioned several years earlier by Adler--Spice \cite{MR2543925}, motivated by Henniart's work for $\GL_n$ \cite{MR1263525}. However, even the correct statement of the analogue of Harish-Chandra's result for general $p$-adic groups was essentially completely unknown. One of the main results of this paper (Theorem \ref{thm:unram pi vreg}) is a resolution of this characterization problem for a class of regular supercuspidal representations corresponding to unramified elliptic maximal tori. This is a vast generalization of a Henniart's $\GL_n$ results. 

It is a folklore conjecture that every supercuspidal representation of a $p$-adic group is isomorphic to the compact induction of some finite-dimensional irreducible representation of a compact-modulo-center subgroup. 
In all known constructions of supercuspidal representations, modulo center, this compact subgroup can (essentially) be taken to be a so-called parahoric subgroup. 
Much of this paper is dedicated to establishing a characterization theorem at parahoric level; this characterization is significantly harder to establish than the above-mentioned result at the level of the $p$-adic group.
We fix some notation: To any elliptic unramified maximal torus $\bfS$ of a connected reductive group $\bfG$ defined over a non-archimedean local field $F$ with finite residue field $\FF_q$, we may associate a unique point $\x$ in the reduced building $\cB^{\red}(\bfG, F)$ together with a parahoric subgroup $G_{\x,0} \subset G \colonequals \bfG(F)$. 
We write $W_{G_{\x,0}}(\bfS)$ for the quotient of $N_{G_{\x,0}}(\bfS)$ (the normalizer group of $\bfS$ in $G_{\x,0}$) by $S:=\bfS(F)$.

\begin{displaytheorem}[Theorem \ref{thm:SG vreg}]\label{thm:display vreg}
Let $\theta \from S \to \bbC^\times$ be a toral character and assume that $q >\!\!> 0$. There exists a unique irreducible representation $\pi$ of $SG_{\x,0}$ such that its character $\Theta_\pi$ at any regular element $\gamma \in S$ is
\begin{equation*}
\Theta_\pi(\gamma) = c \cdot \sum_{w \in W_{G_{\x,0}}(\bfS)} \theta({}^w \gamma)
\end{equation*}
for some constant $c \in \bbC^\times$ which does not depend on $\gamma$, where ${}^{w}\gamma:=w\gamma w^{-1}$.
Moreover, $c \in \{\pm 1\}$.
\end{displaytheorem}

Completely separately and independently to the above algebraic developments, in recent years there has been a push towards constructing supercuspidal representations \textit{geometrically} using constructions analogous to Deligne--Lusztig varieties for finite groups of Lie type. 
Central to this picture is Lusztig's work \cite{MR2048585} on such varieties for reductive groups over finite rings in equal characteristic, Stasinski's subsequent work \cite{MR2558788} for mixed characteristic, and the first author's joint work with Ivanov \cite{MR4197070} generalizing these works to arbitrary parahoric subgroups $G_{\x,0}$ associated to unramified maximal tori. 
It is expected that Lusztig's conjecture on loop Deligne--Lusztig constructions for $p$-adic groups \cite{MR546595} is very closely related to the parahoric picture, as demonstrated in \cite{CI_ADLV} in the setting of $\GL_n$ and its inner forms.
Among other things, the present paper resolves a basic and major gap in this geometric program: we prove that the irreducible representations of $SG_{\x,0}$ arising from the cohomology of parahoric Deligne--Lusztig varieties indeed compactly induce to supercuspidal representations of $G$. 

Following \cite{MR4197070}, to every positive integer $r$, one can construct a smooth, separated, finite-type affine $\overline \FF_q$-scheme $X_r$ with a natural action by $G_{\x,0} \times S_0$ where $S_0 = S \cap G_{\x,0}$. 
This action can be extended by the center $Z_\bfG$ so that for any depth-$r$ character $\theta \from S \to \bbC^\times$, the corresponding isotropic subspace $H_c^*(X_r)[\theta] \colonequals H_c^*(X_r, \overline \QQ_\ell)[\theta]$ of the cohomology $H_c^*(X_r, \overline \QQ_\ell) \colonequals \sum_i (-1)^i H_c^i(X_r, \overline \QQ_\ell)$ is in fact a (virtual) representation of $Z_\bfG G_{\x,0} = SG_{\x,0}$. (We note that because $\bfS$ is elliptic and unramified, we have $S = Z_{\bfG} S_0$.) 

\begin{displaytheorem}[Theorems \ref{thm:comparison SG}, \ref{thm:geom L packets}, \ref{thm:stability}]\label{thm:display Xr}
Let $\theta \from S \to \bbC^\times$ be extra toral\footnote{In the body of this paper, \textit{extra toral} is called \textit{0-toral}. This condition arises naturally in \cite{MR2048585} and in subsequent work by others inspired by \cite{MR2048585}; in these past purely geometric investigations, \textit{extra toral} is called \textit{regular}.} of depth $r > 0$ and assume $q >\!\!> 0$.
\begin{enumerate}[label=(\roman*)]
\item
The compact induction $\cInd_{S G_{\x,0}}^G(H_c^*(X_r)[\theta])$ is an irreducible supercuspidal representation of $G$.
\item
The correspondence $(\bfS, \theta) \mapsto \cInd_{SG_{\x,0}}^G(H_c^*(X_r)[\theta])$ preserves stability and $X_r$ gives a geometric realization of extra toral supercuspidal $L$-packets \`a la DeBacker--Spice \cite{MR3849622}.
\end{enumerate}
\end{displaytheorem}

Allow us to immediately spoil the punch line relating this result and the discussion of characterizations of representations \`a la Harish-Chandra: Theorem \ref{thm:display Xr} is an application of (a more precise version of) Theorem \ref{thm:display vreg}.

We mention that when $r = 0$, the variety $X_r$ is a classical Deligne--Lusztig variety and the conclusions of Theorem \ref{thm:display Xr} are true for depth-$0$ characters $\theta$ in general position: (i) is due to Moy--Prasad  \cite{MR1371680}, and (ii) is due to DeBacker--Reeder \cite{MR2480618} and Kazhdan--Varshavsky \cite{MR2214251}. 

For $r > 0$, it was proved in \cite{MR4197070} that $H_c^*(X_r, \overline \QQ_\ell)[\theta]$ is an irreducible representation of $SG_{\x,0}$. The additional assertion in Theorem \ref{thm:display Xr}(i) that its compact induction to $G$ is irreducible (and therefore supercuspidal) has been studied by various people in special cases---for inner forms of $\GL_n$ \cite{CI_ADLV, CI_loopGLn} and for $\x$ hyperspecial, $r$ odd \cite{MR3703469}---by techniques totally different to ours (see Sections \ref{sec:GLn}, \ref{sec:small p} for further discussion). We remark that the assumption $q >\!\!> 0$ is innocuous at least for Coxeter tori---we will prove in subsequent work that it is no stronger than the assumptions needed for Yu's construction \cite{MR1824988} (see Remark \ref{rem:q >> 0}). In particular, our approach specialized to $\GL_n$ relaxes the $p>n$ assumption in \cite{CI_loopGLn} to $p > 2$ (see Section \ref{sec:GLn}). 

We in fact prove something stronger than the supercuspidality assertion of Theorem \ref{thm:display Xr}(i): we explicitly describe the supercuspidal $\cInd_{SG_{\x,0}}^G(H_c^*(X_r)[\theta])$ in terms of Yu's construction and Kaletha's reparametrization. This resolves a generalization of a question of Lusztig on comparing the representations in \cite{MR2048585} with non-cohomological constructions. Supercuspidal representations in the extra toral setting had already been constructed and parametrized by Adler \cite{MR1653184}; our choice to write our paper within Yu's and Kaletha's  framework is in anticipation of future work relaxing the genericity assumptions (toral, extra toral) on $\theta$. 

Now let us explain the content of Theorem \ref{thm:display Xr}(ii) in the context of past works. Following the construction of discrete series $L$-packets for real groups \cite{MR1011897} and of depth-0 $L$-packets of $p$-adic groups \cite{MR2480618, MR2214251}, one could extrapolate that for supercuspidal representations parametrized by characters $\theta$ of elliptic maximal tori $\bfS$, $L$-packets should be parametrized by stable conjugacy classes of $(\bfS,\theta)$. Using Adler's parametrization $(\bfS,\theta) \mapsto \pi_{(\bfS,\theta)}$ of extra toral supercuspidal representations, Reeder \cite{MR2427973} verified constancy of the formal degree on this packet of supercuspidals in the case that $\bfS$ is unramified. Later, DeBacker--Spice \cite{MR3849622}, still working in the setting that $\bfS$ is unramified, proved that this packet of supercuspidals fails (!) to be \textit{stable} (see Section \ref{sec:Lpacket} for more details); to make it stable, they prove that one must instead consider the twisted parametrization $(\bfS, \theta) \mapsto \pi_{(\bfS, \theta \cdot \varepsilon[\theta])}$ by a quadratic character $\varepsilon[\theta]$ which depends on $(\bfS, \theta)$. 
This stability result has been generalized to extra toral supercuspidals corresponding to tamely ramified $\bfS$ by Kaletha \cite{MR4013740}, whose theory of regular supercuspidal representations develops a parametrization of a much larger class of supercuspidals in terms of $(\bfS,\theta)$ and includes a candidate generalized twisting character $\varepsilon[\theta]$.
The contribution of Theorem \ref{thm:display Xr} to this picture is that $\cInd_{SG_{\x,0}}^G(H_c^*(X_r)[\theta]) \cong \pi_{(\bfS, \theta \cdot \varepsilon[\theta])}$, so that in terms of the cohomologically arising parametrization of these supercuspidals, \textit{no external twisting} is required to obtain a \textit{stable set} of supercuspidals from a stable conjugacy class of $(\bfS, \theta)$. 
We emphasize this point: the geometry seems to innately know about the automorphic side of the local Langlands correspondence. 

\subsection{Outline of the paper}

A subtle point throughout this paper is taking stock of what assumptions one needs on $p$. For the most part, we have chosen to work in the greatest generality possible for each ingredient going into this paper, especially so as to illustrate the reasons various small primes are excluded. We collect a summary of these assumptions in Section \ref{sec:notations}.

In Section \ref{sec:Yu}, we recall Yu's construction of supercuspidal representations and Kaletha's theory of regular supercuspidal representations. 
In particular, we recall how to associate to a tame elliptic regular pair $(\bfS, \phi)$ a representation $\cc \tau_d$ of $SG_{\x,0}$ whose compact induction is the irreducible supercuspidal representation $\pi_{(\bfS, \phi)}$. 
We warn the reader that in the literature, it is more popular to work with a representation of the full stabilizer $G_{\x}$ of the point $\x$; it takes some care to work on 
$SG_{\x,0} \subseteq G_{\x}$ 
instead, which we need to do for geometrically motivated reasons later. 
We additionally relax the ellipticity assumption on Kaletha's Howe factorization of tame elliptic pairs (Section \ref{subsec:Howe}) and use this to state a geometric conjecture (Conjecture \ref{conj:drinfeld}) later in the paper. 

Starting in Section \ref{sec:AS}, we assume that $\bfS$ is unramified. Sections \ref{sec:AS} and  \ref{sec:ell unram vreg} culminate in two characterization theorems for toral characters---one for  representations of $SG_{\x,0+}$ 
(Theorem \ref{thm:SG+ vreg})
and one for representations of $SG_{\x,0}$ 
(Theorem \ref{thm:SG vreg}). 
Of central importance in our analysis is a class of elements of $SG_{\x,0}$ called \textit{unramified very regular elements} in the sense of \cite{MR4197070}; this class includes the set of elements $S_{\vreg}$ in $S$ which are regular semisimple in $G$. The main result of Section \ref{sec:AS} is Proposition \ref{prop:AS-vreg}, a simple and explicit character formula for the $SG_{\x,0}$-representation $\cc \tau_d$ on the unramified very regular locus of $SG_{\x,0}$.

In Section \ref{sec:ell unram vreg}, we only work with tame elliptic regular pairs $(\bfS, \phi)$ where $\bfS$ is unramified and $\phi$ is toral. We prove in Section \ref{subsec:SG+} that using a clever trick (Lemma \ref{lem:hom G+ +}), our characterization theorem on $SG_{\x,0+}$ (Theorem \ref{thm:SG+ vreg}) essentially follows by Frobenius reciprocity. 
The same strategy works to establish our characterization theorem on $SG_{\x,0}$ (Theorem \ref{thm:SG vreg}, presented in the Introduction as Theorem \ref{thm:display vreg}), though with extra arguments at two points (Lemmas \ref{lem:hom G+} and \ref{lem:toral Henn}, the latter of which is a direct generalization of a theorem of Henniart \cite{MR1263525} in the setting of $\GL_n$). 
For the all arguments in this section, we need to assume that $S_{\vreg}$ generates $S$ as a group. We show in Section \ref{subsec:vreg} that this is the case if a certain inequality \eqref{ineq} related to a density of the unramified very regular elements holds, and that \eqref{ineq} holds if $q >\!\!>0$. Moreover, the bound on $q$ can be reduced to a calculation on reductive groups over finite fields because of a transferring trick (Lemma \ref{lem:transfer}).

Our focus shifts in Section \ref{sec:geom}, where we recall what is known about the $\overline \FF_q$-schemes $X_r$ and their cohomology. In this section, we make no assumptions on $p$ or on the ellipticity of $\bfS$. 
We recall the Drinfeld stratification \cite{CI_DrinfeldStrat}, which consists of subvarieties $X_r^{(\bfL)} \subset X_r$ indexed by certain twisted Levi subgroups $\bfL$ containing $\bfS$. 
We conjecture (Conjecture \ref{conj:drinfeld}) that the Drinfeld stratification is the geometric version of the ``stratification'' on the set of regular supercuspidal representations given by the $0$th piece $\bfG^0$ of the Howe factorization.

We prove our main results about the cohomology of $X_r$ and its relation to supercuspidal representations in Section \ref{sec:comparison}; note that supercuspidality corresponds to ellipticity of $\bfS$. In our framework, this comparison is simply an application of Section \ref{sec:ell unram vreg} to the cohomological representations discussed in Section \ref{sec:geom}. We compare $H_c^*(X_r)[\theta]$ to the $SG_{\x,0}$-representation $\cc \tau_d$ in Theorem \ref{thm:comparison SG} and prove Conjecture \ref{conj:drinfeld} in the setting $\bfL = \bfS$ in Theorem \ref{thm:drinfeld}. We note that since we obtain these results as corollaries of the characterization theorems, there is an intriguing mystery surrounding the geometric representations $H_c^*(X_r)[\theta]$ for the small $p$ excluded by Kaletha's theory of regular supercuspidal representations \cite{MR4013740,Kaletha}. We make some comments about this in Section \ref{sec:small p}. In Section \ref{sec:GLn}, we focus on the setting $\bfG = \GL_n$, explicitly calculate the twisting character $\varepsilon^{\ram}[\theta]$, and show that a stronger version of the geometrically-proved supercuspidality results of \cite{CI_loopGLn} follows as a special case (Corollary \ref{cor:GLn sc}) of our comparison theorem (Theorem \ref{thm:comparison SG}). For convenience, we include a diagram depicting the main structure of the results needed to prove our comparison theorems (Theorems \ref{thm:comparison SG}, \ref{thm:comparison SG+}):
\[
\xymatrix{
SG^{d}_{\x,0} & \cc\tau_{d} &&& |H_{c}^{\ast}(X_{r},\overline{\Q}_{\ell})[\theta]| \ar@{=}_-{\text{Theorem \ref{thm:comparison SG}}}^-{\text{(Theorem \ref{thm:SG vreg})}}[lll]\\
SG^{d}_{\x,0+} \ar@{}[u]|{\bigcup} & \Ind_{\cc K^{d}}^{SG^{d}_{\x,0+}} \cc\rho'_{d}\otimes\phi_{d} \ar@{-->}[u] &&& |H_{c}^{\ast}(X_{r}\cap SG^{d}_{\x,0+},\overline{\Q}_{\ell})[\theta]|\ar@{-->}^-{\text{Theorem \ref{thm:drinfeld}}}_{\text{(Conjecture \ref{conj:drinfeld}})}[u] \ar@{=}_-{\text{Theorem \ref{thm:comparison SG+}}}^-{\text{(Theorem \ref{thm:SG+ vreg})}}[lll]\\
\cc K^{d} \ar@{}[u]|{\bigcup} & \cc\rho'_{d}\otimes\phi_{d} \ar@{-->}[u] &&&
}
\]
Here, the dashed vertical arrows indicate induction; these representations, all of which appear in Yu's construction, are recalled in Sections \ref{subsec:Yu}, \ref{subsec:stab-parah}.

In Section \ref{sec:Lpacket} we see the implications of our comparison theorem in the context of the local Langlands correspondence. We discuss Kaletha's construction of $L$-packets for extra toral supercuspidal representations and use our comparison to deduce (Theorem \ref{thm:geom L packets}) that $L$-packets of extra toral supercuspidal representations associated to unramified $\bfS$ are realized by the natural correspondence arising via the cohomology of $X_r$. This yields Theorem \ref{thm:stability}, which is presented in the Introduction as Theorem \ref{thm:display Xr}(ii).

Finally, in Section \ref{sec:vreg characterization}, we prove that toral supercuspidal representations associated to unramified $\bfS$ are determined by their character on $S_{\vreg}$ (Theorem \ref{thm:unram pi vreg}). The structure of this argument has a similar flavor to the parahoric-level characterization theorems of Section \ref{sec:ell unram vreg}, but neither section implies the other logically. As mentioned in the Introduction, Theorem \ref{thm:unram pi vreg} is a $p$-adic analogue of Harish-Chandra's characterization of discrete series representations of real groups, and is the first of its kind at this level of generality. In particular, it allows one to characterize Kaletha's construction of these $L$-packets purely in terms of their character values on regular elements of $S$.

\medbreak
\noindent{\bfseries Acknowledgment.}\quad We would like to thank the University of Tokyo - Princeton Strategic Partnership for the 2019 workshop in Arithmetic Geometry, where the authors began this project. We would further like to thank the Princeton Mathematics Department both for supporting the second author's visit via Christopher Skinner's visitor funds and for excellent working conditions. 
Finally, we are grateful to Tasho Kaletha for helpful discussions regarding assumptions on small residual characteristics.
The first author was partially supported by NSF grant DMS-1802905. The second author was partially supported by JSPS KAKENHI Grant Number 19J00846 (PD).

\newpage

\setcounter{tocdepth}{2}
\tableofcontents

\newpage

\section{Notations and assumptions}\label{sec:notations}

Let $F$ be a non-archimedean local field with finite reside field $\cO_F/\mfp_F \cong \FF_q$ of prime characteristic $p$, where we write $\mcO_{F}$ and $\mfp_{F}$ for the ring of its integers and the maximal ideal, respectively.
We let $F^{\ur}$ denote the completion of the maximal unramified extension of $F$.
We write $\Gamma_{F}$ for the absolute Galois group of $F$.

For an algebraic variety $\J$ over $F$, we denote the set of its $F$-valued points by $J$.
When $\J$ is an algebraic group, we write $\bfZ_{\J}$ for its center.

Let us assume that $\J$ is a connected reductive group over $F$.
We follow the notation around Bruhat--Tits theory used by \cite{MR2431235, MR2543925, MR3849622}.
(See, for example, \cite[Section 3.1]{MR2431235} for  details.)
Especially, $\mcB(\J,F)$ (resp.\ $\mcB^{\red}(\J,F)$) denotes the enlarged (resp.\ reduced) Bruhat--Tits building of $\J$ over $F$.
For a point $\x\in\mcB(\J,F)=\mcB^{\red}(\J,F)\times X_{\ast}(\bfZ_{\J})_{\R}$, we write $\bar{\x}$ for the image of $\x$ in $\mcB^{\red}(\J,F)$, and $J_{\bar{\x}}$ for the stabilizer of $\bar{\x}$ in $J$.
We define $\widetilde{\R}$ to be the set $\R\sqcup\{r+\mid r\in\R\}\sqcup\{\infty\}$ with a natural order.
Then for any $r\in\widetilde{\R}$ we can consider the $r$-th Moy--Prasad filtration $J_{\x,r}$ of $J$ with respect to the point $\x$.
For any $r,s\in\widetilde{\R}_{\geq0}$ satisfying $r<s$, we write $J_{\x,r:s}$ for the quotient $J_{\x,r}/J_{\x,s}$.
Recall that $J_{\x,0:0+}$ can be regarded as the set $\mathbb{J}(\F_{q})$ of $\F_{q}$-valued points of a connected reductive group $\mathbb{J}$ defined over $\F_{q}$ (such a group $\mathbb{J}$ can be realized as the reductive quotient of the special fiber of the parahoric subgroup scheme attached to $\x$, see \cite[Section 3.2]{MR1371680}).

\subsection{Assumptions on $F$}\label{subsec:assumptions}

Let $\bfG$ be a tamely ramified connected reductive group over $F$. Unless otherwise stated, we will assume that $p$ is odd, $p$ is not bad for $\bfG$ (in the sense of \cite[Definition A.5]{MR2431235}), and that $p \nmid |\pi_1(\bfG_{\der})|$ and $p \nmid |\pi_1(\widehat \bfG_{\der})|$.
There are a few sections in the paper where we either relax or strengthen our assumptions on $F$; we specify these subsections here:
\begin{enumerate}[label=\textbullet]
\item Sections \ref{subsec:Yu}, \ref{subsec:stab-parah} hold with the relaxed assumption that $p$ is odd, but this is inconsequential for us.
\item Section \ref{subsec:Howe} holds without any assumption on $p$ except for Lemma \ref{lem:theta-stab}.
\item In Sections \ref{subsec:SG+}, \ref{subsec:SG}, we additionally assume \eqref{ineq}, an assumption on the size of the residue field of $F$. This additional assumption is also needed in Sections \ref{sec:comparison}, \ref{sec:Lpacket}, \ref{sec:vreg characterization}, as these later sections rely on various arguments presented in Sections \ref{subsec:SG+}, \ref{subsec:SG}.
\item Section \ref{sec:geom} recalls geometric constructions of representations of parahoric subgroups and holds with no assumptions on $p$. The discrepancy between this and the additional assumptions on $p$ needed in other sections is an interesting point; we make some remarks about this in Section \ref{sec:small p}.
\end{enumerate}

\section{Yu's supercuspidal representations}\label{sec:Yu}
In this section, we first briefly review Yu's construction of supercuspidal representations (see \cite{MR1824988}; for an exposition, \cite[Section 2]{MR2543925}). Then, we summarize Kaletha's result (\cite[Section 3.4]{MR4013740}) on a reparametrization of Yu's supercuspidal representations, which is based on the so-called Howe factorization for (certain) characters of elliptic maximal tori. We recall a part of Kaletha's Howe factorization process in Section \ref{subsec:Howe}; we will use this to state a geometric conjecture later (Conjecture \ref{conj:drinfeld}).
We finish with a discussion of passing from the full stabilizer $G_{\bar \x}$ to the smaller group $SG_{\x,0}$ in Section \ref{subsec:stab-parah}, especially establishing some notation that will be used throughout the paper.

We assume that $\bfG$ is a tamely ramified connected reductive group over $F$.

\subsection{Yu's supercuspidal representations}\label{subsec:Yu}

The constructions of this subsection hold with the relaxed assumption that $p$ is odd. 
In \cite{MR1824988}, Yu introduced the notion of a \textit{cuspidal $\bfG$-datum} and to each such datum attached an irreducible supercuspidal representation of $G$. 
Recall that a cuspidal $\bfG$-datum is a quintuple 
\[
\Sigma
=
(\vec{\bfG},\vec{\phi},\vec{r},\x,\rho'_{0})
\]
consisting of the following objects:
\begin{itemize}
\item
$\vec{\bfG}$ is a sequence $\bfG^{0}\subsetneq\bfG^{1}\subsetneq\cdots\subsetneq\bfG^{d}=\bfG$ of tame twisted Levi subgroups (i.e., each $\bfG^{i}$ is a subgroup of $\bfG$ which is defined over $F$ and becomes a Levi subgroup of $\bfG$ over a tamely ramified extension of $F$) such that $\bfZ_{\bfG^{0}}/\bfZ_{\bfG}$ is anisotropic,
\item
$\x$ is a point of $\mcB(\bfG^{0},F)$ whose image $\bar{\x}$ in $\mcB^{\red}(\bfG^{0},F)$ is a vertex,
\item
$\vec{r}$ is a sequence $0\leq r_{0}<\cdots<r_{d-1}\leq r_{d}$ of real numbers such that $0<r_{0}$ when $d>0$,
\item
$\vec{\phi}$ is a sequence $(\phi_{0},\ldots,\phi_{d})$ of characters $\phi_{i}$ of $G^{i}$ satisfying
\begin{itemize}
\item
for $0\leq i<d$, $\phi_{i}$ is $\bfG^{i+1}$-generic of depth $r_{i}$ at $\x$, and
\item
for $i=d$, 
\[
\begin{cases}
\depth_{\x}(\phi_{d})=r_{d}& \text{if $r_{d-1}<r_{d}$,}\\
\phi_{d}=\mathbbm{1} & \text{if $r_{d-1}=r_{d}$,}
\end{cases}
\]
\end{itemize}
\item
$\rho_{0}'$ is an irreducible representation of $G^{0}_{\bar{\x}}$ whose restriction to $G^{0}_{\x,0}$ contains the inflation of a cuspidal representation of the quotient $G^{0}_{\x,0:0+}$.
\end{itemize}

Yu's supercuspidal representation $\pi_{\Sigma}$ associated to $\Sigma$ is constructed as follows.
We first put
\[
(s_{0},\ldots,s_{d})
\colonequals 
\Bigl(\frac{r_{0}}{2},\ldots,\frac{r_{d}}{2}\Bigr)
\]
and define the subgroups $K^{i}$, $J^{i}$, and $J^{i}_{+}$ of $G$ for $1\leq i \leq d$ by
\[
K^{i}
\colonequals 
G^{0}_{\bar{\x}}(G^{0},\ldots,G^{i})_{\x,(0+,s_{0},\ldots,s_{i-1})},
\]
\[
J^{i}
\colonequals 
(G^{i-1},G^{i})_{\x,(r_{i-1},s_{i-1})},
\]
\[
J^{i}_{+}
\colonequals 
(G^{i-1},G^{i})_{\x,(r_{i-1},s_{i-1}+)},
\]
where the right-hand sides denote the subgroups associated to pairs consisting of a tame twisted Levi sequence and an admissible sequence (see \cite[Sections 1 and 2]{MR1824988}).
Note that we have $K^{i+1}=K^{i}J^{i+1}$.
For $i=0$, we put
\[
K^{0}
\colonequals 
G^{0}_{\bar{\x}}.
\]
Then we construct a representation $\rho'_{i+1}$ of $K^{i+1}$ from $\rho'_{i}$ of $K^{i}$ inductively in the following manner.
By investigating the quotient $J^{i}/J^{i}_{+}$ (which has a symplectic structure derived from the character $\phi_{i-1}$), we obtain a finite Heisenberg group as a quotient of the group $J^{i}$.
Then, as a consequence of the Stone--von Neumann theorem, we get a Heisenberg--Weil representation $\tilde{\phi_{i}}$ of the semi-direct product $G^{i}_{\bar{\x}}\ltimes J^{i+1}$.
The tensor representation 
\[
(\tilde{\phi_{i}}|_{K^{i}\ltimes J^{i+1}})
\otimes 
\bigl((\rho'_{i}\otimes\phi_{i}|_{K^{i}})\ltimes\mathbbm{1}\bigr)
\]
of $K^{i}\ltimes J^{i+1}$ descends to $K^{i}J^{i+1}=K^{i+1}$ (factors through the canonical map $K^{i}\ltimes J^{i+1}\twoheadrightarrow K^{i}J^{i+1}$), and we define the representation $\rho'_{i+1}$ of $K^{i+1}$ to be the descended one.
Then Yu's supercuspidal representation $\pi_{\Sigma}$ is defined by
\[
\pi_{\Sigma}\colonequals \cInd_{K^{d}}^{G}\rho'_{d}\otimes\phi_{d}.
\]
In this paper, let us call irreducible supercuspidal representations of $G$ obtained from cuspidal $\bfG$-data in this way \textit{Yu's supercuspidal representations}.

We also recall the definitions of a few more groups and representations which will be needed later (for describing the Adler--Spice character formula in Section \ref{sec:AS}):
\[
K_{\sigma_{i}}
\colonequals 
G^{i-1}_{\bar{\x}}G^{i}_{\x,0+}
\quad (K_{\sigma_{0}}\colonequals G^{0}_{\bar{\x}}),
\]
\[
\tilde{\rho}'_{i}
\colonequals 
\Ind_{K^{i}}^{G^{i-1}_{\bar{\x}}G^{i}_{\x,s_{i-1}}}\rho'_{i},
\]
\[
\sigma_{i}
\colonequals 
\Ind_{K^{i}}^{K_{\sigma_{i}}}\rho'_{i}\,\,
(\cong \Ind_{G^{i-1}_{\bar{\x}}G^{i}_{\x,s_{i-1}}}^{K_{\sigma_{i}}} \tilde{\rho}'_{i}),
\]
\[
\tau_{i}
\colonequals 
\Ind_{K^{i}}^{G^{i}_{\bar{\x}}} \rho'_{i}\otimes\phi_{i}\,\,
(\cong \Ind_{K_{\sigma_{i}}}^{G^{i}_{\bar{\x}}} \sigma_{i}\otimes\phi_{i}).
\]

We finally recall the notion of a generic reduced cuspidal $\bfG$-datum due to Hakim--Murnaghan (\cite{MR2431732}).
By the theory of Moy--Prasad, the induced representation
\[
\pi_{-1}
\colonequals 
\cInd_{K^{0}}^{G^{0}}\rho'_{0}
\]
is an irreducible depth-zero supercuspidal representation of $G^{0}$ (\cite[Proposition 6.6]{MR1371680}).
Conversely, any irreducible depth-zero supercuspidal representation $\pi_{-1}$ of $G^{0}$ is obtained by the compact induction of a representation $\rho'_{0}$ satisfying the condition mentioned above in a unique (up to conjugation) way (\cite[Proposition 6.8]{MR1371680}).
From this observation we conclude that the triple $(\vec{\bfG},\pi_{-1},\vec{\phi})$ is essentially equivalent to the original quintuple $(\vec{\bfG},\vec{\phi},\vec{r},\x,\rho'_{0})$.
In \cite{MR2431732}, Hakim--Murnaghan called such triples \textit{generic reduced cuspidal $\bfG$-data} and defined an equivalence relation called \textit{$\bfG$-equivalence} on them.
Here we do not recall the definition of the $\bfG$-equivalence (see \cite[Definition 6.3]{MR2431732}). 
The important nature of this equivalence relation is that it describes the ``fibers'' of Yu's construction: for two given generic reduced cuspidal $\bfG$-data $\Sigma$ and $\Sigma'$, the associated supercuspidal representations $\pi_{\Sigma}$ and $\pi_{\Sigma'}$ are isomorphic if and only if two data $\Sigma$ and $\Sigma'$ are $\bfG$-equivalent.
In other words, Yu's construction gives the following bijective map:
\[
\xymatrix@R=10pt{
&&&\{\text{irred.\ s.c.\ rep'ns of $G$}\}/{\sim}\\
\{\text{gen.\ red.\ cusp.\ $\bfG$-data}\}/\text{$\bfG$-eq.} \ar@{->}[rrr]^-{1:1}_-{\text{Yu's construction}} &&& \{\text{Yu's s.c.\ rep'ns of $G$}\}/{\sim}\ar@{}[u]|{\bigcup}
}
\]

\subsection{Kaletha's reparametrization of Yu's supercuspidal representations}\label{subsec:Kaletha-TER}
We first recall Kaletha's classification of regular depth-zero supercuspidal representations.
Let $\bfG^{0}$ be a tamely ramified connected reductive group over $F$.
Let us suppose that we have an irreducible depth-zero supercuspidal representation $\pi_{-1}$ of $G^{0}$ and that the restriction $\pi_{-1}|_{G^{0}_{\x,0}}$ (for some $\x\in\mcB(\bfG^{0},F)$ whose image $\bar{\x}$ in $\mcB^{\red}(\bfG^{0},F)$ is a vertex) contains the inflation of an irreducible cuspidal representation $\kappa$ of $G^{0}_{\x,0:0+}$.
We put $\bbG^{0}_{\x}$ to be a connected reductive group over $\F_{q}$ obtained by taking the reductive quotient of the special fiber of the parahoric subgroup scheme attached to $\x$.
Then we have a natural identification $\bbG^{0}_{\x}(\F_{q})\cong G^{0}_{\x,0:0+}$.
By Deligne--Lusztig theory, the cuspidality of $\kappa$ implies that there exists a pair $(\bbS,\bar{\phi})$ of \begin{itemize}
\item
an elliptic maximal torus $\bbS$ of $\bbG^{0}_{\x}$ defined over $\F_{q}$, and
\item
a character $\bar{\phi}$ of $\bbS(\F_{q})$
\end{itemize}
such that the associated Deligne--Lusztig representation $\pm R_{\bbS}^{\bbG^{0}_{\x}}(\bar{\phi})$ contains $\kappa$.
By \cite[Lemma 3.4.4]{MR4013740}, there exists a maximally unramified (in the sense of \cite[Definition 3.4.2]{MR4013740}) elliptic maximal torus  $\bfS$ of $\bfG^{0}$ defined over $F$ whose connected N\'eron model has $\bbS$ as its special fiber.
We let $N_{G^{0}}(\bfS)$ be the normalizer group of $\bfS$ in $G^{0}$ and put
\[
W_{G^{0}}(\bfS)
:=
N_{G^{0}}(\bfS)/S.
\]
Following Kaletha \cite[Definitions 3.4.16 and 3.4.19]{MR4013740}, we say
\begin{itemize}
\item
the character $\bar{\phi}$ is \textit{regular} if the stabilizer of $\bar{\phi}$ in $W_{G^{0}}(\bfS)$ is trivial, and 
\item
the irreducible depth-zero supercuspidal representation $\pi_{-1}$ is \textit{regular} if $\bar{\phi}$ associated to $\pi_{-1}$ in the above manner is regular
\end{itemize}
(note that the group $W_{G^{0}}(\bfS)$ acts on $\bar{\phi}$ since we have $S_{0:0+}\cong\bbS(\F_{q})$).
In summary, we may associate a pair $(\bfS,\bar{\phi})$ to each irreducible depth zero supercuspidal representation $\pi_{-1}$ of $G^{0}$ and define the notion of regularity for $\pi$ by looking at the pair $(\bfS,\bar{\phi})$.

\begin{rem}\label{rem:DL}
If $\bar{\phi}$ is regular, then it is in general position in the sense of Deligne--Lusztig \cite[Fact 3.4.18]{MR4013740}.
If $\bar{\phi}$ is in general position, $(-1)^{r(\bbS)-r(\bbG^{0}_{\x})}R_{\bbS}^{\bbG^{0}_{\x}}(\bar{\phi})$ is an irreducible representation, where $r(\bbS)$ and $r(\bbG^{0}_{\x})$ are split ranks of $\bbS$ and $\bbG^{0}_{\x}$, respectively.
Thus $\kappa$ is necessarily equal to $(-1)^{r(\bbS)-r(\bbG^{0}_{\x})}R_{\bbS}^{\bbG^{0}_{\x}}(\bar{\phi})$ itself.
Furthermore, the orthogonality relation of Deligne--Lusztig (\cite[Theorem 6.8]{MR393266}) assures that such a pair $(\bbS,\bar{\phi})$ is unique up to $\bbG^{0}_{\x}$-conjugacy.
Hence $(\bfS,\bar{\phi})$ is unique up to $G^{0}_{\x,0}$-conjugacy.
\end{rem}

We next consider the ``converse'' of the above procedure.
Let us suppose that we have a maximally unramified elliptic maximal torus $\bfS$ of $\bfG^{0}$ and a regular depth-zero character $\phi_{-1}$ of $S$, i.e., the character $\bar{\phi}$ of $S_{0:0+}=\bbS(\F_{q})$ induced from $\phi_{-1}$ is regular in the above sense.
Since the torus $\bfS$ is elliptic, its reduced apartment $\mathcal{A}^{\red}(\bfS,F)$ of $\mcB^{\red}(\bfG^{0},F)$ consists of only one point, i.e., the unique Frobenius-fixed point in $\mathcal{A}^{\red}(\bfS,F^{\ur})$ (see the paragraph before \cite[Lemma 3.4.3]{MR4013740}).
We take a point $\x\in\mcB(\bfG^{0},F)$ whose image $\bar{\x}$ in $\mcB^{\red}(\bfG^{0},F)$ equals this (single) point of $\mathcal{A}^{\red}(\bfS,F)$.
Note that then $S$ normalizes $G_{\x,0}^{0}$ and we have $\bbS\subset\bbG^{0}_{\x}$.
As explained in Remark \ref{rem:DL}, from the pair $(\bbS,\bar{\phi})$, we get an irreducible cuspidal representation 
\[
\kappa_{(\bfS,\bar{\phi})}
\colonequals 
(-1)^{r(\bbS)-r(\bbG^{0}_{\x})}R_{\bbS}^{\bbG^{0}_{\x}}(\bar{\phi})
\]
of $\bbG^{0}_{\x}(\F_{q})$.
In \cite[Section 3.4.4]{MR4013740}, Kaletha constructs an extension of (the inflation of) the representation $\kappa_{(\bfS,\bar{\phi})}$ to $SG^{0}_{\x,0}$ in a geometric way.
Let $\kappa_{(\bfS,\phi_{-1})}$ denote the extended representation of $SG^{0}_{\x,0}$.
Now Kaletha's classification theorem of regular depth zero supercuspidal representations is summarized as follows:
\begin{prop}[{\cite[Lemma 3.4.20 and Proposition 3.4.27]{MR4013740}}]\label{prop:depth-zero}
The representation 
\[
\pi_{(\bfS,\phi_{-1})}^{\bfG^{0}}
\colonequals 
\cInd_{SG^{0}_{\x,0}}^{G^{0}}\kappa_{(\bfS,\phi_{-1})}
\]
is an irreducible depth-zero regular supercuspidal representation of $G^{0}$.
Conversely, every irreducible depth-zero regular supercuspidal representation of $G^{0}$ is obtained in this way.
Furthermore, two such representations $\pi_{(\bfS,\phi_{-1})}^{\bfG^{0}}$ and $\pi_{(\bfS',\phi'_{-1})}^{\bfG^{0}}$ are isomorphic if and only if the pairs $(\bfS,\phi_{-1})$ and $(\bfS',\phi'_{-1})$ are $G^{0}$-conjugate.
\end{prop}

We now return to Yu's supercuspidal representations.
Let $\Sigma=(\vec{\bfG},\pi_{-1},\vec{\phi})$ be a generic reduced cuspidal $\bfG$-datum and $\pi_{\Sigma}$ its associated supercuspidal representation of $G$.
We call $\Sigma$ \textit{regular} if $\pi_{-1}$ is regular.
We call $\pi_{\Sigma}$ a \textit{regular supercuspidal representation} if $\Sigma$ is regular.
Let us suppose that $\Sigma$ is regular.
Then, thanks to Proposition \ref{prop:depth-zero}, we have a pair $(\bfS,\phi_{-1})$ of maximally unramified elliptic maximal torus $\bfS$ of $\bfG^{0}$ and its regular depth-zero character $\phi_{-1}$ satisfying $\pi_{-1}\cong\pi_{(\bfS,\phi_{-1})}^{\bfG^{0}}$.
We put $\bfG^{-1}\colonequals \bfS$ and define a character $\phi$ of $S$ by
\[
\phi
\colonequals 
\prod_{i=-1}^{d} \phi_{i}|_{S}.
\]

Kaletha's reparametrizing result is as follows:
\begin{prop}[{\cite[Proposition 3.7.8]{MR4013740}}]\label{prop:TER-pair}
The map 
\[
(\vec{\bfG}, \pi_{-1},\vec{\phi})
\mapsto
(\bfS,\phi)
\]
defined in the above manner induces a bijection from the set of $\bfG$-equivalence classes of regular generic reduced cuspidal $\bfG$-data to the set of $G$-conjugacy classes of tame elliptic regular pairs in $\bfG$.
\end{prop}

\begin{rem}
Note that we need our baseline assumptions on $p$ ($p$ odd, $p$ not bad for $\bfG$, $p \nmid |\pi_1(\widehat \bfG_{\der})$, and $p \nmid |\pi_1(\bfG_{\der})|$) for this proposition, especially in establishing the surjectivity part of the map (this is called the ``Howe factorization'' process, which will be recalled more precisely in the next subsection).
One crucial step in proving \cite[Proposition 3.7.8]{MR4013740} is to establish \cite[Lemma 3.6.8]{MR4013740}, whose proof utilizes a technical result of Yu concerning the genericity of characters \cite[Lemma 8.1]{MR1824988}. 
The assumption required by \cite[Lemma 8.1]{MR1824988} is that $p$ is not a torsion prime for the root datum of the dual group $\bfG$, which is equivalent to that $p$ is not bad for $\bfG$ and does not divide the order of $|\pi_{1}(\widehat{\bfG}_{\der})|$.
Therefore we need to assume $p\nmid|\pi_{1}(\widehat{\bfG}_{\der})|$ in addition to the non-badness for the root datum of $\bfG$ imposed in the beginning of \cite[Section 3.6]{MR4013740}.
This subtlety is carefully explained in \cite{Kaletha}.
See \cite[Section 4]{Kaletha} for details.
\end{rem}

Recall that a pair $(\bfS,\phi)$ of a maximal torus $\bfS$ of $\bfG$ defined over $F$ and a character $\phi\colon S\rightarrow\C^{\times}$ is called a \textit{tame elliptic regular pair} if it satisfies the following conditions (\cite[Definition 3.7.5]{MR4013740}):
\begin{itemize}
\item
$\bfS$ is a tamely ramified elliptic maximal torus of $\bfG$,
\item
the action of the inertia subgroup $I_{F}$ of $\bfGal(\overline{F}/F)$ on the root subsystem
\[
R_{0+}
\colonequals 
\{\alpha\in R(\bfS,\bfG) \mid \phi|_{\Nr_{E/F}(\alpha^{\vee}(E_{0+}^{\times}))}\equiv\mathbbm{1}\}
\]
of the root system $R(\bfS,\bfG)$ preserves a set of positive roots, where
\begin{itemize}
\item
$E$ is the tamely ramified extension of $F$ splitting $\bfS$, and
\item
$\Nr_{E/F}$ is the norm map $\bfS(E)\rightarrow\bfS(F)$, 
\end{itemize}
\item
the restriction $\phi|_{S_{0}}$ has trivial stabilizer for the action of the group $W_{G^{0}}(\bfS)$, where $\bfG^0 \subset \bfG$ is the reductive group with maximal torus $\bfS$ and root system $R_{0+}$.
\end{itemize}

According to Proposition \ref{prop:TER-pair}, we conclude that (isomorphism classes of) regular supercuspidal representations bijectively correspond to ($G$-conjugacy classes of) tame elliptic regular pairs.
Let $\pi_{(\bfS,\phi)}$ denote the representation corresponding to a tame elliptic regular pair $(\bfS,\phi)$.
\[
\xymatrix{
\{\text{gen.\ red.\ cusp.\ $\bfG$-data}\}/\text{$\bfG$-eq.} \ar@{->}[r]^-{1:1} & \{\text{Yu's s.c.\ rep'ns of $G$}\}/{\sim}\\
\{\text{regular gen.\ red.\ cusp.\ $\bfG$-data}\}/\text{$\bfG$-eq.}\ar@{}[u]|{\bigcup} \ar@{->}[r]^-{1:1} & \{\text{regular s.c.\ rep'ns of $G$}\}/{\sim}\ar@{}[u]|{\bigcup}\\
\{\text{tame elliptic regular pairs}\}/\text{$G$-conj.}\ar@{<->}[u]^-{1:1}_-{\text{(Prop.\ \ref{prop:TER-pair})}} \ar@{->}[ur]_-{\quad\quad\quad(\bfS,\phi)\mapsto\pi_{(\bfS,\phi)}}&
}
\]

\begin{rem}
If we take a cuspidal $\bfG$-datum $(\vec{\bfG},\vec{\phi},\vec{r},\x,\rho'_{0})$ corresponding to $(\vec{\bfG}, \pi_{-1},\vec{\phi})$, then we have $\pi_{-1}\cong\cInd_{K^{0}}^{G^{0}}\rho'_{0}$ (recall that $K^{0}=G^{0}_{\bar{\x}}$).
On the other hand, since we have $\pi_{-1}\cong\pi_{(\bfS,\phi_{-1})}^{\bfG^{0}}=\cInd_{SG^{0}_{\x,0}}^{G^{0}}\kappa_{(\bfS,\phi_{-1})}$, we may suppose that
\[
\rho'_{0}
\cong
\Ind_{SG^{0}_{\x,0}}^{G^{0}_{\bar{\x}}}\kappa_{(\bfS,\phi_{-1})}.
\]
\end{rem}

Finally we introduce the notion of Howe-unramifiedness as follows:
\begin{defn}\label{defn:Howe-unram}
If a regular supercuspidal representation $\pi$ is associated to a tame elliptic regular pair $(\bfS,\phi)$ with unramified $\bfS$, we say that $\pi$ is \textit{Howe-unramified}.
\end{defn}

\subsection{Tame twisted Levi sequence associated to a character}\label{subsec:Howe}
As the surjectivity part of Proposition \ref{prop:TER-pair} shows, we may associate to any tame elliptic regular pair $(\bfS,\phi)$ a sequence of tame twisted subgroups $\vec{\bfG}=(\bfG_{-1},\ldots,\bfG_{d})$ and a sequence of characters $(\phi_{-1},\ldots,\phi_{d})$.
Indeed, in the proof of Proposition \ref{prop:TER-pair} (\cite[Proposition 3.7.8]{MR4013740}), Kaletha gives a construction of such sequences explicitly.
As explained in Section \cite[Section 3.6]{MR4013740}, this can be understood as a generalization of the \textit{Howe factorization}, a factorization of characters of $S$ defined in the $\GL_n$ setting in his construction of supercuspidal representations of $\GL_n$.
The Howe factorization comes with an associated sequence of subgroups which capture the relative genericity of the individual factors in the product. Although the characters in this factorization are \textit{not} unique, the associated subgroups \textit{are}. 
In this section, following \cite[Section 3.6]{MR4013740}, we recall how the Howe factorization attaches a tame twisted Levi sequence to a tame elliptic regular pair. 
In fact, we may work in a more general setting; let $\bfS$ be a tamely ramified maximal torus of $\bfG$ defined over $F$ and $\theta$ a character of $S$.
Hence here $(\bfS,\theta)$ is allowed to be a pair which is not tame elliptic regular.

For $r\in\widetilde{\R}_{>0}$, we define a subset $R_{r}$ of $R(\bfS,\bfG)$ as in \cite[1107 page, (3.6.1)]{MR4013740} by
\[
R_{r}
\colonequals 
\{
\alpha\in R(\bfS,\bfG) 
\mid
\theta|_{\Nr_{E/F}(\alpha^{\vee}(E_{r}^{\times}))} \equiv \mathbbm{1}
\},
\]
where $E$ is the splitting field of the torus $\bfS$.
Note that $R_{0+}$ introduced in the definition of a tame elliptic regular pair is nothing but a special case of $R_{r}$ where $r$ is taken to be $0+$.
We let $r_{d-1}>\cdots >r_{1}>r_{0}$ be the real numbers satisfying $R_{r}\subsetneq R_{r+}$.
We put $r_{d}\colonequals \depth(\theta)$ and $r_{-1}\colonequals 0$ (note that $r_{d}\geq r_{d-1}$).
Let $\bfG^{0}\subsetneq\cdots\subsetneq\bfG^{d-1}$ be the tamely ramified reductive subgroups of $\bfG$ corresponding to the sequence $R_{r_{0}}\subsetneq\cdots\subsetneq R_{r_{d-1}}$ (i.e., each $\bfG^{i}$ is the reductive subgroup of $\bfG$ which contains $\bfS$ and has $R_{r_{i}}$ as its roots).
We put $\bfG^{-1}\colonequals \bfS$ and $\bfG^{d}\colonequals \bfG$ so that we have $\bfG^{-1}\subset\bfG^{0}\subsetneq\cdots\subsetneq\bfG^{d-1}\subsetneq\bfG^{d}$.
When we want to emphasize the dependence on $(\bfS, \theta)$, we write $\bfG^i(\bfS,\theta)$ for the twisted Levi $\bfG^i$ associated to $(\bfS, \theta)$ via the Howe factorization.

We introduce the notion of torality and $0$-torality as follows:
\begin{defn}\label{defn:toral}
Let $\bfS$ be a tamely ramified maximal torus of $\bfG$ defined over $F$.
Let $\theta$ be a character of $S$.
\begin{enumerate}
\item
We call $\theta$ a \textit{toral} character if $\bfG^{0}(\bfS,\theta)=\bfS$.
\item
We call $\theta$ a \textit{$0$-toral} character if $d=1$ and $\bfG^{0}(\bfS,\theta)=\bfS$ (i.e., $\vec{\bfG}$ consists only of $\bfG^{0}=\bfS$ and $\bfG^{1}=\bfG$).
\end{enumerate}
\end{defn}

We note that, when $p$ is not bad for $\bfG$, each $\bfG^{i}(\bfS,\theta)$ is a tame twisted Levi subgroup of $\bfG$ by \cite[Lemma 3.6.1]{MR4013740}.

We now return to the setting of elliptic $\bfS$. 
The following lemma will be useful for us, especially in Section \ref{sec:comparison}. We warn the reader that this lemma does not hold for arbitrary $p$!
We discuss subtleties regarding small residue characteristic in Section \ref{sec:small p}.

\begin{lemma}\label{lem:theta-stab}
Assume $p$ is odd, $p$ is not bad for $\bfG$, $p \nmid |\pi_1(\bfG_{\der})|$ and $p \nmid |\pi_1(\widehat \bfG_{\der})|$. Let $(\bfS, \theta)$ be a tame elliptic regular pair. If $\theta$ is a toral character of $S$, then $\theta|_{S_{0+}}$ has trivial $W_G(\bfS)$-stabilizer.
\end{lemma}

\begin{proof}
By \cite[Proposition 3.6.7]{MR4013740}, $\theta$ has a Howe factorization. More precisely, there exists a regular generic reduced cuspidal $\bfG$-datum $(\vec{\bfG}, \pi_{-1}, \vec{\theta})$ where $\vec{\bfG}$ is the twisted Levi sequence attached to $\theta$ as above, $\pi_{-1} \cong \pi_{(\bfS, \theta_{-1})}^{\bfG^0}$, and $\vec{\theta} = (\theta_0, \ldots, \theta_d)$ is a choice of characters of $(G^0, \ldots, G^d)$ such that $\theta = \prod_{i=-1}^d \theta_i|_S$. By \cite[Lemma 3.6.5 (2)]{MR4013740}, we have $\Stab_{W_{G}(\bfS)}(\theta|_{S_{0+}})=\Stab_{W_{G^{0}}(\bfS)}(\theta_{-1}|_{S_{0+}}) = W_{G^{0}}(\bfS)$, where the last equality holds since $\theta_{-1}|_{S_{0+}}$ is trivial. The conclusion clearly follows as $\theta$ being toral implies that $\bfG^0 = \bfS$.
\end{proof}

\subsection{Point stabilizer vs.\ parahoric subgroup}\label{subsec:stab-parah}
Let $(\bfS,\phi)$ be a tame elliptic regular pair.
As explained in Section \ref{subsec:Kaletha-TER}, this pair gives rise to a $\bfG$-equivalence class of cuspidal $\bfG$-data; let $\Sigma= (\vec{\bfG}, \vec{\phi}, \vec{r},\x,\rho'_{0})$ be a representative of this equivalence class, where $\rho'_{0}\cong\Ind_{SG^{0}_{\x,0}}^{G^{0}_{\bar{\x}}}\kappa_{(\bfS,\phi_{-1})}$.
Recall that, in Section \ref{subsec:Yu}, we considered the groups
\[
G^{i}_{\bar{\x}}
\supset
K_{\sigma_{i}}
\supset
G^{i-1}_{\bar{\x}}G^{i}_{\x,s_{i-1}}
\supset
K^{i}
\]
and defined the representations $\tau_{i}$, $\sigma_{i}$, and $\tilde{\rho}'_{i}$ by inducing $\rho'_{i}$:
\[
\tau_{i}
\colonequals 
\Ind_{K_{\sigma_{i}}}^{G^{i}_{\bar{\x}}} \sigma_{i}\otimes\phi_{i},\quad
\sigma_{i}
\colonequals 
\Ind_{G^{i-1}_{\bar{\x}}G^{i}_{\x,s_{i-1}}}^{K_{\sigma_{i}}} \tilde{\rho}'_{i},\quad
\tilde{\rho}'_{i}
\colonequals 
\Ind_{K^{i}}^{G^{i-1}_{\bar{\x}}G^{i}_{\x,s_{i-1}}}\rho'_{i}.
\]

For our convenience, we also introduce the following slightly smaller groups by replacing the role of $G^i_{\bar \x}$ with $SG^i_{\x,0}$:

\[
\xymatrix@R=10pt{
G_{\bar{\x}}^{i} & SG_{\x,0}^{i} \ar@{}[l]|*{\supset} \\
K_{\sigma_{i}}\colonequals G_{\bar{\x}}^{i-1}G_{\x,0+}^{i} \ar@{}[u]|{\bigcup} & \cc K_{\sigma_{i}}\colonequals SG_{\x,0}^{i-1}G_{\x,0+}^{i} \ar@{}[l]|*{\supset} \ar@{}[u]|{\bigcup} \\
G_{\bar{\x}}^{i-1}G_{\x,s_{i-1}}^{i} \ar@{}[u]|{\bigcup} & SG_{\x,0}^{i-1}G_{\x,s_{i-1}}^{i} \ar@{}[l]|*{\supset}  \ar@{}[u]|{\bigcup}\\
K^{i}\colonequals G_{\bar{\x}}^{0}(G^{0},\ldots,G^{i})_{\x,(0+,s_{0},\ldots,s_{i-1})} \ar@{}[u]|{\bigcup} & \cc K^{i}\colonequals SG_{\x,0}^{0}(G^{0},\ldots,G^{i})_{\x,(0+,s_{0},\ldots,s_{i-1})}\ar@{}[l]|*{\supset} \ar@{}[u]|{\bigcup}
}
\]

We define a representation $\cc\rho'_{0}$ of $SG_{\x,0}^{0}$ by
\[
\cc\rho'_{0}
\colonequals 
\kappa_{(\bfS,\phi_{-1})}
\]
and construct representations $\cc\rho'_{i}$ of $\cc K^{i}$ in the same manner as before starting from $\cc\rho'_{0}$ instead of $\rho'_{0}$.

\begin{lem}\label{lem:stab-parah}
The $K^i$-representation $\rho'_{i}$ is isomorphic to the induction of the $\cc K^i$-representation $\cc\rho'_{i}$:
\[
\rho'_{i}
\cong
\Ind_{\cc K^{i}}^{K^{i}}\cc\rho'_{i}.
\]
\end{lem}

\begin{proof}
The assertion for $i=0$ follows from the definition of $\cc\rho'_{0}$ that we have $\rho'_{0}\cong\Ind_{SG^{0}_{\x,0}}^{G^{0}_{\bar{\x}}}\kappa_{(\bfS,\phi_{-1})}$.
Let us check the assertion for $i+1$ by assuming its validity for $i$.
Recall that the representation $\rho'_{i+1}$ is the push-forward of 
\[
(\tilde{\phi_{i}}|_{K^{i}\ltimes J^{i+1}})
\otimes 
\bigl((\rho'_{i}\otimes\phi_{i}|_{K^{i}})\ltimes\mathbbm{1}\bigr)
\]
via the canonical map $K^{i}\ltimes J^{i+1}\twoheadrightarrow K^{i}J^{i+1}=K^{i+1}$.
Similarly, the representation $\cc\rho'_{i+1}$ is the push-forward of 
\[
(\tilde{\phi_{i}}|_{\cc K^{i}\ltimes J^{i+1}})
\otimes 
\bigl((\cc\rho'_{i}\otimes\phi_{i}|_{\cc K^{i}})\ltimes\mathbbm{1}\bigr)
\]
via the canonical map $\cc K^{i}\ltimes J^{i+1}\twoheadrightarrow \cc K^{i}J^{i+1}=\cc K^{i+1}$. Noting that $(\rho_i' \otimes \phi_i|_{K^i}) \ltimes \mathbbm 1 = (\rho_i' \ltimes \mathbbm 1) \otimes (\phi_i|_{K^i} \ltimes \mathbbm 1)$, it is enough to check that we have
\begin{multline*}
\Ind_{\cc K^{i}\ltimes J^{i+1}}^{K^{i}\ltimes J^{i+1}}
\Bigl((\tilde{\phi_{i}}|_{\cc K^{i}\ltimes J^{i+1}})
\otimes 
\bigl((\cc\rho'_{i}\otimes\phi_{i}|_{\cc K^{i}})\ltimes\mathbbm{1}\bigr)\Bigr) \\
\cong
(\tilde{\phi_{i}}|_{K^{i}\ltimes J^{i+1}})
\otimes 
\bigl((\rho'_{i} \ltimes \mathbbm 1) \otimes (\phi_{i}|_{K^{i}} \ltimes\mathbbm{1})\bigr).
\end{multline*}
Obviously $\tilde{\phi_{i}}|_{\cc K^{i}\ltimes J^{i+1}}$ is the restriction of $\tilde{\phi_{i}}|_{K^{i}\ltimes J^{i+1}}$ and $\phi_i|_{\cc K^i}$ is the restriction of $\phi_i|_{K^i}$. By the projection formula,
\begin{multline*}
\Ind_{\cc K^{i}\ltimes J^{i+1}}^{K^{i}\ltimes J^{i+1}}
\Bigl((\tilde{\phi_{i}}|_{\cc K^{i}\ltimes J^{i+1}})
\otimes 
\bigl((\cc\rho'_{i}\otimes\phi_{i}|_{\cc K^{i}})\ltimes\mathbbm{1}\bigr)\Bigr) \\
\cong
(\tilde{\phi_{i}}|_{K^{i}\ltimes J^{i+1}})
\otimes
\Bigl(\Ind_{\cc K^{i}\ltimes J^{i+1}}^{K^{i}\ltimes J^{i+1}} (\cc\rho'_{i}\ltimes\mathbbm{1})
\otimes
(\phi_{i}|_{K^{i}}\ltimes\mathbbm{1})\Bigr).
\end{multline*}
As we are supposing that $\rho'_{i}\cong\Ind_{\cc K^{i}}^{K^{i}}\cc\rho'_{i}$, we get $\Ind_{\cc K^{i}\ltimes J^{i+1}}^{K^{i}\ltimes J^{i+1}} (\cc\rho'_{i}\ltimes\mathbbm{1}) \cong \rho_i' \ltimes \mathbbm 1$.
\end{proof}

We define representations $\cc\tau_{i}$, $\cc\sigma_{i}$, and $\cc\tilde{\rho}'_{i}$ by inducing $\cc\rho'_{i}$:
\[
\cc\tau_{i}
\colonequals 
\Ind_{\cc K_{\sigma_{i}}}^{SG^{i}_{\x,0}} \cc\sigma_{i}\otimes\phi_{i},\quad
\cc\sigma_{i}
\colonequals 
\Ind_{SG^{i-1}_{\x,0}G^{i}_{\x,s_{i-1}}}^{\cc K_{\sigma_{i}}} \cc\tilde{\rho}'_{i},\quad
\cc\tilde{\rho}'_{i}
\colonequals 
\Ind_{\cc K^{i}}^{SG^{i-1}_{\x,0}G^{i}_{\x,s_{i-1}}}\cc\rho'_{i}.
\]
Then Lemma \ref{lem:stab-parah} implies that
\[
\tau_{i}
\cong
\Ind_{SG^{i}_{\x,0}}^{G^{i}_{\bar{\x}}} \cc\tau_{i},\quad
\sigma_{i}
\cong
\Ind_{\cc K_{\sigma_{i}}}^{K_{\sigma_{i}}} \cc\sigma_{i},\quad
\tilde{\rho}'_{i}
\cong
\Ind_{SG^{i-1}_{\x,0}G^{i}_{\x,s_{i-1}}}^{G^{i-1}_{\bar{\x}}G^{i}_{\x,s_{i-1}}}\cc\tilde{\rho}'_{i}.
\]

\[
\xymatrix@R=10pt{
G_{\bar{\x}}^{i} & SG_{\x,0}^{i} \ar@{}[l]|*{\supset} & \tau_{i} & \cc\tau_{i}\ar@{-->}[l]\\
K_{\sigma_{i}} \ar@{}[u]|{\bigcup} & \cc K_{\sigma_{i}} \ar@{}[l]|*{\supset} \ar@{}[u]|{\bigcup} & \sigma_{i}\otimes\phi_{i} \ar@{-->}[u]& \cc\sigma_{i}\otimes\phi_{i}\ar@{-->}[l]\ar@{-->}[u]\\
G_{\bar{\x}}^{i-1}G_{\x,s_{i-1}}^{i} \ar@{}[u]|{\bigcup} & SG_{\x,0}^{i-1}G_{\x,s_{i-1}}^{i} \ar@{}[l]|*{\supset}  \ar@{}[u]|{\bigcup} & \tilde{\rho}'_{i}\otimes\phi_{i} \ar@{-->}[u]& \cc\tilde{\rho}'_{i}\otimes\phi_{i}\ar@{-->}[l]\ar@{-->}[u]\\
K^{i} \ar@{}[u]|{\bigcup} & \cc K^{i} \ar@{}[l]|*{\supset} \ar@{}[u]|{\bigcup} & \rho'_{i}\otimes\phi_{i} \ar@{-->}[u]& \cc\rho'_{i}\otimes\phi_{i}\ar@{-->}[l]\ar@{-->}[u]
}
\]

\section{Adler--DeBacker--Spice's character formula}\label{sec:AS}

The purpose of this section is to give a character formula for a crucial intermediate representation used in the Yu's construction of supercuspidal representations, following the theory of Adler--DeBacker--Spice (\cite{MR2543925, MR3849622}). 
We first note that several technical assumptions on $p$ are required so that their theory works well (see \cite[Section 2]{MR2431235}), but it is enough to assume only the oddness and the non-badness of $p$ whenever $\bfG$ is tamely ramified (see \cite[Section 4.1]{MR4013740}). 

We focus on the setting of \textit{Howe-unramified} regular supercuspidal representations. 
To this end, let $(\bfS, \phi)$ be a tame elliptic regular pair such that
\begin{equation*}
\text{$\bfS$ is an unramified.}
\end{equation*}
Let $(\vec{\bfG}, \pi_{-1}, \vec{\phi})$ be a regular generic reduced cuspidal $\bfG$-datum corresponding to $(\bfS, \phi)$ as in Section \ref{subsec:Kaletha-TER} and recall the various intermediate representations described in Section \ref{subsec:stab-parah}. 
In Section \ref{subsec:characters_vreg} we give a simple character formula (Proposition \ref{prop:AS-vreg}) for the $SG_{\x,0}$-representation $\cc \tau_d$ on the locus of \textit{unramified very regular elements} (Definition \ref{defn:ur-vreg}, following \cite{MR4197070}).

\subsection{DeBacker--Spice's invariants}\label{subsec:DS}
In this section, we recall several invariants introduced by DeBacker--Spice in \cite[Section 4.3]{MR3849622} to describe the characters of Yu's supercuspidal representations.

Let $\J$ be a connected reductive group over $F$.
We take a maximal torus $\bfT_{\J}$ of $\J$ defined over $F$.
Then we get the set $R(\J,\bfT_{\J})$ of absolute roots of $\bfT_{\J}$ in $\J$ which has an action of $\Gamma_{F}$.
For each $\alpha\in R(\J,\bfT_{\J})$, we put $\Gamma_{\alpha}$ (resp.\ $\Gamma_{\pm\alpha}$) to be the stabilizer of $\alpha$ (resp.\ $\{\pm\alpha\}$) in $\Gamma_{F}$.
Let $F_{\alpha}$ (resp.\ $F_{\pm\alpha}$) be the subfield of $\ol{F}$ fixed by $\Gamma_{\alpha}$ (resp.\ $\Gamma_{\pm\alpha}$): 
\[
F\subset F_{\pm\alpha}\subset F_{\alpha}
\quad
\longleftrightarrow
\quad
\Gamma_{F}\supset \Gamma_{\pm\alpha}\supset \Gamma_{\alpha}.
\]
\begin{itemize}
\item
When $F_{\alpha}=F_{\pm\alpha}$, we say $\alpha$ is an \textit{asymmetric} root.
\item
When $F_{\alpha}\supsetneq F_{\pm\alpha}$, we say $\alpha$ is a \textit{symmetric} root.
Note that, in this case, the extension $F_{\alpha}/F_{\pm\alpha}$ is necessarily quadratic.
Furthermore,
\begin{itemize}
\item
when $F_{\alpha}/F_{\pm\alpha}$ is unramified, we say $\alpha$ is \textit{symmetric unramified}, and
\item
when $F_{\alpha}/F_{\pm\alpha}$ is ramified, we say $\alpha$ is \textit{symmetric ramified}.
\end{itemize}
\end{itemize}
Note that a root $\alpha$ is symmetric if and only if the $\Gamma_{F}$-orbit of $\alpha$ contains $-\alpha$.
We write $R(\J,\bfT_{\J})^{\sym}$, $R(\J,\bfT_{\J})_{\sym}$, $R(\J,\bfT_{\J})_{\sym,\ur}$, and $R(\J,\bfT_{\J})_{\sym,\ram}$ for the set of asymmetric roots, symmetric roots, symmetric unramified roots, and symmetric ramified roots, respectively.

According to \cite[Definition 3.6]{MR3849622}, for each $\alpha\in R(\J,\bfT_{\J})$ and a point $\y\in\mcB^{\red}(\J,F)$, we define a set $\ord_{\y}\alpha$ of real numbers by
\[
\ord_{\y}\alpha
\colonequals 
\{i\in\R \mid \mfg_{\alpha}(F_{\alpha})_{\y,i:i+}\neq0\},
\]
where $\mfg_{\alpha}$ is the root space of $\alpha$ in the Lie algebra $\mfg$ of $\bfG$ and $\mfg_{\alpha}(F_{\alpha})_{\y,i}\colonequals \mfg(F_{\alpha})_{\y,i}\cap\mfg_{\alpha}(F_{\alpha})$.
Then, for a positive real number $r\in\R_{>0}$, we define a subset $R^{\J}_{\y,r/2}$ of $R(\J,\bfT_{\J})$ by
\[
R^{\J}_{\y,r/2}
\colonequals 
\{\alpha\in R(\J,\bfT_{\J}) \mid r/2\in\ord_{\y}\alpha\}.
\]
For an element $\delta$ of $T_{\J}=\bfT_{\J}(F)$, we put
\[
R^{\J}_{\y,(r-\ord_{\delta})/2}
\colonequals 
\{
\alpha\in R(\J,\bfT_{\J})
\mid
\alpha(\delta)\neq1,\, (r-\ord_{\delta}{\alpha})/2\in\ord_{\y}\alpha
\},
\]
where $\ord_{\delta}{\alpha}\colonequals \ord(\alpha(\delta)-1)$.

Now let us recall the definition of the invariants $\varepsilon^{\ram}$, $\varepsilon_{\sym,\ram}$, and $\tilde{e}$ (\cite[Definition 4.14]{MR3849622}).
Let $\delta$ be an element of the parahoric subgroup $T_{\J,0}$ of $T_{\J}$.
\begin{itemize}
\item
We first define a sign $\varepsilon_{\alpha}(\delta)$ for an asymmetric or symmetric unramified root $\alpha$ as follows:
\begin{itemize}
\item
For an asymmetric root $\alpha\in R(\J,\bfT_{\J})^{\sym}$, we put
\[
\varepsilon_{\alpha}(\delta)
\colonequals 
\Jac{\overline{\alpha(\delta)}}{k_{F_{\alpha}}^{\times}},
\]
where $\Jac{-}{k_{F_{\alpha}}^{\times}}$ is the unique nontrivial quadratic character of the multiplicative group $k_{F_{\alpha}}^{\times}$ of the residue field $k_{F_{\alpha}}$ of $F_{\alpha}$.
\item
For a symmetric unramified root $\alpha\in R(\J,\bfT_{\J})_{\sym,\ur}$, we put
\[
\varepsilon_{\alpha}(\delta)
\colonequals 
\Jac{\overline{\alpha(\delta)}}{k_{F_{\alpha}}^{1}},
\]
where $\Jac{-}{k_{F_{\alpha}}^{1}}$ is the unique nontrivial quadratic character of the kernel $k_{F_{\alpha}}^{1}$ of the norm map $\Nr_{k_{F_{\alpha}}/k_{F_{\pm\alpha}}}\colon k_{F_{\alpha}}^{\times}\twoheadrightarrow k_{F_{\pm\alpha}}^{\times}$.
\end{itemize}
Here, in both cases, $\alpha(\delta)$ belongs to $\mcO_{F_{\alpha}}^{\times}$ since $\delta$ belongs to the parahoric subgroup $T_{\J,0}$ of $T_{\J}$.
Thus we can take its reduction $\overline{\alpha(\delta)}\in k_{F_{\alpha}}^{\times}$.
Note that, in the latter case, the definition makes sense since $\alpha(\delta)$ always belongs to the kernel of the norm map $\Nr_{F_{\alpha}/F_{\pm\alpha}}\colon F_{\alpha}^{\times}\rightarrow F_{\pm\alpha}^{\times}$, hence $\overline{\alpha(\delta)}\in k_{F_{\alpha}}^{1}$.

Then we define
\[
\varepsilon^{\sym}(\J,\bfT_{\J}, r,\delta)
\colonequals 
\prod_{\alpha\in\Gamma_{F}\times\{\pm1\}\backslash(R^{\J}_{\y,r/2}\cap R(\J,\bfT_{\J})^{\sym})} \varepsilon_{\alpha}(\delta),
\]
\[
\varepsilon_{\sym,\ur}(\J,\bfT_{\J}, r,\delta)
\colonequals 
\prod_{\alpha\in\Gamma_{F}\backslash(R^{\J}_{\y,r/2}\cap R(\J,\bfT_{\J})_{\sym,\ur})} \varepsilon_{\alpha}(\delta),
\]
\[
\varepsilon^{\ram}(\J,\bfT_{\J}, r,\delta)
\colonequals 
\varepsilon^{\sym}(\J,\bfT_{\J}, r,\delta)
\cdot
\varepsilon_{\sym,\ur}(\J,\bfT_{\J}, r,\delta).
\]
Here the index set of the first product denotes the set of orbits under the action of $\Gamma_{F}\times\{\pm1\}$ (the action of $-1\in\{\pm1\}$ is given by $\alpha\mapsto-\alpha$).
\item
We put $\varepsilon_{\sym,\ram}(\J,\bfT_{\J}, r,\delta)$ to be the product of 
\[
(-1)^{\rank_{F_{\pm\alpha}}(\J_{\pm\alpha})-1}(-\mathfrak{G})^{f_{\alpha}}
\Jac{t_{\alpha}}{k_{F_{\alpha}}^{\times}}
\sgn_{F_{\pm\alpha}}(\J_{\pm\alpha})
\]
over $\alpha\in \Gamma_{F}\backslash (R^{\J}_{\y,(r-\ord_{\delta})/2}\cap R(\J,\bfT_{\J})_{\sym,\ram})$.
We do not recall the definitions of various symbols appearing here (see \cite[Definition 4.14]{MR3849622} for the details).
We only remark that the index set is empty when $\bfT_{\J}$ is unramified.
In particular, $\varepsilon_{\sym,\ram}(\J,\bfT_{\J}, r,\delta)$ is always trivial as long as $\bfT_{\J}$ is unramified.
\item
We put
\[
\tilde{e}(\J,\bfT_{\J}, r,\delta)
\colonequals 
(-1)^{|\Gamma_{F}\backslash R^{\J}_{\y,(r-\ord_{\delta})/2}|}.
\]
\end{itemize}

\begin{rem}\label{rem:signs}
\begin{itemize}
\item[(i)]
When $\bfT_{\J}$ is elliptic in $\J$, for any $\delta\in T_{\J}$, the image $\alpha(\delta)$ of $\delta$ under the root $\alpha\colon \bfT_{\J}(F_{\alpha})\rightarrow F_{\alpha}^{\times}$ belongs to $\mcO_{F_{\alpha}}^{\times}$.
Hence the above definition makes sense for any $\delta\in T_{\J}$ and $\varepsilon^{\ram}(\J,\bfT_{\J}, r,\delta)$ defines a character on $T_{\J}$.
\item[(ii)]
By definition, all of $\varepsilon^{\ram}(\J,\bfT_{\J}, r,\delta)$, $\varepsilon_{\sym,\ram}(\J,\bfT_{\J}, r,\delta)$, and $\tilde{e}(\J,\bfT_{\J}, r,\delta)$ are invariant under $J$-conjugation (\cite[Remark 4.18]{MR3849622}).
\end{itemize}
\end{rem}

\subsection{Character values at unramified very regular elements}\label{subsec:characters_vreg}
Let $(\bfS,\phi)$ be a tame elliptic regular pair.
We take a regular generic reduced cuspidal $\bfG$-datum
\[
(\vec{\bfG}, \pi_{-1},\vec{\phi})
\colonequals 
\bigl(\bfG^{0}\subsetneq\bfG^{1}\subsetneq\cdots\subsetneq\bfG^{d}, \pi_{-1}\cong\pi_{(\bfS,\phi_{-1})}^{\bfG^{0}}, (\phi_{0},\ldots,\phi_{d})\bigr)
\]
corresponding to $(\bfS,\phi)$ as in Section \ref{sec:Yu}.
In the following, we assume that
\[
\text{$\bfS$ is an unramified elliptic maximal torus of $\bfG$.}
\]
Let us list a few consequences of this assumption:

\begin{itemize}
\item
The whole group $\bfG$ is necessarily unramified (hence so are $\bfG^{0},\ldots,\bfG^{d-1}$).
Note that, in fact, the unramifiedness of $\bfS$ is equivalent to that of $\bfG^{0}$ (see the paragraph after \cite[Definition 3.4.2]{MR4013740}).
\item
Let $\x$ be a point of $\mcB(\bfG^{0},F)$ belonging to the apartment $\mathcal{A}(\bfS,F)$ of $\bfS$.
Recall that the image of such a point in $\mcB^{\red}(\bfG^{0},F)$ is uniquely determined.
Then we have 
\[
G^{0}_{\bar{\x}}\supset SG^{0}_{\x,0}=Z_{\bfG^{0}}G^{0}_{\x,0}=Z_{\bfG}G^{0}_{\x,0}.
\]
(This follows from the property $S=S_{0}Z_{\bfG}$, see \cite[Lemma 7.1.1]{MR2805068}).
\item
Thanks to the equality $SG^{0}_{\x,0}=Z_{\bfG}G^{0}_{\x,0}$, Kaletha's extension $\kappa_{(\bfS,\phi_{-1})}$ of the representation $\kappa_{(\bfS,\bar{\phi})}$ of $G^{0}_{\x,0}$ to $SG^{0}_{\x,0}$ is simply described as follows: for $g=z\cdot g_{0}\in Z_{\bfG}G^{0}_{\x,0}$, we have
\[
\kappa_{(\bfS,\phi_{-1})}(g)
=
\phi_{-1}(z)\cdot\kappa_{(\bfS,\bar{\phi})}(g_{0}).
\]
\end{itemize}

Following \cite{MR4197070}, we introduce the notion of the \textit{unramified very regularity} for elements of $SG_{\x,0}$ as follows:
\begin{defn}\label{defn:ur-vreg}
We say that an element $\gamma\in SG_{\x,0}\,(=Z_{\bfG}G_{\x,0})$ is \textit{unramified very regular (with respect to $\bfS$)} if it satisfies the following conditions:
\begin{enumerate}
\item
$\gamma$ is regular semisimple in $\bfG$,
\item
the connected centralizer $\bfT_{\gamma}\colonequals C_{\bfG}(\gamma)^{\circ}$ is an unramified maximal torus of $\bfG$ whose apartment $\mathcal{A}(\bfT_{\gamma},F)$ contains the point $\x$, and
\item
$\alpha(\gamma)\not\equiv1\pmod{\mfp}$ for any root $\alpha$ of $\bfT_{\gamma}$ in $\bfG$.
\end{enumerate}
\end{defn}

The purpose of this section is to give an explicit formula of the character $\Theta_{\cc\tau_{d}}$ of $\cc\tau_{d}$ at unramified very regular elements of $SG_{\x,0}$ (Proposition \ref{prop:AS-vreg}).
Our formula is just an easy consequence of the theory of Adler--DeBacker--Spice (established in \cite{MR2543925} and \cite{MR3849622}) and not new at all.
However we have to be careful of the following points:
\begin{itemize}
\item
for our purpose, we want to compute the character of $\cc\tau_{d}$, not $\tau_{d}$ as in \cite{MR2543925}, and
\item
in the full character formula \cite[Theorem 7.1]{MR2543925}, the index set is expressed via a set of conjugates of $\gamma$, not a set of elements conjugating $\gamma$ as in Proposition \ref{prop:AS-vreg}.
\end{itemize}
For these reasons, we cannot deduce the formula of Proposition \ref{prop:AS-vreg} from the results of \cite{MR2543925} immediately.
What we will do in the following is just to repeat the proofs of \cite[Theorems 6.4 and 7.1]{MR2543925} while paying attention to these points.
However, we note that most of delicate computations carried out in \cite{MR2543925} can be skipped by focusing only on unramified very regular elements.

We start from checking that every unramified very regular elements of $SG_{\x,0}$ satisfies the following basic properties necessary for the computation of Adler--Spice:
\begin{lem}\label{lem:AS-vreg}
Let $\gamma$ be an unramified very regular element of $SG_{\x,0}$.
Then, for any $r\in\R_{\geq0}$, 
\begin{itemize}
\item[(i)]
$\gamma$ has a normal $r$-approximation, and
\item[(ii)]
the point $\x$ belongs to the set $\mcB_{r}(\gamma)$ defined in \cite[Definition 9.5]{MR2431235}.
\end{itemize}
\end{lem}

\begin{proof}
Since we have $SG_{\x,0}=Z_{\bfG}G_{\x,0}$ by the unramifiedness of $\bfS$, it is enough to treat the case where $\gamma\in G_{\x,0}$.

When $r=0$, by definition, any element of $G_{\bar{\x}}$ has a normal $0$-approximation (\cite[Definition 6.8]{MR2431235}).
Furthermore, since we have 
\[
\mcB_{0}(\gamma)
\colonequals 
\{\y\in\mcB(\bfG,F) \mid \gamma\cdot\bar{\y}=\bar{\y}\}
\]
(\cite[Definition 9.5]{MR2431235}), the assertion (ii) is obvious.

In the following, we consider the case where $r>0$.
We first check the assertion (i).
Since $\gamma$ belongs to $G_{\x,0}$, in particular it is bounded.
Moreover, by the unramified very regularity, $\gamma$ belongs to an unramified torus $\bfT_{\gamma}$, which is the connected centralizer of $\gamma$.
Since we are assuming that $p$ is not bad for $\bfG$, \cite[Theorem 3.6]{Fintzen:18} implies that $\bfT_{\gamma}$ satisfies the condition $(\mathbf{Gd}^{G})$ of \cite[Definition 6.3]{MR2431235} (see also \cite[Remarks 3.4 and 3.7]{Fintzen:18}).
Hence $\gamma$ has a normal $r$-approximation by \cite[Lemma 8.1]{MR2431235}.
More precisely, by applying \cite[Lemma 8.1]{MR2431235} to $\bfG'=\bfG$, $d=0$, $r=r$, and $\bfS=\bfT_{\gamma}$, we can find a normal $r$-approximation of $\gamma$.
Note that we do not need to check the assumption that $\widetilde{G}_{0+}$ is contained in the image of $G_{0+}$ of \cite[Lemma 8.1]{MR2431235} by applying Spice's topological Jordan decomposition (\cite{MR2408311}) to $\gamma$ itself instead of $\overline{\gamma}$ in the proof of \cite[Lemma 8.1]{MR2431235}.

We take a normal $r$-approximation $(\gamma_{i})_{0\leq i<r}$ of $\gamma$ and put
\[
\gamma_{<r}
\colonequals 
\prod_{0\leq i <r}\gamma_{i}
\quad\text{and}\quad
\gamma_{\geq r}
\colonequals 
\gamma\cdot\gamma_{<r}^{-1}.
\]
We next check that the point $\x$ belongs to the set $\mcB_{r}(\gamma)$.
For this, we recall that the set $\mcB_{r}(\gamma)$ is defined by
\[
\mcB_{r}(\gamma)
\colonequals 
\{\y\in\mcB(C_{\bfG}^{(r)}(\gamma),F) \mid Z_{G}^{(r)}(\gamma)\gamma\cap G_{\y,r}\neq\varnothing\}.
\]
Here the group $C_{\bfG}^{(r)}(\gamma)$ is defined by
\[
C_{\bfG}^{(r)}(\gamma)
\colonequals 
\Bigl(\bigcap_{0\leq i<r}C_{\bfG}(\gamma_{i})\Bigr)^{\circ}
\]
and $Z_{G}^{(r)}(\gamma)$ denotes the group consisting of the $F$-valued points of the center of $C_{\bfG}^{(r)}(\gamma)$.
We note that the depth-zero part $\gamma_{0}$ of $\gamma$ is regular semisimple in $\bfG$ by the unramified very regularity of $\gamma$.
Indeed, by the definition of a normal approximation, we have 
\[
\alpha(\gamma)
=
\alpha(\gamma_{<r}\cdot \gamma_{\geq r})
=
\prod_{0\leq i <r}\alpha(\gamma_{i})\cdot \alpha(\gamma_{\geq r})
\equiv
\alpha(\gamma_{0})
\pmod{\mfp}
\]
for any root $\alpha$ of $\bfT_{\gamma}$ in $\bfG$.
Since $\alpha(\gamma)\not\equiv1\pmod{\mfp}$, so is $\alpha(\gamma_{0})$, hence in particular $\gamma_{0}$ is regular semisimple.
Hence the connected centralizer $C_{\bfG}(\gamma_{0})^{\circ}$ of $\gamma_{0}$ in $\bfG$ is a maximal torus.
By noting the commutativity of $\gamma$ with $\gamma_{0}$, we see that $C_{\bfG}(\gamma_{0})^{\circ}$ is nothing but $\bfT_{\gamma}$.
Thus the group $C_{\bfG}^{(r)}(\gamma)$ is contained in $C_{\bfG}(\gamma_{0})^{\circ}=\bfT_{\gamma}$.
On the other hand, again by the definition of the normal approximation, there exists a tame torus containing $\gamma_{i}$ for any $0\leq i <r$.
As $\gamma_{0}$ is regular semisimple, this torus is necessarily equal to $\bfT_{\gamma}$.
By this observation, we know that each group $C_{\bfG}(\gamma_{i})$ contains $\bfT_{\gamma}$.
In conclusion, we have
\[
C_{\bfG}^{(r)}(\gamma)=\bfT_{\gamma}
\quad\text{and}\quad
Z_{G}^{(r)}(\gamma)=T_{\gamma}.
\]
By the unramified very regularity of $\gamma$, the building $\mcB(C_{\bfG}^{(r)}(\gamma),F)$ contains the point $\x$.
Moreover, obviously the intersection $Z_{G}^{(r)}(\gamma)\gamma\cap G_{\x,r}$ appearing in the definition of $\mcB_{r}(\gamma)$ is not empty (both of $Z_{G}^{(r)}(\gamma)\gamma=T_{\gamma}\gamma=T_{\gamma}$ and $G_{\x,r}$ contains $1$).
\end{proof}

\begin{lem}[{\cite[Corollary 3.4.26]{MR4013740}}]\label{lem:Kaletha}
Let $\gamma_{0}$ be a regular semisimple element of $SG^{0}_{\x,0}$ whose image in $G^{0}_{\ad}$ is topologically semisimple.
Then the character $\Theta_{\kappa_{(\bfS,\phi_{-1})}}(\gamma_{0})$ of $\kappa_{(\bfS,\phi_{-1})}$ at $\gamma_{0}$ is zero unless $\gamma_{0}$ is $SG^{0}_{\x,0}$-conjugate to an element of $S$.
When $\gamma_{0}$ is $SG^{0}_{\x,0}$-conjugate to an element of $S$ (then we may assume that $\gamma_{0}$ itself lies in $S$), we have
\[
\Theta_{\kappa_{(\bfS,\phi_{-1})}}(\gamma_{0})
=
(-1)^{r(\bbG^{0}_{\x})-r(\bbS)}\sum_{w\in W_{G^{0}_{\x,0}}(\bfS)}\phi_{-1}({}^{w}\gamma_{0}),
\] 
where 
\begin{itemize}
\item
$r(\bbG^{0}_{\x})$ and $r(\bbS)$ are split ranks of $\bbG^{0}_{\x}$ and $\bbS$, respectively, 
\item
$W_{G^{0}_{\x,0}}(\bfS)=N_{G^{0}_{\x,0}}(\bfS)/S_{0}$, and
\item
${}^{w}\gamma_{0}$ denotes the $w$-conjugate of $\gamma_{}$; ${}^{w}\gamma_{0}=w\gamma_{0}w^{-1}$.
\end{itemize}
\end{lem}

\begin{lem}\label{lem:cc K^{d}}
Let $\gamma\in\cc K^{d}=SG^{0}_{\x,0}(G^{0},\ldots,G^{d})_{\x,(0+,s_{0},\ldots,s_{d-1})}$ be an unramified very regular element.
Then $\gamma$ is $\cc K^d$-conjugate to an element of $SG^{0}_{\x,0}$.
\end{lem}

\begin{proof}

Let $\gamma$ be an unramified very regular element of $\cc K^{d}$.
By Lemma \ref{lem:AS-vreg}, we can take a normal $s_{d-1}$-approximation to $\gamma$ and then $\x$ belongs to $\mcB_{s_{d-1}}(\gamma)$.
Then, by applying \cite[Proposition 9.14]{MR2431235} to $r=s_{d-1}$ and $(\bfG',\bfG)=(\bfG^{d-1},\bfG^{d})$, there exists an element $k\in[\gamma;\x,s_{d-1}]_{G^{d}}$ satisfying ${}^{k}Z_{G}^{(s_{d-1})}(\gamma)=kZ_{G}^{(s_{d-1})}(\gamma)k^{-1}\subset G^{d-1}$.

Here we recall that, for a connected reductive group $\J$ over $F$ and an element $\delta\in J$ having a normal $t$-approximation ($t\in \R_{\geq0}$) with respect to $\y\in\mcB(\J,F)$, the group $[\delta; \y,t]_{J}$ is defined to be the group $\vec{J}_{\y,\vec{s}}$ (\cite[Definition 5.14]{MR2431235}) for
\begin{itemize}
\item[--]
$\vec{\J}\colonequals (C_{\J}^{(t-i)}(\delta))_{0<i\leq t}$, and
\item[--]
$\vec{s}\colonequals (i)_{0<i\leq t}$
\end{itemize}
(see \cite[Definition 6.6]{MR2431235} or \cite[Section 1.4]{MR2543925}).
Since our $\gamma$ is unramified very regular, we have
\[
C_{\bfG^{d}}^{(s_{d-1}-i)}(\gamma)
=
\begin{cases}
\bfT_{\gamma} & 0<i<s_{d-1},\\
\bfG^{d} & i=s_{d-1},
\end{cases}
\]
where $\bfT_{\gamma}$ is the connected centralizer of $\gamma$ in $\bfG^{d}$.
Hence we have
\[
[\gamma; \x, s_{d-1}]_{G^{d}}
=
T_{\gamma,0+}G^{d}_{\x,s_{d-1}}.
\]

By also noting that $\gamma\in Z_{G}^{(s_{d-1})}(\gamma)=T_{\gamma}$, we conclude that there exists an element $k'\in G_{\x,s_{d-1}}$ satisfying ${}^{k'}\gamma\in G^{d-1}$.
Therefore, by taking $\cc K^{d}$-conjugation, we may assume that $\gamma$ itself lies in $G^{d-1}$ (note that $G_{\x,s_{d-1}}\subset \cc K^{d}$).
Then, by a descent property of the subgroups associated to concave functions, we have
\begin{align*}
\gamma
\in
\cc K^{d}\cap G^{d-1}
&=
SG^{0}_{\x,0}(G^{0},\ldots,G^{d})_{\x,(0+,s_{0},\ldots,s_{d-1})}\cap G^{d-1}\\
&=
SG^{0}_{\x,0}(G^{0},\ldots,G^{d-1})_{\x,(0+,s_{0},\ldots,s_{d-2})}
=
\cc K^{d-1}
\end{align*}
(similarly to the proof of \cite[Lemma 2.4]{MR2543925}, this is justified by using \cite[Lemmas 5.29 and 5.33]{MR2431235}).
By repeating this argument inductively, we finally conclude that $\gamma$ belongs to $SG^{0}_{\x,0}$ by taking $\cc K^{d}$-conjugation.
\end{proof}

The following sign character is of central importance.

\begin{definition}\label{def:varepsilon ram}
Define the character $\varepsilon^{\ram}[\phi] \from S \to \bbC^\times$ to be 
\begin{equation*}
\varepsilon^{\ram}[\phi](\gamma)\colonequals \prod_{i=0}^{d-1}\varepsilon^{\ram}(\bfG^{i+1}/\bfG^{i},r_{i},\gamma),
\end{equation*}
where
\[
\varepsilon^{\ram}(\bfG^{i+1}/\bfG^{i},r_{i},\gamma)
\colonequals 
\frac{\varepsilon^{\ram}(\bfG^{i+1},\bfS,r_{i},\gamma)}{\varepsilon^{\ram}(\bfG^{i},\bfS,r_{i},\gamma)}.
\]
\end{definition}

\begin{prop}\label{prop:rho_d'}
Let $\gamma$ be an unramified very regular element of $\cc K^{d}$.
\begin{enumerate}
\item
If $\gamma$ is not $\cc K^d$-conjugate to an element of $S$, then we have $\Theta_{\cc\rho_d' \otimes \phi_d}(\gamma)=0$.
\item
If $\gamma$ is $\cc K^d$-conjugate to an element of $S$ (in this case, we assume that $\gamma$ itself belongs to $S$), then we have
\[
\Theta_{\cc\rho_d' \otimes \phi_d}(\gamma)
=
(-1)^{r(\bbG^{0}_{\x})-r(\bbS)+r(\bfS,\phi)}
\sum_{w\in W_{G^{0}_{\x,0}}(\bfS)}
\varepsilon^{\ram}[\phi]({}^{w}\gamma)\phi({}^{w}\gamma),
\]
where
\begin{itemize}
\item
$r(\bbG^{0}_{\x})$ and $r(\bbS)$ are the split ranks of $\bbG^{0}_{\x}$ and $\bbS$,
\item
$r(\bfS,\phi)
\colonequals 
\sum_{i=0}^{d-1} |\Gamma_{F}\backslash (R^{\bfG^{i+1}}_{\x,r_{i}/2}\smallsetminus R^{\bfG^{i}}_{\x,r_{i}/2})|$,
\item
$W_{G_{\x,0}^{0}}(\bfS)=N_{G_{\x,0}^{0}}(\bfS)/S_{0}$.
\end{itemize}
\end{enumerate}
\end{prop}

\begin{proof}
Since the character $\Theta_{\cc\rho_d' \otimes \phi_d}$ is invariant under $\cc K^{d}$-conjugation, we may suppose that $\gamma$ belongs to $SG^{0}_{\x,0}$ by Lemma \ref{lem:cc K^{d}}.
By definition, $\cc\rho'_{d}$ is the representation of $\cc K^d = \cc K^{d-1} J^d$ descended from the $\cc K^{d-1} \ltimes J^d$-representation $(\tilde \phi_{d-1}|_{\cc K^{d-1} \ltimes J^d}) \otimes ((\cc\rho_{d-1}' \otimes \phi_{d-1}|_{\cc K^{d-1}}) \ltimes \mathbbm{1})$. Hence
\[
\Theta_{\cc\rho'_{d}}(\gamma)
=\Theta_{\tilde\phi_{d-1}}(\gamma\ltimes1) \cdot \Theta_{\cc\rho_{d-1}'}(\gamma) \cdot \phi_{d-1}(\gamma).
\]
For the same reason, we have
\[
\Theta_{\cc\rho'_{d-1}}(\gamma)
=\Theta_{\tilde\phi_{d-2}}(\gamma\ltimes1) \cdot \Theta_{\cc\rho_{d-2}'}(\gamma) \cdot \phi_{d-2}(\gamma).
\]
By repeating this computation inductively, we get
\[
\Theta_{\cc\rho'_{d}\otimes\phi_{d}}(\gamma)
=
\Theta_{\cc\rho'_{0}}(\gamma)
\prod_{i=0}^{d-1}\Theta_{\tilde{\phi}_{i}}(\gamma\ltimes1)
\prod_{i=0}^{d}\phi_{i}(\gamma).
\]

We recall that $\cc\rho'_{0}$ is given by $\kappa_{(\bfS,\phi_{-1})}$, hence we have
\[
\Theta_{\cc\rho'_{0}}(\gamma)
=
\Theta_{\kappa_{(\bfS,\phi_{-1})}}(\gamma_{0}).
\]
If $\gamma_{0}$ is not $SG^{0}_{\x,0}$-conjugate to an element of $S$, then we have $\Theta_{\kappa_{(\bfS,\phi_{-1})}}(\gamma_{0})=0$ by Lemma \ref{lem:Kaletha}.
Hence we have $\Theta_{\cc\rho'_{d}\otimes\phi_{d}}(\gamma)=0$ in this case.
By noting that $\gamma_{0}$ is $SG^{0}_{\x,0}$-conjugate to an element of $S$ if and only if so is $\gamma$, we get the assertion (1).

We next compute $\Theta_{\cc\rho'_{d}\otimes\phi_{d}}(\gamma)$ by assuming that $\gamma$ belongs to $S$.
Note that then we have $\bfT_{\gamma}=\bfS$.
In this case, by the argument in the previous paragraph and Lemma \ref{lem:Kaletha}, we have
\begin{align*}
\Theta_{\cc\rho'_{d}\otimes\phi_{d}}(\gamma)
&=
(-1)^{r(\bbG^{0}_{\x})-r(\bbS)}
\sum_{w\in W_{G^{0}_{\x,0}}(\bfS)}\phi_{-1}({}^{w}\gamma_{0})
\prod_{i=0}^{d-1}\Theta_{\tilde{\phi}_{i}}(\gamma\ltimes1)
\prod_{i=0}^{d}\phi_{i}(\gamma)\\
&=
(-1)^{r(\bbG^{0}_{\x})-r(\bbS)}
\prod_{i=0}^{d-1}\Theta_{\tilde{\phi}_{i}}(\gamma\ltimes1)
\sum_{w\in W_{G^{0}_{\x,0}}(\bfS)}
\phi({}^{w}\gamma)
\end{align*}
(recall that $\phi=\prod_{i=-1}^{d}\phi_{i}$ and $\phi_{-1}$ is of depth zero).

By \cite[Proposition 3.8]{MR2543925}, we have
\[
\Theta_{\tilde \phi_{i}}(\gamma\ltimes1)
=
|(C_{G^{i}}^{(0+)}(\gamma),C_{G^{i+1}}^{(0+)}(\gamma))_{\x,(r_{i},s_{i}):(r_{i},s_{i}+)}|^{\frac{1}{2}}
\cdot
\varepsilon(\phi_{i},\gamma)
\]
with the notations used in \cite[Proposition 3.8]{MR2543925}.
Since $\gamma$ is unramified very regular, its depth zero part is regular semisimple in $\bfG$ so that $C_{G^{i}}^{(0+)}(\gamma)=C_{G^{i+1}}^{(0+)}(\gamma) =T_\gamma$.
Hence the factor $|(C_{G^{i}}^{(0+)}(\gamma),C_{G^{i+1}}^{(0+)}(\gamma))_{\x,(r_{i},s_{i}):(r_{i},s_{i}+)}|^{\frac{1}{2}}$ is trivial.

To understand the factors $\varepsilon(\phi_i, \gamma)$ for $0 \leq i < d$, we use \cite[Proposition 4.21]{MR3849622}, which implies that the product 
\[
\mathfrak{G}(\phi_{i},\gamma)\varepsilon(\phi_{i},\gamma)
\]
is equal to 
\[
\varepsilon_{\sym,\ram}(\pi',\gamma)
\cdot
\varepsilon^{\ram}(\pi',\gamma)
\cdot
\tilde{e}(\pi',\gamma),
\]
where $\mathfrak{G}(\phi_{i},\gamma)$ is the constant defined in \cite[Definition 5.24]{MR2543925}.
Here we temporarily follow the notation of \cite{MR3849622}; especially, we are applying results in \cite{MR3849622} by taking $(\bfG,\bfG')$ to be $(\bfG^{i+1},\bfG^{i})$.
See \cite[Definition 4.14]{MR3849622} for the definitions of three terms in the above product.
Again noting that the centralizer group $C_{\bfG}^{(r_{i})}(\gamma)$ is equal to $\bfT_{\gamma}$ by the unramified very regularity of $\gamma$, hence equal to $\bfS$, we get
\[
\varepsilon_{\sym,\ram}(\pi',\gamma)
\cdot
\varepsilon^{\ram}(\pi',\gamma)
\cdot
\tilde{e}(\pi',\gamma)
=
\]
\[
\varepsilon_{\sym,\ram}(\bfG^{i+1}/\bfG^{i},r_{i},\gamma)
\cdot
\varepsilon^{\ram}(\bfG^{i+1}/\bfG^{i},r_{i},\gamma)
\cdot
\tilde{e}(\bfG^{i+1}/\bfG^{i},r_{i},\gamma),
\]
where
\[
\varepsilon_{\sym,\ram}(\bfG^{i+1}/\bfG^{i},r_{i},\gamma)
\colonequals 
\frac{\varepsilon_{\sym,\ram}(\bfG^{i+1},\bfS,r_{i},\gamma)}{\varepsilon_{\sym,\ram}(\bfG^{i},\bfS,r_{i},\gamma)}
\]
(similarly for $\varepsilon^{\ram}$ and $\tilde{e}$).

According to the description in \cite[Proposition 5.2.13]{MR2543925}, the invariant $\mathfrak{G}(\phi_i, \gamma)$ is equal to $1$ when $\gamma$ is unramified very regular because the sets $\dot{\Upsilon}_{\sym}(\phi_i, \gamma)$, $\dot{\Upsilon}_{\sym,\ram}(\phi_{i},\gamma)$ are empty (see \cite[Notation 5.2.11]{MR2543925}). 
On the other hand, by Remark \ref{rem:signs}, 
\begin{itemize}
\item
all three terms $\varepsilon_{\sym,\ram}(\bfG^{i+1}/\bfG^{i},r_{i},\gamma)$, $\varepsilon^{\ram}(\bfG^{i+1}/\bfG^{i},r_{i},\gamma)$, and $\tilde{e}(\bfG^{i+1}/\bfG^{i},r_{i},\gamma)$, are invariant under $G^{0}$-conjugation, and
\item
the first term $\varepsilon_{\sym,\ram}$ is trivial since $\bfS$ is unramified.
\end{itemize}

Thus we get 
\begin{align*}
\Theta_{\cc\rho_d' \otimes \phi_d}(\gamma)
&=
(-1)^{r(\bbG^{0}_{\x})-r(\bbS)}
\prod_{i=0}^{d-1}\tilde{e}(\bfG^{i+1}/\bfG^{i},r_{i},\gamma)
\cdot
\varepsilon^{\ram}[\phi](\gamma)
\sum_{w\in W_{G^{0}_{\x,0}}(\bfS)}
\phi({}^{w}\gamma)\\
&=
(-1)^{r(\bbG^{0}_{\x})-r(\bbS)}
\prod_{i=0}^{d-1}\tilde{e}(\bfG^{i+1}/\bfG^{i},r_{i},\gamma)
\sum_{w\in W_{G^{0}_{\x,0}}(\bfS)}
\varepsilon^{\ram}[\phi]({}^{w}\gamma)\phi({}^{w}\gamma)
\end{align*}
We finally investigate the product $\prod_{i=0}^{d-1}\tilde{e}(\bfG^{i+1}/\bfG^{i},r_{i},\gamma)$.
Each $\tilde{e}(\bfG^{i+1}/\bfG^{i},r_{i},\gamma)$ is given by the quotient $\tilde{e}(\bfG^{i+1},\bfS,r_{i},\gamma)/\tilde{e}(\bfG^{i},\bfS,r_{i},\gamma)$ and we have
\[
\tilde{e}(\bfG^{i+1},\bfS,r_{i},\gamma)
=
(-1)^{|\Gamma_{F}\backslash R^{\bfG^{i+1}}_{\x,(r_{i}-\ord_{\gamma})}|}
\]
and
\[
\tilde{e}(\bfG^{i},\bfS,r_{i},\gamma)
=
(-1)^{|\Gamma_{F}\backslash R^{\bfG^{i}}_{\x,(r_{i}-\ord_{\gamma})}|}.
\]
By the unramified very regularity of $\gamma$, for any root $\alpha\in R(\bfG^{i},\bfS)$, we have
\[
\alpha(\gamma)\neq1
\quad\text{and}\quad
\ord_{\gamma}{\alpha}=\ord(\alpha(\gamma)-1)=0.
\]
Thus we have
\begin{align*}
R^{\bfG^{i+1}}_{\x,(r_{i}-\ord_{\gamma})/2}
&\colonequals 
\{
\alpha\in R(\bfG^{i+1},\bfS)
\mid
\alpha(\gamma)\neq1,\, (r_{i}-\ord_{\gamma}{\alpha})/2\in\ord_{\x}\alpha
\}\\
&=
\{
\alpha\in R(\bfG^{i+1},\bfS)
\mid
r_{i}/2\in\ord_{\x}\alpha
\}\\
&=:
R^{\bfG^{i+1}}_{\x,r_{i}/2}.
\end{align*}
Similarly, we have $R^{\bfG^{i}}_{\x,(r_{i}-\ord_{\gamma})/2}=R^{\bfG^{i}}_{\x,r_{i}/2}$.

Therefore, using notation defined in the assertion, we get
\[
\Theta_{\cc\rho_d' \otimes \phi_d}(\gamma)
=
(-1)^{r(\bbG^{0}_{\x})-r(\bbS)+r(\bfS,\phi)}
\sum_{w\in W_{G^{0}_{\x,0}}(\bfS)}
\varepsilon^{\ram}[\phi]({}^{w}\gamma)\phi({}^{w}\gamma). \qedhere
\]
\end{proof}

\begin{prop}\label{prop:AS-vreg}
Let $\gamma$ be an unramified very regular element of $SG_{\x,0}$.
\begin{enumerate}
\item
If $\gamma$ is not $SG_{\x,0}$-conjugate to an element of $S$, then we have $\Theta_{\cc\tau_{d}}(\gamma)=0$.
\item
If $\gamma$ is $SG_{\x,0}$-conjugate to an element of $S$ (in this case, we assume that $\gamma$ itself belongs to $S$ by taking conjugation), we have
\[
\Theta_{\cc\tau_{d}}(\gamma)
=
(-1)^{r(\bbG^{0}_{\x})-r(\bbS)+r(\bfS,\phi)}
\sum_{w\in W_{G_{\x,0}}(\bfS)}
\varepsilon^{\ram}[\phi]({}^{w}\gamma)\phi({}^{w}\gamma),
\]
where $W_{G_{\x,0}}(\bfS)\colonequals N_{G_{\x,0}}(\bfS)/S_{0}$.
\end{enumerate}
\end{prop}

\begin{proof}
Let $\gamma$ be an unramified very regular element of $SG_{\x,0}$.
Then, since we have $\cc\tau_{d}\cong\Ind_{\cc K^{d}}^{SG_{\x,0}}(\cc\rho'_{d}\otimes\phi_{d})$ by definition, the Frobenius formula for induced representations implies that
\[
\Theta_{\cc\tau_{d}}(\gamma)
=
\sum_{\begin{subarray}{c}g\in \cc K^{d}\backslash SG_{\x,0}\\ {}^{g}\gamma\in\cc{K}^{d} \end{subarray}}
\Theta_{\cc\rho'_{d}\otimes\phi_{d}}({}^{g}\gamma).
\]
Hence, if $\gamma$ is not $SG_{\x,0}$-conjugate to an element of $\cc K^{d}$, then the character is equal to zero.
Moreover, by Proposition \ref{prop:rho_d'}, if ${}^{g}\gamma$ is not $\cc K^{d}$-conjugate to an element of $S$, then we have $\Theta_{\cc\rho'_{d}\otimes\phi_{d}}({}^{g}\gamma)=0$.
Therefore we get the assertion (1).

In the following, we assume that $\gamma$ itself belongs to $S$.
Then, again by the same argument as in the previous paragraph, we get
\[
\Theta_{\cc\tau_{d}}(\gamma)
=
\sum_{\begin{subarray}{c}g\in \cc K^{d}\backslash SG_{\x,0}\\ {}^{g}\gamma\in S\end{subarray}}
\Theta_{\cc\rho'_{d}\otimes\phi_{d}}({}^{g}\gamma),
\]
where the sum is over elements of $\cc K^{d}\backslash SG_{\x,0}$ containing a representative $g$ satisfying ${}^{g}\gamma\in S$.

We investigate the index set of this formula.
First, as $\cc{K}^{d}$ contains $S$, we may suppose that $g$ belongs to $G_{\x,0}$.
Since $\gamma$ is a regular semisimple element belonging to $S$, if $g\in G_{\x,0}$ satisfies ${}^{g}\gamma\in S$, then $g$ belongs to $N_{G_{\x,0}}(\bfS)$.
Conversely, any element $g$ of $N_{G_{\x,0}}(\bfS)$ satisfies ${}^{g}\gamma\in S$.
In other words, the index set can be rewritten as $\cc K^{d}\cap N_{G_{\x,0}}(\bfS) \backslash N_{G_{\x,0}}(\bfS)$.
In fact, we have 
\[
\cc K^{d}\cap N_{G_{\x,0}}(\bfS)\,
\bigl(\colonequals SG_{\x,0}^{0}(G^{0},\ldots,G^{d})_{\x,(0+,s_{0},\ldots,s_{d-1})}\cap N_{G_{\x,0}}(\bfS) \bigr)
=
N_{G^{0}_{\x,0}}(\bfS).
\]
Indeed, the inclusion $\cc K^{d}\cap N_{G_{\x,0}}(\bfS)\supset N_{G^{0}_{\x,0}}(\bfS)$ is trivial.
To check the converse inclusion, let us take an element $g$ of $\cc K^{d}\cap N_{G_{\x,0}}(\bfS)$ and show that $g$ belongs to $G^{0}_{\x,0}$.
As $g$ belongs to $\cc K^{d}$, we may write $g=g^{0}k$ with elements $g^{0}\in SG^{0}_{\x,0}$ and $k\in(G^{0},\ldots,G^{d})_{\x,(0+,s_{0},\ldots,s_{d-1})}$.
Since $g$ normalizes $S$, we have ${}^{k}S\subset G^{0}$.
Then, by using \cite[Lemma 9.10]{MR2431235} with $(\bfG',\bfG)\colonequals (\bfG^{d-1},\bfG^{d})$, we get $k\in G^{d-1}_{\x,0+}S_{0+}=G^{d-1}_{\x,0+}$.
Hence we have $k\in G^{d-1}_{\x,0+}\cap (G^{0},\ldots,G^{d})_{\x,(0+,s_{0},\ldots,s_{d-1})}=(G^{0},\ldots,G^{d-1})_{\x,(0+,s_{0},\ldots,s_{d-2})}$.
In a similar way, by using \cite[Lemma 9.10]{MR2431235} repeatedly with $(\bfG',\bfG)\colonequals (\bfG^{d-2},\bfG^{d-1}),\ldots, (\bfG^{0},\bfG^{1})$, we finally get $k\in G^{0}_{\x,0+}$.
Hence we get $g=g^{0}k\in SG^{0}_{\x,0}$.
As $g$ lies in $N_{G_{\x,0}}(\bfS)\subset G_{\x,0}$, we furthermore get $g\in G^{0}_{\x,0}$.

Now, by Proposition \ref{prop:rho_d'}, we have 
\begin{align*}
\Theta_{\cc\tau_{d}}(\gamma)
&=
\sum_{g\in N_{G^{0}_{\x,0}}(\bfS)\backslash N_{G_{\x,0}}(\bfS)}
(-1)^{r(\bbG^{0}_{\x})-r(\bbS)+r(\bfS,\phi)}
\sum_{w\in W_{G^{0}_{\x,0}}(\bfS)}
\varepsilon^{\ram}[\phi]({}^{wg}\gamma)\phi({}^{wg}\gamma)\\
&=
(-1)^{r(\bbG^{0}_{\x})-r(\bbS)+r(\bfS,\phi)}
\sum_{w\in W_{G_{\x,0}}(\bfS)}
\varepsilon^{\ram}[\phi]({}^{w}\gamma)\phi({}^{w}\gamma). \qedhere
\end{align*}
\end{proof}

\section{Parahoric representations characterized by $S_{\vreg}$} \label{sec:ell unram vreg}

Let $(\bfS,\phi)$ be an elliptic regular pair and let $(\vec{\bfG}, \pi_{-1}, \vec{\phi})$ be a regular generic reduced cuspidal $\bfG$-datum corresponding to $(\bfS, \phi)$ as in Section \ref{subsec:Kaletha-TER}. Recall from Section \ref{subsec:stab-parah} that from $(\vec{\bfG}, \pi_{-1}, \vec{\phi})$, Yu constructs various intermediate representations; in this section we will be especially interested in $\cc\rho_d'$ and $\cc\tau_d$.

In Section \ref{subsec:characters_vreg}, we worked with elliptic regular pairs $(\bfS, \phi)$ where $\bfS$ is unramified.
In this section (and in fact in the rest of the paper, excluding Section \ref{sec:geom}), we additionally assume:
\[
\text{$\phi$ is toral; equivalently, $\bfG^0(\bfS, \phi) = \bfS$.}
\]
In this case, for any $i$, the group $\cc K^{i}=SG_{\x,0}^{0}(G^{0},\ldots,G^{i})_{\x,(0+,s_{0},\ldots,s_{i-1})}$ defined in Section \ref{subsec:stab-parah} is equal to the (\textit{a priori} slightly larger) group $K^{i}=G_{\bar{\x}}^{0}(G^{0},\ldots,G^{i})_{\x,(0+,s_{0},\ldots,s_{i-1})}$, which furthermore equals $Z_{\bfG}S_{0}(G^{0},\ldots,G^{i})_{\x,(0+,s_{0},\ldots,s_{i-1})}$.
Accordingly, for any $i$, we have $\cc\rho'_{i}=\rho'_{i}$. 

\subsection{Unramified very regular elements}\label{subsec:vreg}

Let $S_{\vreg}$ denote the set of unramified very regular elements of $SG_{\x,0}$ contained in $S$.

\begin{lemma}\label{lem:sg+ vreg}
For any $s \in S_{\vreg}$ and any $g_+ \in G_{\x,0+}$, the product $\gamma \colonequals s \cdot g_+$ is unramified very regular and $\bfT_\gamma$ is $G_{\x,0+}$-conjugate to $\bfS$.
\end{lemma}

\begin{proof}
We first take a topological Jordan decomposition (or, equivalently, a normal $(0+)$-approximation) $s=s_{0}\cdot s_{+}$ with topologically semisimple part $s_{0}$ and topologically unipotent part $s_{+}$.
Then the product decomposition $\gamma=s_{0}\cdot (s_{+}g_{+})$ gives a $(0+)$-approximation to $\gamma$; more precisely, the pair $(\underline{\gamma},\x)$ of the good sequence $\underline{\gamma}=(\gamma_{i})_{0\leq i<0+}$ consisting of only one element $\gamma_{0}\colonequals s_{0}$ and the point $\x$ is a $(0+)$-approximation to $\gamma$. 
By \cite[Lemma 9.2]{MR2431235}, there exists an element $k \in G_{\x,0+}$ such that ${}^{k}\underline{\gamma}=({}^{k}\gamma_{i})_{0\leq i<0+}$ is a normal $(0+)$-approximation to $\gamma$. 
In other words, we have the following:
\begin{itemize}
\item
Since ${}^{k}\underline{\gamma}=({}^{k}\gamma_{i})_{0\leq i<0+}$ is a $(0+)$-approximation to $\gamma$, we have $\gamma\in{}^{k}\gamma_{0}G_{\y,0+}$ for some point $\y$ of $\mathcal{B}(C_{\bfG}^{(0+)}({}^{k}\underline{\gamma}),F)$.
Here note that $C_{\bfG}^{(0+)}({}^{k}\underline{\gamma})=C_{\bfG}({}^{k}\gamma_{0})^{\circ}={}^{k}\bfS$ by the regularity of $\gamma_{0}=s_{0}$, which follows from the unramified very regularity of $s$.
Thus, as $k\in G_{\x,0+}$ stabilizes $\x$, we have 
\[
\mathcal{B}(C_{\bfG}^{(0+)}({}^{k}\underline{\gamma}),F)
=\mathcal{A}({}^{k}\bfS,F)
=k\cdot\mathcal{A}(\bfS,F)
\ni k\cdot\x=\x.
\]
Since $\mathcal{A}^{\red}({}^{k}\bfS,F)$ consists of only one point, we have $\mathcal{A}^{\red}({}^{k}\bfS,F)=\{\bar{\x}\}=\{\bar{\y}\}$.
Let us write $\gamma={}^{k}\gamma_0\cdot\gamma_{+}$ with $\gamma_{+}\in G_{\y,0+}=G_{\x,0+}$.
\item
Since ${}^{k}\underline{\gamma}$ is normal, $\gamma$ lies in $C_{\bfG}^{(0+)}({}^{k}\underline{\gamma})={}^{k}\bfS$.
Hence so does $\gamma_{+}$.
\end{itemize}

We now check the unramified very regularity of $\gamma$ using the decomposition $\gamma = {}^k \gamma_{0} \cdot \gamma_+$.
Since $\gamma$ belongs to ${}^k \bfS$, $\gamma$ is semisimple. Moreover, $\gamma$ is a regular element satisfying $\alpha(\gamma) \not\equiv 1 \pmod \mfp$ for any $\alpha \in R({}^k \bfS, \bfG)$. Indeed, for any root $\alpha \in R({}^k \bfS, \bfG)$, we have
\[
\alpha(\gamma)
=\alpha({}^{k}\gamma_{0})\cdot\alpha(\gamma_{+}).
\]
While the unramified very regularity of $s$ implies that $\alpha({}^{k}\gamma_{0})\not\equiv 1 \pmod\mfp$, we have $\alpha(\gamma_{+}) \equiv 1 \pmod\mfp$ by the topological unipotency of $\gamma_{+}$.
In particular, we have $\alpha(\gamma)\not\equiv1\pmod\mfp$.
This also shows the regularity of $\gamma$ since the regularity is a weaker condition; $\alpha(\gamma)\neq1$ for any $\alpha \in R({}^{k}\bfS,\bfG)$.
Finally, since the connected centralizer $\bfT_{\gamma}$ is given by ${}^{k}\bfS$, we have $\mathcal{A}(\bfT_{\gamma},F)=\mathcal{A}({}^{k}\bfS,F)\ni\x$.
\end{proof}

\begin{lemma}\label{lem:SG+ vreg}
We have $S_{\vreg} G_{\x,0+} = \{\gamma \in SG_{\x,0+} \mid \text{$\gamma$ is unramified very regular}\}.$ Moreover, every unramified very regular element of $SG_{\x,0+}$ is $G_{\x,0+}$-conjugate to an element of $S_{\vreg}$.
\end{lemma}

\begin{proof}
By Lemma \ref{lem:sg+ vreg}, every element of $S_{\vreg} G_{\x,0+}$ is unramified very regular. To see the reverse inclusion, let $\gamma \in SG_{\x,0+}$ be unramified very regular. We may write $\gamma = s \cdot g_+$ for some $s \in S$ and $g_+ \in G_{\x,0+}$. By Lemma \ref{lem:sg+ vreg}, $s \in S_{\vreg}$, so $\gamma \in S_{\vreg} G_{\x,0+}$. 

Since every element of $S_{\vreg}G_{\x,0+}$ is $G_{\x,0+}$-conjugate to an element of $S_{\vreg}$ by Lemma \ref{lem:sg+ vreg}, the final assertion now holds.
\end{proof}

We put $S_{0,\vreg}:=S_{\vreg}\cap S_{0}$.
Note that, by the definition of unramified very regular elements, we can easily see that $S_{\vreg}=Z_{\bfG}S_{0,\vreg}$ (recall that $S=Z_{\bfG}S_{0}$).

\begin{defn}\label{defn:finite-vreg}
We define a subset $\bbS(\F_{q})_{\vreg}$ of $\bbS(\F_{q})$ to be the image of $S_{0,\vreg}$ under the reduction map $S_{0}\twoheadrightarrow S_{0:0+}\cong\bbS(\F_{q})$.
Let $\bbS(\F_{q})_{\nvreg}$ denote its complement in $\bbS(\F_{q})$, i.e., $\bbS(\F_{q})_{\nvreg}\colonequals \bbS(\F_{q})\smallsetminus\bbS(\F_{q})_{\vreg}$.
\end{defn}

\begin{lem}\label{lem:greater-than-2}
When $|\bbS(\F_{q})|/|\bbS(\F_{q})_{\nvreg}|>2$, for any element $s\in\bbS(\F_{q})$, there exist elements $t_{1},t_{2}\in\bbS(\F_{q})_{\vreg}$ such that $s=t_{1}t_{2}$.
\end{lem}

\begin{proof}
For a given $s\in\bbS(\F_{q})$, we consider a subset $s\cdot\bbS(\F_{q})_{\vreg}$ of $\bbS(\F_{q})$.
If we have $|s\cdot\bbS(\F_{q})_{\vreg}|>|\bbS(\F_{q})_{\nvreg}|$, then $s\cdot\bbS(\F_{q})_{\vreg}$ is not contained in $\bbS(\F_{q})_{\nvreg}$.
In other words, $s\cdot\bbS(\F_{q})_{\vreg}$ intersects $\bbS(\F_{q})_{\vreg}$.
Thus there exist elements $t_{1},t_{2}\in\bbS(\F_{q})_{\vreg}$ such that $st_{1}=t_{2}$.
By noting that the set $S_{0,\vreg}$ is stable under the inversion, $t_{1}^{-1}$ lies in $\bbS(\F_{q})_{\vreg}$ and we get the assertion.
The inequality $|s\cdot\bbS(\F_{q})_{\vreg}|>|\bbS(\F_{q})_{\nvreg}|$ is equivalent to $|\bbS(\F_{q})|/|\bbS(\F_{q})_{\nvreg}|>2$.
\end{proof}

\begin{cor}\label{cor:greater-than-2}
When $|\bbS(\F_{q})|/|\bbS(\F_{q})_{\nvreg}|>2$, for any element $s\in S$, there exist elements $t_{1},t_{2}\in S_{\vreg}$ such that $s=t_{1}t_{2}$. 
In particular, $S_{0,\vreg}$ generates $S_{0}$ as a group.
\end{cor}

\begin{proof}
Let $s$ be an element of $S_{0}$.
By the assumption and Lemma \ref{lem:greater-than-2}, we can take elements $t_{1}$ and $t_{2}$ of $S_{0,\vreg}$ such that $t_{1}t_{2}$ and $s$ have the same image in $S_{0:0+}\cong\bbS(\F_{q})$.
In other words, there exists an element $t_{+}\in S_{0+}$ satisfying $t_{1}t_{2}t_{+}=s$.
By Lemma \ref{lem:sg+ vreg}, $t_{2}t_{+}$ is an unramified very regular element of $S$.
This shows that $s$ can be written as a product of two elements of $S_{0,\vreg}$.
\end{proof}

Let us show that the inequality
\begin{align*}\label{ineq}
\frac{|\bbS(\F_{q})|}{|\bbS(\F_{q})_{\nvreg}|}>2
\tag{$\star$}
\end{align*}
is satisfied when $q$ is sufficiently large.

\begin{lem}\label{lem:transfer}
The unramified elliptic maximal torus $\bfS$ of $\bfG$ transfers to an unramified elliptic maximal torus $\bfS^{\ast}$ of the quasi-split inner form $\bfG^{\ast}$ of $\bfG$ such that the associated point $\x^{\ast}$ of the building $\mcB^{(\red)}(\bfG^{\ast},F)$ corresponds to a Chevalley valuation of $\bfG^{\ast}$.
\end{lem}

\begin{proof}
The precise meaning of the ``transfer'' is as follows: there exists an inner twist $\psi\colon\bfG\rightarrow\bfG^{\ast}$ such that its restriction to $\bfS$ induces an isomorphism $\psi|_{\bfS}\colon\bfS\rightarrow\bfS^{\ast}$ defined over $F$.
The existence of a transfer $\bfS^{\ast}$ of $\bfS$ is guaranteed by the ellipticity of $\bfS$ or the quasi-splitness of $\bfG^{\ast}$, see, for example, \cite[Section 3.2]{MR4013740}.
Note that $\bfS^{\ast}$ is unramified and elliptic since $\psi|_{\bfS}$ is defined over $F$ and maps $\bfZ_{\bfG}$ to $\bfZ_{\bfG^{\ast}}$.

Hence it suffices to show that such an $\bfS^{\ast}$ can be taken so that the associated point $\x^{\ast}\in\mcB^{(\red)}(\bfG^{\ast},F)$ corresponds to a Chevalley valuation.
For this, we utilize the results of Kaletha in \cite[Section 3.4.1]{MR4013740}.
Since $\bfG^{\ast}$ is quasi-split, there exists a point $\x^{\ast}_{1}\in\mcB^{(\red)}(\bfG^{\ast},F)$ which corresponds to a Chevalley valuation; such a point $\x^{\ast}_{1}$ is superspecial in the sense of Kaletha (see \cite[Remark 3.4.9]{MR4013740}).
By applying \cite[Lemma 3.4.12 (1)]{MR4013740} to $\x^{\ast}_{1}$ and $\bfS^{\ast}\subset\bfG^{\ast}$, we can find a maximal torus $\bfS_{1}^{\ast}$ of $\bfG^{\ast}$ stably conjugate to $\bfS^{\ast}$ with associated point $\x^{\ast}_{1}$.
By replacing $\bfS^{\ast}$ with $\bfS^{\ast}_{1}$ (hence $\x^{\ast}$ with $\x^{\ast}_{1}$), we obtain the desired assertion.
\end{proof}

\begin{lem}\label{lem:transfer2}
Let $\psi\colon\bfG\rightarrow\bfG^{\ast}$ be an inner twist whose restriction to $\bfS$ gives an isomorphism $\bfS\cong\bfS^{\ast}$ defined over $F$.
Then, under the isomorphism $\bbS(\F_{q})\rightarrow\bbS^{\ast}(\F_{q})$ induced by $\psi$, $\bbS(\F_{q})_{\vreg}$ is identified with $\bbS^{\ast}(\F_{q})_{\vreg}$.
\end{lem}

\begin{proof}
As $\psi|_{\bfS}$ is defined over $F$, we have identifications $S\cong S^{\ast}$ and $\bbS(\F_{q})\cong\bbS^{\ast}(\F_{q})$ which are consistent with reduction morphisms:
\[
\xymatrix{
S_{0} \ar@{->>}[d] \ar^-{\psi}_-{\cong}[r]& S^{\ast}_{0} \ar@{->>}[d]\\
\bbS(\F_{q}) \ar^-{\psi}_-{\cong}[r]& \bbS^{\ast}(\F_{q})
}
\]
Since $\psi$ induces a bijection $R(\bfS,\bfG)\rightarrow R(\bfS^{\ast},\bfG^{\ast})\colon \alpha\mapsto\alpha\circ\psi|_{\bfS}^{-1}$, $\psi$ gives an identification of $S_{\vreg}$ with $S^{\ast}_{\vreg}$, and also $S_{0,\vreg}$ with $S^{\ast}_{0,\vreg}$, by the definition of unramified very regularity.
As $\bbS(\F_{q})_{\vreg}$ (resp.\ $\bbS^{\ast}(\F_{q})_{\vreg}$) is defined to be the image of $S_{0,\vreg}$ (resp.\ $S^{\ast}_{0,\vreg}$) under the reduction morphism, $\psi$ gives an identification of $\bbS(\F_{q})_{\vreg}$ with $\bbS^{\ast}(\F_{q})_{\vreg}$.
\end{proof}

\begin{prop}\label{prop:q suff large}
The inequality (\ref{ineq}) is satisfied when $q$ is sufficiently large.
\end{prop}

\begin{proof}
We take an inner twist $\psi\colon\bfG\rightarrow\bfG^{\ast}$ transferring $\bfS$ to $\bfS^{\ast}$ as in Lemma \ref{lem:transfer}.
Then, by Lemma \ref{lem:transfer2}, we have 
\[
\frac{|\bbS(\F_{q})|}{|\bbS(\F_{q})_{\nvreg}|}
=
\frac{|\bbS^{\ast}(\F_{q})|}{|\bbS^{\ast}(\F_{q})_{\nvreg}|}.
\]
Hence it is enough to show the assertion for $\bfS^{\ast}\subset\bfG^{\ast}$ whose associated point $\x^{\ast}\in\mathcal{B}^{(\red)}(\bfG^{\ast},F)$ corresponds to a Chevalley valuation.

Let $\bbG^{\ast}_{\x^{\ast}}$ be the reductive quotient of the special fiber of the parahoric subgroup scheme of $\bfG^{\ast}$ with respect to $\x^{\ast}$, and regard $\bbS^{\ast}$ as a maximal torus of 
 $\bbG^{\ast}_{\x^{\ast}}$.
Since the point $\x^{\ast}$ corresponds to a Chevalley valuation, the reduction map induces a bijection from $R(\bfS^{\ast},\bfG^{\ast})$ to $R(\bbS^{\ast},\bbG^{\ast}_{\x^{\ast}})$.
This implies that an element $\gamma\in S^{\ast}_{0}$ is unramified very regular if and only if its reduction $\overline{\gamma}\in\bbS^{\ast}(\F_{q})$ is regular semisimple in $\bbG^{\ast}_{\x^{\ast}}$.
Therefore the set $\bbS^{\ast}(\F_{q})_{\vreg}$ is nothing but the set $\bbS^{\ast}_{\reg}(\F_{q})$ of $\F_{q}$-valued points of regular semisimple locus $\bbS^{\ast}_{\reg}$ of $\bbS^{\ast}$ in $\bbG^{\ast}_{\x^{\ast}}$.
Similarly, $\bbS^{\ast}(\F_{q})_{\nvreg}$ ($\colonequals\bbS^{\ast}(\F_{q})\smallsetminus\bbS^{\ast}(\F_{q})_{\vreg}$) is nothing but $\bbS_{\nreg}^*(\F_{q})$ ($\colonequals\bbS^{\ast}(\F_{q})\smallsetminus\bbS^{\ast}_{\reg}(\F_{q})$).

Now the statement follows from a well-known fact of finite reductive groups.
For example, by \cite[Lemma 2.3.11]{MR4211779}, there exists a positive number $C>0$ depending only on the root datum of $\bbG^{\ast}_{\x^{\ast}}$ such that we have
\[
\frac{|\bbS^{\ast}_{\reg}(\F_{q})|}{|\bbS^{\ast}(\F_{q})|}\geq1-\frac{C}{q},
\]
or equivalently,
\[
\frac{|\bbS^{\ast}(\F_{q})|}{|\bbS^{\ast}_{\nreg}(\F_{q})|}\geq\frac{q}{C}.
\]
Thus the inequality $(\star)$ is satisfied whenever $q>2C$.
\end{proof}

\begin{remark}
The subtlety handled the above lemmas culminating in Proposition \ref{prop:q suff large} is exactly that in general $\bbS(\FF_q)_{\vreg} \subsetneq \bbS_{\reg}(\FF_q)$. For example, in the extreme case that the parahoric $G_{\x,0}$ is Iwahori, $\bbS_{\reg}(\FF_q) = \bbS(\FF_q)$.
\end{remark}

\begin{remark}\label{rem:q >> 0}
There is a natural question of how large $q$ must be in order for \eqref{ineq} to be satisfied. At least for Coxeter tori, this bound is very mild (forthcoming work). For example, if $\bfS$ is the Coxeter torus of $\bfG = G_2$, then $|\bbS(\FF_q)_{\nreg}|$ is either $1$ or $3$, depending on $q$, and $|\bbS(\FF_q)| \geq 7$ for $q > 3$. Hence only $q=2$ does not satisfy \eqref{ineq}, so in fact \eqref{ineq} is a \textit{weaker} condition than the conditions on $p$ required by the theory of Kaletha and Yu reviewed in Section \ref{sec:Yu}. The $\GL_n$ case (which was known already to Henniart \cite{MR1235293}) is explained in Section \ref{sec:GLn}.
\end{remark}

\subsection{Representations of $SG_{\x,0+}$ characterized by their trace on $S_{\vreg}$}\label{subsec:SG+}

In this section, we prove (Theorem \ref{thm:SG+ vreg}) that some irreducible virtual representations of $SG_{\x,0+}$ are characterized by character values on unramified very regular elements of $SG_{\x,0}$.
In this subsection, we will additionally assume that the inequality (\ref{ineq}) (introduced in Section \ref{subsec:vreg}) holds.
Note that then, in particular, we have $S_{\vreg}\neq\varnothing$.

This section pertains to smooth irreducible representations $\pi_+$ of $SG_{\x,0+}$ such that for a nonzero constant $c \in \bbC$,
\begin{equation*}\label{e:char vreg +}
\Theta_{\pi_+}(\gamma) = c \cdot \theta(\gamma), \qquad \text{for all $\gamma \in S_{\vreg}$}
\tag{$\dagger_{+}$}
\end{equation*}
for a character $\theta$ of $S$ with $\bfG^0(\bfS, \theta) = \bfS$. 
This determines the $\Theta_{\pi_+}$ on the unramified very regular elements of $SG_{\x,0+}$ since every unramified regular element of $SG_{\x,0+}$ is $G_{\x,0+}$-conjugate to an element of $S_{\vreg}$ by Lemma \ref{lem:SG+ vreg}.

Note that we have the following chain of open compact subgroups:
\begin{align*}
S_{0}G_{\x,0+}
&\supset S_{0}G_{\x,0+}\cap K^{d}=S_{0}(G^{0},\ldots,G^{i})_{\x,(0+,s_{0},\ldots,s_{i-1})}\\
&\supset G_{\x,0+}\cap K^{d}=(G^{0},\ldots,G^{i})_{\x,(0+,s_{0},\ldots,s_{i-1})}
\end{align*}

\begin{lemma}\label{lem:hom G+ +}
Let ${\pi_+}$ be a smooth irreducible representation of $SG_{\x,0+}$ satisfying \eqref{e:char vreg +} for the character $\theta = \phi\cdot\varepsilon^{\ram}[\phi]$ of $S$.
Then 
\begin{equation*}
\Hom_{G_{\x,0+}\cap K^{d}}(\rho_d' \otimes \phi_d, \pi_+) \neq0.
\end{equation*}
\end{lemma}

\begin{proof}
Because of the smoothness assumption, the action of $S_0G_{\x,0+}$ on ${\pi_+}$ factors through a finite quotient $S_0G_{\x,0+}/G_{\x,r}$ for sufficiently large $r \in \bbR_{>0}$. All sums in the rest of the proof can be and are interpreted as taking place in this finite quotient.

We first prove that
\begin{equation}\label{e:vreg nonzero +}
\sum_{\gamma \in S_{0,\vreg} G_{\x,0+}\cap K^{d}} \Theta_{\rho'_{d} \otimes \phi_d}(\gamma) \cdot \overline{\Theta_{\pi_+}(\gamma)} \neq 0.
\end{equation}
By Proposition \ref{prop:rho_d'}, $\Theta_{\rho'_{d} \otimes \phi_d}(\gamma) = 0$ if $\gamma \in S_{0,\vreg} G_{\x,0+}\cap K^{d}$ is not $K^d$-conjugate to an element of $S$. 
Thus we get 
\[
\sum_{\gamma \in S_{0,\vreg} G_{\x,0+}\cap K^{d}} \Theta_{\rho'_{d} \otimes \phi_d}(\gamma) \cdot \overline{\Theta_{\pi_+}(\gamma)}
=
\sum_{\begin{subarray}{c}\gamma\in S_{0,\vreg}G_{\x,0+}\cap K^{d} \\ \text{${}^{k}\gamma\in S$ for some $k\in K^{d}$}\end{subarray}} \Theta_{\rho'_{d} \otimes \phi_d}(\gamma) \cdot \overline{\Theta_{\pi_+}(\gamma)}.
\]

Let us investigate the index set of the sum on the right-hand side.
If $\gamma\in S_{0,\vreg} G_{\x,0+}\cap K^{d}$ is $K^d$-conjugate to an element of $S$ (say ${}^{k}\gamma\in S$ with $k\in K^{d}$), then ${}^{k}\gamma$ belongs to $S_{0,\vreg}$.
Conversely, for any element $\gamma\in S_{0,\vreg}$ and $k\in K^{d}$, ${}^{k}\gamma$ always belongs to $S_{0,\vreg}G_{\x,0+}\cap K^{d}$.
In other words, we have a surjective map
\[
S_{0,\vreg}\times K^{d}
\twoheadrightarrow
\{\gamma\in S_{0,\vreg}G_{\x,0+}\cap K^{d} \mid \text{${}^{k}\gamma\in S$ for some $k\in K^{d}$}\}
\]
given by $(\gamma',k)\mapsto {}^{k}\gamma'$.
Suppose that the images of $(\gamma_{1},k_{1})\in S_{0,\vreg}\times K^{d}$ and $(\gamma_{2},k_{2})\in S_{0,\vreg}\times K^{d}$ under this map coincide, i.e., ${}^{k_{1}}\gamma_{1}={}^{k_{2}}\gamma_{2}$.
Then, by taking the connected centralizer groups of ${}^{k_{1}}\gamma_{1}$ and ${}^{k_{2}}\gamma_{2}$ in $\bfG$, we get ${}^{k_{1}}\bfS={}^{k_{2}}\bfS$.
Recall that we have $K^{d}\cap N_{G_{\x,0}}(\bfS)=S_{0}$ as proved in the proof of Proposition \ref{prop:AS-vreg} (note that here we are supposing that $\bfG^{0}=\bfS$), hence we have $N_{K^{d}}(\bfS)=S$ (note $K^{d}\subset Z_{\bfG}G_{\x,0}$).
Hence we get $k_{1}^{-1}k_{2}\in S$.
In other words, the map $(\gamma',k)\mapsto {}^{k}\gamma'$ induces a bijection
\[
S_{0,\vreg}\times (K^{d}/S)
\xrightarrow{1:1}
\{\gamma\in S_{0,\vreg}G_{\x,0+}\cap K^{d} \mid \text{${}^{k}\gamma\in S$ for some $k\in K^{d}$}\}.
\]

By this observation, we get
\[
\sum_{\begin{subarray}{c}\gamma\in S_{0,\vreg}G_{\x,0+}\cap K^{d} \\ \text{${}^{k}\gamma\in S$ for some $k\in K^{d}$}\end{subarray}} \Theta_{\rho'_{d} \otimes \phi_d}(\gamma) \cdot \overline{\Theta_{\pi_+}(\gamma)}
=
\sum_{\gamma'\in S_{0,\vreg}}
\sum_{k\in K^{d}/S}
\Theta_{\rho'_{d} \otimes \phi_d}({}^{k}\gamma') \cdot \overline{\Theta_{\pi_+}({}^{k}\gamma')}.
\]
Since both of $\Theta_{\rho'_{d}\otimes\phi_{d}}$ and $\overline{\Theta_{\pi_+}}$ are invariant under $K^{d}$-conjugation (note that $K^{d}$ is contained in $SG_{\x,0+}$, where $\pi_{+}$ is defined), we get
\[
\sum_{\gamma'\in S_{0,\vreg}}
\sum_{k\in K^{d}/S}
\Theta_{\rho'_{d} \otimes \phi_d}({}^{k}\gamma') \cdot \overline{\Theta_{\pi_+}({}^{k}\gamma')}
=
|K^{d}/S|
\sum_{\gamma'\in S_{0,\vreg}}
\Theta_{\rho'_{d} \otimes \phi_d}(\gamma') \cdot \overline{\Theta_{\pi_+}(\gamma')}.
\]
Here, as explained in the beginning of this proof, precisely the symbol $|K^{d}/S|$ denotes the index of the subgroup $SG_{\x,r}$ in $K^{d}$.

Therefore, to get (\ref{e:vreg nonzero +}), is suffices to check the above sum over $\gamma'\in S_{0,\vreg}$ is not zero.
By Proposition \ref{prop:rho_d'} and our assumption on $\pi_{+}$, for any $\gamma'\in S_{0,\vreg}$, we have
\[
\Theta_{\rho'_{d} \otimes \phi_d}(\gamma') \cdot \overline{\Theta_{\pi_+}(\gamma')}
=
(-1)^{r(\bfS,\phi)}\cdot \bar{c} \cdot \theta(\gamma')\overline{\theta(\gamma')}.
\]
Thus we get
\[
\sum_{\gamma'\in S_{0,\vreg}}
\Theta_{\rho'_{d} \otimes \phi_d}(\gamma') \cdot \overline{\Theta_{\pi_+}(\gamma')}.
=
(-1)^{r(\bfS,\phi)}\bar{c}
\sum_{\gamma'\in S_{0,\vreg}} \theta(\gamma')\overline{\theta(\gamma')}.
\]
Since $\theta(\gamma')\overline{\theta(\gamma')}$ is a positive number, this sum is not zero (recall that $S_{0,\vreg}\neq\varnothing$).

We next show
\[
\sum_{\gamma \in G_{\x,0+}\cap K^{d}} \Theta_{\rho'_{d} \otimes \phi_d}(\gamma) \cdot \overline{\Theta_{\pi_+}(\gamma)} \neq 0
\]
by using (\ref{e:vreg nonzero +}) we proved just now.
For any subset $A \subset R(\bfS, \bfG)$, we consider the subgroup 
\[
S_A \colonequals \{\delta \in S_0 \mid \text{$\alpha(\delta) \equiv 1 \pmod{\mfp}$ for all $\alpha \in A$}\}
\]
of $S$.
Observe that $S_{0}=S_{\varnothing}$ and that $S_{0,\vreg} = S_0 \smallsetminus \bigcup_{\varnothing \neq A \subset R(\bfS, \bfG)} S_{A}$. 
By the principle of inclusion and exclusion, the left-hand side of \eqref{e:vreg nonzero +} is equal to
\begin{align*}
\sum_{A \subset R(\bfS, \bfG)} (-1)^{|A|} \sum_{\gamma \in S_A G_{\x,0+}\cap K^{d}} \Theta_{\rho'_{d} \otimes \phi_d}(\gamma) \cdot \overline{\Theta_{\pi_+}(\gamma)}.
\end{align*}
It follows that the sum $\sum_{\gamma \in S_A G_{\x,0+}\cap K^{d}} \Theta_{\rho'_{d} \otimes \phi_d}(\gamma) \cdot \overline{\Theta_{\pi_+}(\gamma)}$ is not zero for some $A \subset R(\bfS, \bfG)$.
In other words, 
$\dim \Hom_{S_A G_{\x,0+}\cap K^{d}}(\rho_d' \otimes \phi_d, \pi_+) \neq 0$ for some subset $A \subset R(\bfS, \bfG)$. Since $S_A G_{\x,0+} \supset G_{\x,0+}$, the desired conclusion follows.
\end{proof}

\begin{theorem}\label{thm:SG+ vreg}
Let $\pi_+$ be a smooth irreducible representation of $SG_{\x,0+}$ satisfying \eqref{e:char vreg +} for the character $\theta = \phi\cdot\varepsilon^{\ram}[\phi]$ of $S$.
Then $c = (-1)^{r(\bfS, \phi)}$ and $\pi_+ \cong \Ind_{K^{d}}^{SG_{\x,0+}}(\rho_d' \otimes \phi_d)$.
\end{theorem}

\begin{proof}
We first note that the central character of $\Ind_{K^{d}}^{SG_{\x,0+}}(\rho_d' \otimes \phi_d)$ is given by $\phi|_{Z_{\bfG}}$ by \cite[Fact 3.7.11]{MR4013740}.
As $\varepsilon^{\ram}[\phi]|_{Z_{\bfG}}$ is trivial by definition (Definition \ref{def:varepsilon ram}), we have $\theta|_{Z_{\bfG}}=\phi|_{Z_{\bfG}}$.
On the other hand, also the central character of $\pi_{+}$ is given by $\theta|_{Z_\bfG}$.
Indeed, if we take an unramified very regular element $\gamma\in S_{\vreg}$ ($S_{\vreg}\neq\varnothing$ by the assumption), then $z\gamma$ is unramified very regular for any $z\in Z_{\bfG}$.
Thus the condition \eqref{e:char vreg +} implies that 
\[
\Theta_{\pi_{+}}(z\gamma)
=
c\cdot\theta(z\gamma)
=
c\cdot\theta(z)\cdot\theta(\gamma)
=
\theta(z)\cdot\Theta_{\pi_{+}}(\gamma).
\]
Therefore, to get the assertion, it is enough to show that $c = (-1)^{r(\bfS, \phi)}$ and $\pi_{+}|_{S_{0}G_{\x,0+}}\cong\Ind_{K^{d}}^{SG_{\x,0+}}(\rho_d' \otimes \phi_d)|_{S_{0}G_{\x,0+}}$ since we have $SG_{\x,0+}=Z_{\bfG}S_{0}G_{\x,0+}$.
Also note that $\Ind_{K^{d}}^{SG_{\x,0+}}(\rho_d' \otimes \phi_d)|_{S_{0}G_{\x,0+}}$ is nothing but $\Ind_{S_{0}G_{\x,0+}\cap K^{d}}^{S_{0}G_{\x,0+}}(\rho_d' \otimes \phi_d)$.

By Lemma \ref{lem:hom G+ +} and Frobenius reciprocity for $G_{\x,0+}\cap K^{d}\subset S_0 G_{\x,0+}\cap K^{d}$, we have
\begin{align}\label{eq:hom SG+ +1}
\Hom_{S_0 G_{\x,0+}\cap K^{d}}\bigl(\Ind_{G_{\x,0+}\cap K^{d}}^{S_0 G_{\x,0+}\cap K^{d}}(\rho_d' \otimes \phi_d), \pi_+\bigr) \neq0.
\end{align}
Since $\rho_d' \otimes \phi_d$ is a representation of $K^{d}$, which contains $S_{0}G_{\x,0+}\cap K^{d}$, the projection formula gives
\[
\Ind_{G_{\x,0+}\cap K^{d}}^{S_0 G_{\x,0+}\cap K^{d}}(\rho_d' \otimes \phi_d)
\cong
(\rho_d' \otimes \phi_d)
\otimes
\bigl(\Ind_{G_{\x,0+}\cap K^{d}}^{S_0 G_{\x,0+}\cap K^{d}}\mathbbm{1}\bigr).
\]
As we have $(S_0G_{\x,0+}\cap K^{d})/(G_{\x,0+}\cap K^{d})\cong S_{0:0+}$, we have 
\[
\Ind_{G_{\x,0+}\cap K^{d}}^{S_0 G_{\x,0+}\cap K^{d}}\mathbbm{1}
\cong
\bigoplus_{\theta_0 \from S_{0:0+} \to \bbC^\times} \theta_{0},
\]
hence
\[
\Ind_{G_{\x,0+}\cap K^{d}}^{S_0 G_{\x,0+}\cap K^{d}}(\rho_d' \otimes \phi_d)
\cong
\bigoplus_{\theta_0 \from S_{0:0+} \to \bbC^\times}(\rho_d' \otimes \phi_d)\otimes\theta_{0}.
\]
Therefore (\ref{eq:hom SG+ +1}) implies that we have
\[
\Hom_{S_0 G_{\x,0+}\cap K^{d}}\bigl((\rho_d'\otimes\phi_d)\otimes\theta_{0}, \pi_+\bigr) \neq0
\]
for at least one $\theta_0 \from S_{0:0+} \to \bbC^\times$.

Next, by Frobenius reciprocity for $S_0 G_{\x,0+}\cap K^{d} \subset S_{0}G_{\x,0+}$, we get
\begin{align}\label{eq:hom SG+ +2}
\Hom_{S_{0}G_{\x,0+}}\bigl(\Ind_{S_0 G_{\x,0+}\cap K^d}^{S_0 G_{\x,0+}}((\rho_d'\otimes\phi_d)\otimes\theta_{0}), \pi_+\bigr) \neq0.
\end{align}
Note that we have $S_{0}G_{\x,0+}/G_{\x,0+} \cong(S_0G_{\x,0+}\cap K^{d})/(G_{\x,0+}\cap K^{d})\cong S_{0:0+}$.
In particular, $\theta_{0}$ can be regarded as a character of $S_{0}G_{\x,0+}$ (which is trivial on $G_{\x,0+}$).
Thus the projection formula gives 
\[
\Ind_{S_0 G_{\x,0+}\cap K^d}^{S_0 G_{\x,0+}}((\rho_d'\otimes\phi_d)\otimes\theta_{0})
\cong
\bigl(\Ind_{S_0 G_{\x,0+}\cap K^d}^{S_0 G_{\x,0+}}(\rho_d'\otimes\phi_d)\bigr)\otimes\theta_{0}.
\]
Since the compact induction of $\rho_d'\otimes\phi_d$ from to $K^{d}$ to $G$ is irreducible by Yu's theory, so is the induced representation $\Ind_{K^d}^{SG_{\x,0+}}(\rho_d'\otimes\phi_d)$.
By noting that $K^{d}=Z_{\bfG}\cdot (S_0 G_{\x,0+}\cap K^d)$ and $SG_{\x,0+}=Z_{\bfG}\cdot S_{0}G_{\x,0+}$, $\Ind_{S_0 G_{\x,0+}\cap K^d}^{S_0 G_{\x,0+}}(\rho_d'\otimes\phi_d)$ is nothing but the restriction of $\Ind_{K^d}^{SG_{\x,0+}}(\rho_d'\otimes\phi_d)$ to $S_{0}G_{\x,0+}$ and irreducible.
On the other hand, $\pi_{+}$ is also irreducible as a representation of $S_{0}G_{\x,0+}$.
Thus, by (\ref{eq:hom SG+ +2}), we get
\[
\pi_{+}
\cong
\bigl(\Ind_{S_0 G_{\x,0+}\cap K^d}^{S_0 G_{\x,0+}}(\rho_d'\otimes\phi_d)\bigr)\otimes\theta_{0}.
\]

Now our task is to show that $\theta_{0}=\mathbbm{1}$.
By a similar, but simpler, argument to the proof of Proposition \ref{prop:AS-vreg}, for all $\gamma \in S_{0,\vreg}$ we have
\begin{align*}
\Theta_{(\Ind_{S_0 G_{\x,0+}\cap K^d}^{S_0 G_{\x,0+}}(\rho_d'\otimes\phi_d))\otimes\theta_{0}}(\gamma)
&=
\Theta_{\Ind_{S_0 G_{\x,0+}\cap K^d}^{S_0 G_{\x,0+}}(\rho_d'\otimes\phi_d)}(\gamma)\cdot\theta_{0}(\gamma)\\
&=
(-1)^{r(\bfS, \phi)}\theta(\gamma)\cdot\theta_0(\gamma).
\end{align*}
Thus, by the assumption on $\pi_{+}$, for all $\gamma \in S_{0,\vreg}$ we have
\[
c\cdot\theta(\gamma)
=
(-1)^{r(\bfS, \phi)}\theta(\gamma)\cdot\theta_0(\gamma).
\]
Hence $\theta_0$ must satisfy $\theta_0|_{S_{0,\vreg}} = c \cdot (-1)^{r(\bfS,\phi)}$. 
By Corollary \ref{cor:greater-than-2}, we can find elements $\gamma_{1}, \gamma_{2}\in S_{0,\vreg}$ such that the product $\gamma_{1}\gamma_{2}\in S_{0}$ is again unramified very regular.
Then we have
\[
\theta_0(\gamma_{1})
=
\theta_0(\gamma_{2})
=
\theta_0(\gamma_{1}\gamma_{2})
= 
c\cdot (-1)^{r(\bfS,\phi)}.
\]
By noting that $\theta_{0}(\gamma_{1}\gamma_{2})=\theta_{0}(\gamma_{1})\theta_{0}(\gamma_{2})$, we see that $c\cdot(-1)^{r(\bfS,\phi)}=1$.
Then the equality $\theta_{0}|_{S_{0,\vreg}}=\mathbbm{1}$ implies $\theta_{0}=\mathbbm{1}$ since $S_{0,\vreg}$ generates $S_{0}$ by Corollary \ref{cor:greater-than-2}.
\end{proof}

\begin{cor}\label{cor:SG+ vreg ind}
Let $\pi_+$ be a smooth irreducible representation of $SG_{\x,0+}$ satisfying all the conditions in Theorem \ref{thm:SG+ vreg}. Then $\Ind_{SG_{\x,0+}}^{SG_{\x,0}}(\pi_+)$ is an irreducible representation of $SG_{\x,0}$.
\end{cor}

\begin{proof}
This follows from Theorem \ref{thm:SG+ vreg} together with the fact that $\Ind_{K^{d}}^{SG_{\x,0}}(\rho_d' \otimes \phi_d)$ is irreducible.
\end{proof}

\subsection{Representations of $SG_{\x,0}$ characterized by their trace on $S_{\vreg}$}\label{subsec:SG}

In this section, we prove (Theorem \ref{thm:SG vreg}) that some irreducible representations of $SG_{\x,0}$ are characterized by character values on unramified very regular elements of $SG_{\x,0}$.
As in the previous subsection, we will assume here that the inequality (\ref{ineq}) holds.
The reader should think of this section in parallel to Section \ref{subsec:SG+}, though we will need an additional argument to establish the analogues of Lemma \ref{lem:hom G+ +} and Theorem \ref{thm:SG+ vreg}. 

In this section, we will consider smooth irreducible representations $\pi$ of $SG_{\x,0}$ such that for some nonzero constant $c \in \bbC$,
\begin{equation*}\label{e:char vreg}
\Theta_\pi(\gamma) 
=
c\sum_{w \in W_{G_{\x,0}}(\bfS)} \theta^w(\gamma), \qquad \text{for all $\gamma \in S_{\vreg}$},
\tag{$\dagger$}
\end{equation*}
where $\theta^w$ is the character of $S$ defined by $\theta^w(\gamma)=\theta({}^{w}\gamma)=\theta(w\gamma w^{-1})$.

\begin{lemma}\label{lem:hom G+}
Let $\pi$ be a smooth irreducible representation of $SG_{\x,0}$ satisfying (\ref{e:char vreg}) for the character $\theta = \phi\cdot\varepsilon^\ram[\phi]$ of $S$. 
Then
\begin{equation*}
\Hom_{G_{\x,0+}\cap K^{d}}(\rho_d' \otimes \phi_d, \pi) \neq 0.
\end{equation*}
\end{lemma}

\begin{proof}
This is very similar to the proof of Lemma \ref{lem:hom G+ +} but with one additional argument.
As in that proof, it is enough to show that we have
\[
\sum_{\gamma \in S_{0,\vreg} G_{\x,0+}\cap K^{d}} \Theta_{\rho'_{d} \otimes \phi_d}(\gamma) \cdot \overline{\Theta_\pi(\gamma)}
\neq
0.
\]
By the same argument as in the proof of Lemma \ref{lem:hom G+ +}, 
\[
\sum_{\gamma \in S_{0,\vreg} G_{\x,0+}\cap K^{d}} \Theta_{\rho'_{d} \otimes \phi_d}(\gamma) \cdot \overline{\Theta_\pi(\gamma)}
=
|K^{d}/S|
\sum_{\gamma'\in S_{0,\vreg}}
\Theta_{\rho'_{d} \otimes \phi_d}(\gamma') \cdot \overline{\Theta_{\pi}(\gamma')}.
\]
By Proposition \ref{prop:rho_d'} and the assumption of $\pi$, we have
\[
\sum_{\gamma'\in S_{0,\vreg}}
\Theta_{\rho'_{d} \otimes \phi_d}(\gamma') \cdot \overline{\Theta_{\pi}(\gamma')}
=
(-1)^{r(\bfS, \phi)}\bar{c}
\sum_{w\in W_{G_{\x,0}(\bfS)}}
\sum_{\gamma'\in S_{0,\vreg}}
\theta(\gamma')\overline{\theta^{w}(\gamma')}.
\]
Recall the summand $\sum_{\gamma'\in S_{0,\vreg}}\theta(\gamma')\overline{\theta(\gamma')}$ corresponding to $w=1$ is not zero by the positivity of $\theta(\gamma')\overline{\theta(\gamma')}$ and the non-emptyness of $S_{0,\vreg}$.
Hence it suffices to check that for $w \in W_{G_{\x,0}}(\bfS)\smallsetminus\{1\}$, the corresponding sum $\sum_{\gamma' \in S_{0,\vreg}} \theta(\gamma') \overline{\theta^w(\gamma')}$ vanishes.
By fixing a set of representatives $\{\tilde{\gamma}'\}$ of the quotient $S_{0,\vreg}/S_{0+}\cong\bbS(\F_{q})_{\vreg}$, we get
\begin{align*}
\sum_{\gamma' \in S_{0,\vreg}} \theta(\gamma') \cdot \overline{\theta^w(\gamma')}
&=
\sum_{\gamma' \in \bbS(\F_{q})_{\vreg}}
\sum_{\gamma_{+}\in S_{0+}}
\theta(\tilde{\gamma}'\gamma_{+})\cdot\overline{\theta^w(\tilde{\gamma}'\gamma_{+})}\\
&=
\sum_{\gamma' \in \bbS(\F_{q})_{\vreg}}
\theta(\tilde{\gamma}')\cdot\overline{\theta^w(\tilde{\gamma}')}
\sum_{\gamma_{+}\in S_{0+}}
\theta(\gamma_{+})\cdot\overline{\theta^w(\gamma_{+})}.
\end{align*}
Since $\theta|_{S_{0+}}$ has trivial $W_{G_{\x,0}}(\bfS)$-stabilizer by Lemma \ref{lem:theta-stab}, it follows that $\sum_{\gamma_{+} \in S_{0+}} \theta(\gamma_+) \cdot \overline{\theta^w(\gamma_+)} =0$ whenever $w\neq1$.
This completes the proof.
\end{proof}

When we prove the analogue of Theorem \ref{thm:SG+ vreg} in the setting of $SG_{\x,0}$-representations, we will need the following lemma. In the $\GL_n$ setting, this result is due to Henniart \cite{MR1263525}, and the proof of the general setting given here is a direct generalization of Henniart's proof.

\begin{lem}\label{lem:toral Henn}
Let $\theta, \theta' \from S_0 \to \bbC^\times$ be two smooth characters and assume that $\theta|_{S_{0+}}$ has trivial $W_{G_{\x,0}}(\bfS)$-stabilizer.
If for some nonzero constant $c \in \bbC$, \begin{equation*}
\sum_{w \in W_{G_{\x,0}}(\bfS)} \theta^w(\gamma) = c \cdot \sum_{w \in W_{G_{\x,0}}(\bfS)} \theta'^w(\gamma)
\end{equation*}
for all $\gamma \in S_{0,\vreg}$, then $c = 1$ and $\theta' = \theta^{w}$ for some $w \in W_{G_{\x,0}}(\bfS)$.
\end{lem}

\begin{proof}
We follow the same strategy as Henniart in \cite[Section 5.3]{MR1263525}. Henniart works in the setting of $G = \GL_p$, but the proof generalizes with no problems. We present it here.

We first show that the conclusion must hold on $S_{0+}$. Fix $\gamma \in S_{0,\vreg}$. Then $\gamma S_{0+} \subset S_{0,\vreg}$. From this, the assumptions of the lemma imply that we have a linear dependence between the $2|W_{G_{\x,0}}(\bfS)|$ (not necessarily distinct) characters $\theta^{w}|_{S_{0+}}$ and $\theta'^{w}|_{S_{0+}}$ for $w \in W_{G_{\x,0}}(\bfS)$. Explicitly, we have the following identity of $S_{0+}$-characters:
\begin{equation*}
c \cdot \theta' = \theta'(\gamma)^{-1} \Bigl(\sum_{w \in W_{G_{\x,0}}(\bfS)} \theta^{w}(\gamma) \cdot \theta^{w} - c \cdot \sum_{1 \neq w \in W_{G_{\x,0}}(\bfS)} \theta'^{w}(\gamma) \cdot \theta'^{w}\Bigr).
\end{equation*}
Using this linear dependence together with the assumption that $\theta|_{S_{0+}}$ has trivial $W_{G_{\x,0}}(\bfS)$-stabilizer, we see that for any $w' \in W_{G_{\x,0}}(\bfS)$, we have
\begin{equation*}
c \cdot \langle \theta', \theta^{w'} \rangle_{S_{0+}} = \frac{\theta^{w'}(\gamma)}{\theta'(\gamma)} - c \cdot \sum_{1 \neq w \in W_{G_{\x,0}}(\bfS)} \frac{\theta'^{w}(\gamma)}{\theta'(\gamma)} \cdot \langle \theta'^{w}, \theta^{w'} \rangle_{S_{0+}}.
\end{equation*}
If $\langle \theta'^{w}, \theta^{w'} \rangle_{S_{0+}} \neq 0$ for some $w \in W_{G_{\x,0}}(\bfS)$, then we have shown that the conclusion of the lemma holds on $S_{0+}$. Otherwise, we see that $\langle \theta', \theta^{w'} \rangle_{S_{0+}} \neq 0$, and again we see that the conclusion of the lemma holds on $S_{0+}$. 

We have now shown that $\theta'|_{S_{0+}} = \theta^{w}|_{S_{0+}}$ for some $w \in W_{G_{\x,0}}(\bfS)$ and by the above, we see that for any $\gamma \in S_{0,\vreg}$, we have $c = c \cdot \langle \theta', \theta^{w} \rangle_{S_{0+}} = \frac{\theta^{w}(\gamma)}{\theta'(\gamma)}$ and hence $\theta'(\gamma) = c \cdot \theta^{w}(\gamma)$. By the same argument as at the end of the proof of Theorem \ref{thm:SG+ vreg}, the assumption $|\bbS(\FF_q)|/|\bbS(\F_{q})_{\nvreg}| > 2$ implies that if a character of $S_0$ restricts identically to $c$ on $S_{0,\vreg}$, then $c = 1$, and the conclusion of the lemma now follows.
\end{proof}

We now prove the analogue of Theorem \ref{thm:SG+ vreg}.

\begin{theorem}\label{thm:SG vreg}

Let $\pi$ be a smooth irreducible representation of $SG_{\x,0}$ satisfying \eqref{e:char vreg} for the character $\theta = \phi\cdot\varepsilon^{\ram}[\phi]$ of $S$.
Then $c = (-1)^{r(\bfS, \phi)}$ and $\pi \cong \cc\tau_{d}$.
\end{theorem}

\begin{proof}
Recall that $\cc\tau_{d}$ is defined to be $\Ind_{K^{d}}^{SG_{\x,0}}(\rho_d' \otimes \phi_d)$ (in our situation, we have $\cc K^{d}=K^{d}$ and $\cc\rho'_{d}=\rho'_{d}$).
By comparing the central characters as in the proof of Theorem \ref{thm:SG+ vreg}, it is enough to show that $c = (-1)^{r(\bfS, \phi)}$ and $\pi \cong \Ind_{S_{0}G_{\x,0+}\cap K^{d}}^{S_{0}G_{\x,0}}(\rho_d' \otimes \phi_d)$ as a representation of $S_{0}G_{\x,0}$.

Lemma \ref{lem:hom G+} and Frobenius reciprocity imply that
\[
\Hom_{G_{\x,0}}\bigl(\Ind_{G_{\x,0+}\cap K^{d}}^{G_{\x,0}}(\rho'_{d}\otimes\phi_{d}),\pi\bigr)\neq0.
\]
Recall from the proof of Theorem \ref{thm:SG+ vreg} that
\[
\Ind_{G_{\x,0+}\cap K^{d}}^{S_0 G_{\x,0+}\cap K^{d}}(\rho_d' \otimes \phi_d) \cong \bigoplus_{\theta_0 \from S_{0:0+} \to \bbC^\times} (\rho_d'\otimes\phi_d)\otimes\theta_{0}.
\]
Thus we have
\[
\Ind_{G_{\x,0+}\cap K^{d}}^{G_{\x,0}}(\rho'_{d}\otimes\phi_{d})
\cong
\bigoplus_{\theta_0 \from S_{0:0+} \to \bbC^\times} \Ind_{S_{0}G_{\x,0+}\cap K^{d}}^{G_{\x,0}}\bigl((\rho_d'\otimes\phi_d)\otimes\theta_{0}\bigr).
\]
Hence, by the irreducibility of $\pi$, we have
\begin{align}\label{eq:SG vreg}
\Hom_{G_{\x,0}}\Bigl(\Ind_{S_{0}G_{\x,0+}\cap K^{d}}^{G_{\x,0}}\bigl((\rho_d'\otimes\phi_d)\otimes\theta_{0}\bigr),\pi\Bigr)\neq0
\end{align}
for a unique $\theta_{0}$.

Let us investigate the representation $\Ind_{S_{0}G_{\x,0+}\cap K^{d}}^{G_{\x,0}}((\rho_d'\otimes\phi_d)\otimes\theta_{0})$.
Recall that $\rho_d'\otimes\phi_d$ is the representation of $K^{d}$ obtained from a tame elliptic regular pair $(\bfS,\phi)$.
If we take an extension $\tilde{\theta}_{0}$ of $\theta_{0}$ from $S_{0}$ to $S$, then the pair $(\bfS,\phi\otimes\tilde{\theta}_{0})$ is again tame elliptic regular.
Furthermore, we can easily check that the Yu's construction (reviewed in Section \ref{subsec:Yu}) attaches to $(\bfS,\phi\otimes\tilde{\theta}_{0})$ the representation $(\rho_d'\otimes\phi_d)\otimes\tilde{\theta}_{0}$ of $K^{d}$.
Therefore the compact induction of $(\rho_d'\otimes\phi_d)\otimes\tilde{\theta}_{0}$ from $K^{d}$ to $G$ is an irreducible (supercuspidal) representation.
In particular, also $\Ind_{K^{d}}^{SG_{\x,0}}((\rho_d'\otimes\phi_d)\otimes\tilde{\theta}_{0})$ is irreducible, hence so is its restriction to $G_{\x,0}$ (note that $SG_{\x,0}=Z_{\bfG}G_{\x,0}$).
As we have $\Ind_{K^{d}}^{SG_{\x,0}}((\rho_d'\otimes\phi_d)\otimes\tilde{\theta}_{0})|_{G_{\x,0}}\cong\Ind_{S_{0}G_{\x,0+}\cap K^{d}}^{G_{\x,0}}((\rho_d'\otimes\phi_d)\otimes\theta_{0})$, we eventually get the irreducibility of $\Ind_{S_{0}G_{\x,0+}\cap K^{d}}^{G_{\x,0}}((\rho_d'\otimes\phi_d)\otimes\theta_{0})$.
Then (\ref{eq:SG vreg}) implies that 
\[
\Ind_{S_{0}G_{\x,0+}\cap K^{d}}^{G_{\x,0}}\bigl((\rho_d'\otimes\phi_d)\otimes\theta_{0}\bigr)
\cong
\pi.
\]

It now remains to show that $c = (-1)^{r(\bfS, \phi)}$ and that $\theta_{0} = \mathbbm{1}$.
By applying Proposition \ref{prop:AS-vreg} to $\Ind_{K^{d}}^{SG_{\x,0}}((\rho_d'\otimes\phi_d)\otimes\tilde{\theta_{0}})$ (which is the representation ``$\cc\tau_{d}$'' arising from the twisted pair $(\bfS,\phi\otimes\tilde{\theta}_{0})$) for all $\gamma \in S_{0,\vreg}$ we have
\begin{align*}
\Theta_{\Ind_{S_{0}G_{\x,0+}\cap K^{d}}^{G_{\x,0}}((\rho_d'\otimes\phi_d)\otimes\theta_{0})}(\gamma)
&=
\Theta_{\Ind_{K^{d}}^{SG_{\x,0}}((\rho_d'\otimes\phi_d)\otimes\tilde{\theta}_{0})}(\gamma)\\
&=
(-1)^{r(\bfS, \phi)} \sum_{w \in W_{G_{\x,0}}(\bfS)} \theta^{w}(\gamma) \theta_{0}^{w}(\gamma).
\end{align*}
(Note that the character $\varepsilon^{\ram}[\phi\otimes\theta_{0}]$ associated to the twisted character $\phi\otimes\theta_{0}$ is the same as $\varepsilon^{\ram}[\phi]$ since $\theta_{0}$ is of depth zero.)
Thus, by the assumption of $\pi$, we get
\[
(-1)^{r(\bfS, \phi)} \sum_{w \in W_{G_{\x,0}}(\bfS)} \theta^{w}(\gamma) \theta_{0}^{w}(\gamma)
=
c\sum_{w \in W_{G_{\x,0}}(\bfS)} \theta^w(\gamma).
\]
We therefore may apply Lemma \ref{lem:toral Henn} and conclude that $c = (-1)^{r(\bfS, \phi)}$ and $\theta\theta_{0}=\theta^{w}$ for some $w\in W_{G_{\x,0}}(\bfS)$.
By restricting to $S_{0+}$, we get $\phi|_{S_{0+}}=\phi^{w}|_{S_{0+}}$ (recall that $\varepsilon^{\ram}[\phi]|_{S_{0+}}=\mathbbm{1}$).
Then it follows that $w=1$ from the assumption that $\phi|_{S_{0+}}$ has trivial $W_{G_{\x,0}}(\bfS)$-stabilizer.
Thus finally we get $\theta\theta_{0}=\theta$, which implies that $\theta_{0}=\mathbbm{1}$.
\end{proof}

\section{Deligne--Lusztig varieties for parahoric subgroups}\label{sec:geom} 

Up to this point in the paper, our discussion has been centered on algebraically constructed representations of parahoric subgroups whose compact induction to $G$ is irreducible (and hence supercuspidal). 
Contrary to every other section of this paper, everything in this section holds with no assumptions on the residue characteristic $p$ of $F$.

In this section, we introduce a class of representations of parahoric subgroups that arise geometrically, via the cohomology of Deligne--Lusztig varieties for parahoric subgroups \cite{MR4197070}. 
These varieties are associated to a maximal torus $\bfS \hookrightarrow \bfG$ such that
\begin{center}
$\bfS$ is unramified but not necessarily elliptic.
\end{center}
When $\bfS$ is elliptic, these representations are closely related to Yu's supercuspidal representations; Sections \ref{sec:comparison} and \ref{sec:Lpacket} are devoted to understanding this relationship and to discussing the interesting consequences of having such a comparison.

\subsection{The varieties $X_r$}\label{sec:parahoricDL}

We review a generalization of Deligne--Lusztig varieties for $G_{\x,0}$ defined in \cite{MR4197070}. In the setting that $G_{\x,0} = \bbG(\cO_F)$ for a reductive group $\bbG$ over $\FF_q$, these varieties were originally defined and studied by Lusztig in \cite{MR2048585} (for $F$ equal characteristic) and later by Stasinski in \cite{MR2558788} (for $F$ mixed characteristic). We warn the reader that there is a change of convention between our present work and the papers \cite{MR2048585, MR2558788, MR4197070}: in our normalization, $\bbG_0$ is a reductive group over a finite field.

Let $\bfS \hookrightarrow \bfG$ be an unramified maximal torus (not necessarily elliptic) defined over $F$ and let $\x \in \cB(\bfG, F)$ be a point in the apartment of $\bfS$. Say $\bfS$ splits over the degree-$n$ unramified extension $F_n$ of $F$ and let $\bfU$ be the unipotent radical of a $F_n$-rational Borel subgroup of $\bfG_{F_n}$ containing $\bfS_{F_{n}}$. 
Following, \cite[Section 2.6]{MR4197070}, for $r \in \bZ_{\geq 0}$, we have group schemes $\bbS_r \subset \bbG_r$ defined over $\FF_q$ such that $\bbG_r(\FF_q) = G_{\x,0:r+}$ and $\bbS_r(\FF_q) = S_{0:r+}$, and a group scheme $\bbU_r \subset \bbG_{r, \FF_{q^n}}$. Let $\sigma \from \bbG_r \to \bbG_r$ denote the Frobenius morphism associated to the $\FF_q$-rational structure on $\bbG_r$.

\begin{definition}\label{def:Xr}
For $r \in \bZ_{\geq 0}$, we define the following $\FF_{q^n}$-subscheme of $\bbG_r$:
\begin{equation*}
X_r \colonequals \{x \in \bG_r \mid x^{-1} \sigma(x) \in \bU_r\}.
\end{equation*}
\end{definition}

$X_r$ is separated, smooth, and of finite type over $\FF_{q^n}$ \cite[Lemma 3.1]{MR4197070}. For $(g,t) \in G_{\x,0:r+} \times S_{0:r+}$ and $x \in X_r$, the assignment $(g,t) * x = g x t$ defines an action of $G_{\x,0:r+} \times S_{0:r+}$ on $X_r$ which pulls back to an action of $G_{\x,0} \times S_0$.

We fix a prime number $\ell$ which is not equal to $p$.
By functoriality, the cohomology groups $H_c^i(X_r, \overline \QQ_\ell)$ are representations of $G_{\x,0} \times S_0$.
If $\theta \from S \to \overline \QQ_\ell^\times$ is a $\overline{\Q}_{\ell}^{\times}$-valued character trivial on $S_{r+}$, the subspace $H_c^i(X_r, \overline \QQ_\ell)[\theta]$ of $H_c^i(X_r, \overline \QQ_\ell)$ on which $S_0$ acts by multiplication by $\theta|_{S_{0}}$, is a representation of $G_{\x,0}$.
We define a virtual representation $R_{\bbS_r, \bU_r}^{\bG_r}(\theta)$ of $G_{\x,0}$ with $\overline{\Q}_{\ell}$-coefficient by
\begin{equation*}
R_{\bbS_r, \bU_r}^{\bG_r}(\theta) \colonequals \sum_{i \geq 0} (-1)^i H_c^i(X_r, \overline \QQ_\ell)[\theta].
\end{equation*}

In the special case that $r = 0$, the variety $X_0$ is an affine fibration over a classical Deligne--Lusztig variety and hence their cohomology is the same up to an even degree shift. Hence the $R_{\bS_0, \bU_0}^{\bG_0}(\theta)$ is exactly the usual Deligne--Lusztig representation of the finite reductive group $\bG_0(\FF_q)$ attached to the character $\theta|_{S_{0}}$ of $\bS_0(\FF_q)$.

The next to results are $r>0$ analogues of known $r=0$ theorems. Proposition \ref{prop:geom vreg} in the $r=0$ setting is a special case of the Deligne--Lusztig character formula \cite[Theorem 4.2]{MR393266} and Theorem \ref{thm:mackey} in the $r=0$ setting is a special case of the Mackey formula for Deligne--Lusztig varieties \cite[Theorem 6.8]{MR393266}.

\begin{prop}[{\cite[Theorem 1.2]{MR4197070}}]
\label{prop:geom vreg}
Let $\theta \from S \to \overline \QQ_\ell^\times$ be any character trivial on $S_{r+}$ and let $\gamma \in G_{\x,0}$ be an unramified very regular element. If $\gamma$ is not $G_{\x,0}$-conjugate to an element of $S$, then the character value $\Theta_{R_{\bS_r, \bU_r}^{\bG_r}(\theta)}(\gamma)$ is equal to zero. If $\gamma$ is $G_{\x,0}$-conjugate to an element of $S$ (in this case, we assume that $\gamma$ itself belongs to $S$ by conjugating), we have
\begin{equation*}
\Theta_{R_{\bbS_r, \bU_r}^{\bG_r}(\theta)}(\gamma) = \sum_{w \in W_{G_{\x,0}}(\bfS)} \theta^{w}(\gamma).
\end{equation*}
\end{prop}

\begin{thm}[{\cite[Theorem 1.1]{MR4197070}}]
\label{thm:mackey}
Let $r > 0$ and let $\theta \from S \to \overline \QQ_\ell^\times$ be a character trivial on $S_{r+}$ such that $\theta|_{\Nr_{E/F}(\alpha^\vee(E_r^\times))} \not\equiv \mathbbm{1}$ for all $\alpha \in R(\bfS, \bfG)$, where $E$ is the splitting field of $\bfS$. 
Let $\bfU' \subset \bfG_{F^{\ur}}$ be another unipotent radical of a Borel subgroup of $\bfG_{F^{\ur}}$ containing $\bfS$. 
Then for any character $\theta' \from S_{0:r+} \to \overline \QQ_\ell^\times$,
\begin{equation}\label{e:mackey}
\bigl\langle R_{\bbS_r, \bU_r}^{\bG_r}(\theta), R_{\bbS_r, \bU_r'}^{\bG_r}(\theta') \bigr\rangle_{G_{\x,0:r+}} = \sum_{w \in W_{G_{\x,0}}(\bfS)} \langle \theta, \theta'^{w} \rangle_{S_{0:r+}}.
\end{equation}
In particular, if $\theta$ has trivial $W_{G_{\x,0}}(\bfS)$-stabilizer, then $R_{\bbS_r, \bU_r}^{\bG_r}(\theta)$ is irreducible (up to a sign) and is independent of the choice of $\bfU$.
\end{thm}

\begin{remark}\label{rem:regular}
Let $\bfS$ be an unramified maximal torus of $\bfG$ and $\theta$ a character of $S$ which is trivial on $S_{r+}$ for $r\in\R_{>0}$.
Then recall that we may associate a sequence of root subsystems $\{R_{r}\}$ (or reductive subgroups) to $\theta$ (see Section \ref{subsec:Howe}).
In terms of this sequence of root subsystems, the condition that
\[
\text{$\theta|_{\Nr_{E/F}(\alpha^\vee(E_r^\times))} \not\equiv \mathbbm{1}$ for all $\alpha \in R(\bfS, \bfG)$}
\]
is equivalent to the condition that 
\[
R_{0}=\cdots=R_{r}=\varnothing.
\]
On the other hand, since $\theta|_{S_{r+}}$ is trivial, we have $R_{r+}=R(\bfS,\bfG)$.
Thus the assumption of Theorem \ref{thm:mackey} is satisfied if and only if its associated sequence of reductive subgroups is given by $(\bfG^{0}(\bfS,\theta),\bfG^{1}(\bfS,\theta))$, where $\bfG^{0}(\bfS,\theta)=\bfS$ and $\bfG^{1}(\bfS,\theta)=\bfG$, in other words, $\theta$ is $0$-toral.
Note that, if we furthermore assume the four conditions on $p$ (odd, not bad for $\bfG$, $p \nmid |\pi_1(\bfG_{\der})|$ and $p\nmid|\pi_{1}(\widehat{\bfG}_{\der})|$), the torality implies that $\theta|_{S_{0+}}$ has trivial $W_{G}(\bfS)$-stabilizer by Lemma \ref{lem:theta-stab}; in particular, $\theta$ has trivial $W_{G_{\x,0}}(\bfS)$-stabilizer.
In summary, when the four conditions on $p$ are satisfied, Theorem \ref{thm:mackey} asserts that the virtual representation $R_{\bbS_r, \bU_r}^{\bG_r}(\theta)$ is irreducible (up to sign) for any $0$-toral character $\theta$ of depth $r$.
\end{remark}

\subsection{The Drinfeld stratification}\label{sec:drinfeld}

In this section, we introduce subschemes of $X_r$ associated to certain twisted Levi subgroups of $\bfG$, following \cite[Definition 3.3.1]{CI_DrinfeldStrat}. The work of this paper will allow us to formulate a conjecture (Conjecture \ref{conj:drinfeld}) about the Drinfeld stratification which had previously (in \textit{op.\ cit.}) only been conjectured for $\GL_n$. We will prove part of Conjecture \ref{conj:drinfeld} in Section \ref{sec:comparison} (see Theorem \ref{thm:drinfeld}).

Let $\bG_r^+$ be the kernel of the natural reduction map $\bG_r \to \bG_0$. Note that $\bG_r^+(\FF_q) = G_{\x,0+:r+}$. For a twisted Levi subgroup $\bfL$ of $\bfG$ containing $\bfS$, we let $\bL_r$ denote the corresponding subgroup scheme of $\bbG_r$. Attached to the twisted Levi $\bfL$, one can naturally associate the following subscheme of $X_r$:
\begin{align*}
X_r^{(\bfL)} \colonequals \{x \in \bG_r \mid x^{-1}\sigma(x) \in (\bL_r \cap \bU_r) \bU_r^+\}.
\end{align*}
It is straightforward to show (\cite[Lemma 3.3.3]{CI_DrinfeldStrat}) that $X_r^{(\bfL)}$ is a disjoint union over $G_{\x,0}/(L_{\x,0}G_{\x,0+})$ 
copies of the scheme $X_r \cap \bL_r \bbG_r^+$.
It follows  (\cite[Lemma 3.3.4]{CI_DrinfeldStrat}) from this that for any character $\theta \from \bbS_r(\FF_q) \to \overline \QQ_\ell^\times$,
\begin{equation*}
H_c^i(X_r^{(\bfL)}, \overline \QQ_\ell)[\theta] \cong \Ind_{\bL_r(\FF_q) \bG_r^+(\FF_q)}^{\bG_r(\FF_q)}\big(H_c^i(X_r \cap \bL_r \bG_r^+, \overline \QQ_\ell)[\theta]\big) \qquad \text{for all $i \geq 0$}.
\end{equation*}

\begin{conj}\label{conj:drinfeld}
Let $\bfL \subset \bfG$ be a twisted Levi subgroup of $\bfG$ containing $\bfS$. 
If $R_{0+}\subset R(\bfS, \bfL)$ (or, equivalently, $\bfG^{0}(\bfS,\theta)\subset\bfL$), then
\begin{equation*}
H_c^{\ast}(X_r, \overline \QQ_\ell)[\theta] \cong H_c^{\ast}(X_r^{(\bfL)}, \overline \QQ_\ell)[\theta].
\end{equation*}
\end{conj}

\begin{remark}
Let $\bfG = \GL_n$. In the setting $\bfL = \bfS$, Conjecture \ref{conj:drinfeld} was proved in \cite[Theorem 4.1]{CI_loopGLn} using geometric techniques (note that this forces $\bfG^{0}(\bfS,\theta)=\bfS$). 
On the other hand, in the $\GL_n$ setting, it is expected that Conjecture \ref{conj:drinfeld} holds without any assumptions on $p$. By \cite[Lemma 3.7.7]{MR4013740}, the notions of a tame elliptic regular pair and an admissible pair coincides.
Then one can translate between Kaletha's generalized Howe factorization for tame elliptic regular pairs and the classical Howe factorization for admissible characters (see \cite[Section 5.2]{MR4206603}). Then Conjecture \ref{conj:drinfeld} is the same assertion of \cite[Conjecture 7.2.1]{CI_DrinfeldStrat}.
\end{remark}

The proof of Theorem \ref{thm:mackey} \cite[Theorem 1.1]{MR4197070} on the cohomology of $X_r$ can be applied almost identically to obtain an analogous statement for the cohomology of $X_r \cap \bL_r \bG_r^+$. 

\begin{theorem}\label{thm:irred L}
Let $r > 0$ and let $\theta \from S_{0} \to \overline \QQ_\ell^\times$ be a $0$-toral character trivial on $S_{r+}$. 
Let $\bfU' \subset \bfG_{F^{\ur}}$ be another unipotent radical of a Borel subgroup of $\bfG_{F^{\ur}}$ containing $\bfS$ and let $X_r'$ be the associated variety. 
Then for any $\theta' \from S_{0:r+} \to \overline \QQ_\ell^\times$,
\begin{multline*}
\Big\langle H_c^*(X_r \cap \bL_r \bG_r^+, \overline \QQ_\ell)[\theta], H_c^*(X'_r \cap \bL_r \bG_r^+, \overline \QQ_\ell)[\theta'] \Big\rangle_{L_{\x,0:r+}G_{\x,0+:r+}}\\
=
\sum_{w \in W_{L_{\x,0}}(\bfS)} \langle \theta, \theta'^{w} \rangle_{S_{0:r+}}.
\end{multline*}
In particular, if $\theta$ has trivial $W_{L_{\x,0}}(\bfS)$-stabilizer, then $H_c^*(X_r \cap \bL_r \bG_r^+, \overline \QQ_\ell)[\theta]$ is irreducible (up to a sign) and is independent of the choice of $\bfU$.
\end{theorem}

\begin{remark}\label{rem:regular+}
Let $\bfS$ be an unramified maximal torus of $\bfG$ and $\theta$ a $0$-toral character of $S$ which is trivial on $S_{r+}$ for $r\in\R_{>0}$.
Let us consider the cohomology $H_c^*(X_r \cap \bL_r \bG_r^+, \overline \QQ_\ell)[\theta]$ for $\bfL=\bfS$.
Then, since the group $W_{L_{\x,0}}(\bfS)$ is trivial, Theorem \ref{thm:irred L} asserts that $H_c^*(X_r \cap \bS_r \bG_r^+, \overline \QQ_\ell)[\theta]$ is irreducible for any $0$-toral character $\theta$ of $S$ of depth $r$.
Note that, in contrast to Remark \ref{rem:regular}, we do not need any assumption on $p$ for this irreducibility result for $0$-toral characters.
\end{remark}

\section{Geometric toral supercuspidal representations}\label{sec:comparison}

We are now ready to prove one of our main theorems of this paper: an explicit comparison of Yu's algebraic construction of supercuspidal representations (Section \ref{sec:Yu}) and the geometric construction of parahoric representations coming from Deligne--Lusztig-type varieties (Section \ref{sec:geom}). 
To do this, we will apply the results of Section \ref{sec:ell unram vreg}---that certain representations are characterized by their character values at unramified very regular elements. 
These comparison methods will result in several theorems as quick corollaries of the work already established in previous sections of this paper. 

As such, we inherit the assumptions on $p$ and the assumptions on $(\bfS, \theta)$ that have been needed for these preceding sections. In this section, we will assume \eqref{ineq} (i.e., $|\bbS(\FF_q)|/|\bbS(\FF_q)_{\nvreg}| > 2$) in addition to our usual assumptions on the residue characteristic of $F$ (i.e., $p$ is odd, $p$ is not bad for $\bfG$, $p \nmid |\pi_1(\bfG_{\der})|$ and $p \nmid |\pi_1(\widehat \bfG_\der)|$). Aside from the $\GL_n$ setting discussed in Section \ref{sec:GLn}), in this section we consider $(\bfS, \theta)$ where
\begin{enumerate}[label=\textbullet]
\item
$\bfS$ is unramified elliptic,
\item
$\theta$ is $0$-toral; equivalently, $\bfG^0(\bfS, \theta) = \bfS$ and $\bfG^1(\bfS,\theta) = \bfG$.
\end{enumerate}
We note that such a pair $(\bfS,\theta)$ is necessarily a tame elliptic regular pair in the sense of Kaletha (Section \ref{subsec:Kaletha-TER}) by definition. 
The first condition is inherited from Section \ref{sec:geom}, though there we considered unramified tori $\bfS$ which were not necessarily elliptic.
The second condition is inherited from Section \ref{sec:ell unram vreg}, though there we considered $\theta$ satisfying the weaker condition that $\theta$ is toral. 

Fix a prime number $\ell$ which is not equal to $p$ and an isomorphism $\iota\colon\C\cong\overline{\Q}_{\ell}$. Let $\theta \from S \to \bbC^\times$ be any character which is trivial on $S_{r+}$. The following lemma demonstrates that the choice of $\iota$ is innocuous from the perspective of considering the \textit{a priori} $\ell$-adic representation $R_{\bS_r,\bU_r}^{\bG_r}(\theta) = H_c^*(X_r, \overline \QQ_\ell)[\theta]$ as a complex representation.

\begin{lemma}\label{lem:indep ell}
Let $\theta_{\ell,\iota}$ denote the $\QQ_\ell^\times$-valued character obtained by composing $\theta$ with the isomorphism $\iota \from \bbC^\times \cong \overline \QQ_\ell^\times$. The complex $SG_{\x,0}$-representation $\iota R_{\bS_r, \bU_r}^{\bG_r}(\theta_{\ell,\iota})$ is independent of $\ell$.
\end{lemma}

\begin{proof}
The scheme $X_r$ (Definition \ref{def:Xr}) is separated and of finite type over $\overline \FF_q$ (see \cite[Lemma 3.1]{MR4197070}); therefore by \cite[Proposition 3.3]{MR393266} (thanks to Andy Gordon for reminding the authors about this), we know that for any $(s,g) \in S_0 \times G_{\x,0}$, the trace $\Tr((s,g)^*, H_c^*(X_r, \overline \QQ_\ell))$ is an algebraic integer independent of $\ell$. By definition,
\begin{equation*}
\Tr\bigl(g^*, R_{\bS_r, \bU_r}^{\bG_r}(\theta_{\ell,\iota})\bigr) = \frac{1}{|S_{0:r+}|} \sum_{s \in S_{0:r+}} \theta(s)^{-1} \cdot \Tr\bigl((s,g)^*, H_c^*(X_r, \overline \QQ_\ell)\bigr).
\end{equation*}
As a character of the subquotient $S_{0:r+}$, the $S$-character $\theta$ is valued in the algebraic integers; it follows from the above then that the character of $R_{\bS_r, \bU_r}^{\bG_r}(\theta_{\ell,\iota})$ considered as a representation of $G_{\x,0}$ takes values in $\overline \QQ$. It follows that $\iota R_{\bS_r, \bU_r}^{\bG_r}(\theta_{\ell,\iota}) \cong \iota' R_{\bS_r, \bU_r}^{\bG_r}(\theta_{\ell',\iota'})$ as $G_{\x,0}$-representations, where $\ell'$ is a prime not equal to $p$ and $\iota' \from \bbC \cong \overline \QQ_{\ell'}$ is a(ny) chosen isomorphism. Since the $Z_{\bfG}$-action on $\iota R_{\bS_r, \bU_r}^{\bG_r}(\theta_{\ell,\iota})$ is given simply by multiplication by $\theta|_{Z_{\bfG}}$, it now further follows that $\iota R_{\bS_r, \bU_r}^{\bG_r}(\theta_{\ell,\iota})$ is independent of $\ell$ and $\iota$ as an $SG_{\x,0}$-representation.
\end{proof}

\subsection{Comparison for $X_r$}\label{subsec:compare SG}

Theorem \ref{thm:comparison SG} is one of the central theorems of this paper. It not only gives an explicit description of the geometrically arising representation $R_{\bbS_r}^{\bG_r}(\theta)$, but puts it in the framework of algebraic constructions of supercuspidal representations, allowing us passage between these drastically different realizations of these representations. From this theorem, we get the irreducibility of $\cInd_{SG_{\x,0}}^{G}(R_{\bS_r,\bU_r}^{\bG_r}(\theta))$ (Equation \eqref{e:geom sc}) as a free consequence, but this is already highly nontrivial. In the setting that $\bfG$ is an inner form of $\GL_n$, a geometric proof of \eqref{e:geom sc} is in \cite{CI_ADLV} (see Theorems 9.1 and 12.5) and under the relaxation of the $0$-torality condition to the torality discussed in Section \ref{sec:GLn}, a geometric proof of \eqref{e:geom sc} is the subject of \cite{CI_loopGLn}, which further relies on \cite{CI_DrinfeldStrat}.

We focus on the $0$-toral case. When $\theta \from S \to \CC^\times$ is $0$-toral, then by Theorem \ref{thm:mackey}, the virtual representation $R_{\bS_r, \bU_r}^{\bG_r}(\theta)$ does not depend on the choice of $\bU_r$, so we replace this notation by the simpler $R_{\bS_r}^{\bG_r}(\theta)$.

\begin{theorem}\label{thm:comparison SG}
Let $\bfS$ be an unramified elliptic maximal torus of $\bfG$ and $\theta\colon S\rightarrow\C^{\times}$ a $0$-toral character of depth $r\in\R_{>0}$.
We put $\phi:=\theta\cdot\varepsilon^{\ram}[\theta]$.
Then we have
\[
R_{\bS_r}^{\bG_r}(\theta)
\cong
(-1)^{r(\bfS, \phi)} \cc\tau_{d},
\]
where $\cc\tau_{d}$ in the right-hand side is Yu's representation associated to the pair $(\bfS,\phi)$ (see Section \ref{sec:Yu}).
In particular, 
\begin{equation}\label{e:geom sc}
\text{$\cInd_{SG_{\x,0}}^{G}\bigl((-1)^{r(\bfS, \phi)} R_{\bS_r}^{\bG_r}(\theta)\bigr) \cong \pi_{(\bfS, \phi)}$ is irreducible supercuspidal.}
\end{equation}
\end{theorem}

\begin{proof}
By the $0$-torality of $\theta$ and the assumption on $p$, Theorem \ref{thm:mackey} implies that $R_{\bS_r}^{\bG_r}(\theta)$ is an irreducible virtual representation of $SG_{\x,0}$ (see Remark \ref{rem:regular}).
Hence either $R_{\bS_r}^{\bG_r}(\theta)$ or its negative is a genuine representation; call it $c \cdot R_{\bS_r}^{\bG_r}(\theta)$. By Proposition \ref{prop:geom vreg}, for $\gamma \in S_{\vreg}$,
\[
\Theta_{c \cdot R_{\bS_r}^{\bG_r}(\theta)}(\gamma)
=
c \cdot \sum_{w \in W_{G_{\x,0}}(\bfS)} \theta^{w}(\gamma).
\]

Note that we have $\phi|_{S_{0+}}\cong\theta|_{S_{0+}}$ since the character $\varepsilon^{\ram}[\theta]$ (Definition \ref{def:varepsilon ram}) is tamely ramified. 
Moreover, since $\varepsilon^{\ram}[\phi]$ is determined by $\phi|_{S_{0+}}$, we have $\varepsilon^{\ram}[\phi]=\varepsilon^{\ram}[\theta]$.
Thus, by also noting that $\varepsilon^{\ram}$ is a sign character, the equality $\phi=\theta\cdot\varepsilon^{\ram}[\theta]$ is equivalent to the equality $\theta=\phi\cdot\varepsilon^{\ram}[\phi]$.
Therefore Theorem \ref{thm:SG vreg} is applicable to $\pi:=c \cdot R_{\bS_r}^{\bG_r}(\theta)$ and we get $c=(-1)^{r(\bfS,\phi)}$ and $\pi\cong\cc\tau_{d}$.
\end{proof}

\begin{remark}\label{rem:SG sign}
Recall that under the assumptions in Theorem \ref{thm:comparison SG}, the virtual representation $R_{\bbS_r}^{\bbG_r}(\theta)$ is irreducible. Suppose that in fact $H_c^*(X_r, \overline \QQ_\ell)[\theta]$ is in fact concentrated in a single degree $r^{\geom}(\theta)$. (This is known to be the case when $\bfG$ is a division algebra; see \cite{MR3848018,MR4168688}.) Then Theorem \ref{thm:comparison SG} gives a \textit{cohomological} meaning to the purely root-theoretically defined integer $r(\bfS,\phi)$ (Proposition \ref{prop:rho_d'}):
\begin{equation*}
r(\bfS,\phi) \equiv r^{\geom}(\theta) \pmod 2.
\end{equation*}
(In the division algebra case, one can check this directly---compare the formula for $r(\bfS,\phi)$ in Proposition \ref{prop:rho_d'} to the formula for the nonvanishing cohomological degree in \cite[Theorem 5.1.1]{MR4168688}.)
\end{remark}

\begin{remark}\label{rem:genericity condition}
The $0$-torality condition is only needed because of Theorem \ref{thm:mackey} (\cite[Theorem 1.1]{MR4197070}).
On the other hand, it is expected (and known in certain cases) that Theorem \ref{thm:mackey} holds beyond the $0$-toral case as long as $\bfS$ is elliptic.
When $\bfG$ is an inner form of $\GL_n$, Theorem \ref{thm:mackey} is known without any constraints on $\theta$ (see Theorem B of \cite{CI_loopGLn}), and we will discuss this in more detail in Section \ref{sec:GLn}. For $\bfG$ arbitrary and $\bfS$ Coxeter, establishing Theorem \ref{thm:mackey} is recent work of Dudas--Ivanov (see Theorem 3.2.3 of \cite{DI2020}) announced during the preparation of the present paper.
With this, the $0$-torality condition on $\theta$ in Theorem \ref{thm:comparison SG} can be relaxed to the torality, with identical proof.
\end{remark}

\subsection{Comparison for $X_r \cap \bS_r \bG_r^+$}\label{subsec:compare SG+}

As in Section \ref{subsec:compare SG}, there is a unique extension of $H_c^*(X_r \cap \bS_r \bG_r^+, \overline \QQ_\ell)[\theta]$ to a (virtual) representation of $SG_{\x,0+}$ with central character $\theta|_{Z_{\bfG}}$.
We again denote this representation by $H_c^*(X_r \cap \bS_r \bG_r^+, \overline \QQ_\ell)[\theta]$.

\begin{theorem}\label{thm:comparison SG+}
Let $\bfS$ be an unramified elliptic maximal torus of $\bfG$ and $\theta\colon S\rightarrow\C^{\times}$ a $0$-toral character of depth $r\in\R_{>0}$.
We put $\phi:=\theta\cdot\varepsilon^{\ram}[\theta]$.
Then we have
\[
H_c^*(X_r \cap \bS_r \bG_r^+, \overline \QQ_\ell)[\theta]
\cong
(-1)^{r(\bfS, \phi)} \Ind_{K^{d}}^{SG_{\x,0+}}(\rho_d' \otimes \phi_d),
\]
where the right-hand side is Yu's representation associated to the pair $(\bfS,\phi)$ (see Section \ref{sec:Yu}).
\end{theorem}

\begin{proof}
By the $0$-torality of $\theta$, Theorem \ref{thm:irred L} implies that $H_c^*(X_r \cap \bS_r \bG_r^+, \overline \QQ_\ell)[\theta]$ is an irreducible virtual representation of $SG_{\x,0+}$ (see Remark \ref{rem:regular+}).
The proof of Proposition \ref{prop:geom vreg} in this setting yields the result that if $\gamma \in SG_{\x,0+}$ is an unramified very regular element, then 
\begin{equation*}
\Theta_{H_c^*(X_r \cap \bS_r \bG_r^+, \overline \QQ_\ell)[\theta]}(\gamma) = 
\begin{cases}
\theta(g\gamma g^{-1}) & \text{if $\gamma \in g^{-1} S g$ for some $g \in G_{\x,0+}$}, \\
0 & \text{otherwise.}
\end{cases}
\end{equation*}
Then the same proof as in Theorem \ref{thm:comparison SG} works by using the above formula and Theorem \ref{thm:SG+ vreg} in place of Proposition \ref{prop:geom vreg} and Theorem \ref{thm:SG vreg}, respectively.
\end{proof}

As a straightforward corollary of Theorem \ref{thm:comparison SG} and \ref{thm:comparison SG+}, we resolve Conjecture \ref{conj:drinfeld} for $0$-toral characters:

\begin{theorem}\label{thm:drinfeld}
Let $\bfS$ be an unramified elliptic maximal torus of $\bfG$ and $\theta\colon S\rightarrow\C^{\times}$ a $0$-toral character of depth $r\in\R_{>0}$.
Then
\begin{equation*}
R_{\bS_r}^{\bG_r}(\theta)
\cong
H_c^*(X_r^{(\bfS)}, \overline \QQ_\ell)[\theta]
\cong
\Ind_{SG_{\x,0+}}^{SG_{\x,0}}\big(H_c^*(X_r \cap \bS_r \bG_r^+, \overline \QQ_\ell)[\theta]\big).
\end{equation*}
\end{theorem}

\begin{proof}
The second isomorphism in the assertion holds in general (see the paragraph before Conjecture \ref{conj:drinfeld}), so it suffices to check that $R_{\bS_r}^{\bG_r}(\theta)$ is isomorphic to $\Ind_{SG_{\x,0+}}^{SG_{\x,0}}\big(H_c^*(X_r \cap \bS_r \bG_r^+, \overline \QQ_\ell)[\theta]\big)$. But this isomorphism holds by Theorem \ref{thm:comparison SG} and \ref{thm:comparison SG+} together with the fact that 
\begin{equation*}
\cc \tau_d \cong  \Ind_{SG_{\x,0+}}^{SG_{\x,0}}\Ind_{K^{d}}^{SG_{\x,0+}}(\rho_d' \otimes \phi_d). \qedhere
\end{equation*}
\end{proof}

\subsection{The case of $\GL_n$}\label{sec:GLn}

Let $\bfG$ be an inner form of $\GL_n$ and let $\bfS$ be an unramified elliptic maximal torus of $\bfG$. 
Then there are no bad primes of $\bfG$ and $|\pi_1(\bfG_{\der})| = |\pi_1(\widehat \bfG_{\der})| = 1$, so the only baseline assumption on the residue characteristic of $F$ (Section \ref{subsec:assumptions}) is that $p > 2$.  
In this setting, $G \cong \GL_{n'}(D_{k_0/n_0})$, where $n = n'n_0$ and $D_{k_0/n_0}$ is the division algebra of dimension $n_0^2$ with Hasse invariant $k_0/n_0$, and $S \cong E^\times$, where $E$ is the unramified degree-$n$ extension of $F$ (i.e., it is the smallest extension of $F$ which splits $\bfS$). 
This is a specialization of the setting of Section \ref{sec:geom}, and in this setting Theorems \ref{thm:mackey}, \ref{thm:irred L} are known by work of the first author and A.\ Ivanov without the $0$-torality assumption on $\theta$. We record this result here:

\begin{theorem}[{\cite[Theorem 3.1]{CI_loopGLn}, \cite[Theorem 5.2.1]{CI_DrinfeldStrat}}]
Let $r > 0$ and let $\theta, \theta' \from S_{0:r+} \to \overline \QQ_\ell^\times$ be any character. Then
\begin{equation*}
\bigl\langle R_{\bS_r}^{\bG_r}(\theta), R_{\bS_r}^{\bG_r}(\theta') \bigr\rangle_{G_{\x,0:r+}} = \sum_{w \in W_{G_{\x,0}}(\bfS)} \langle \theta, \theta'^{w} \rangle_{S_{0:r+}}.
\end{equation*}
Furthermore, $H_c^*(X_r \cap \bS_r \bG_r^+, \overline \QQ_\ell)[\theta]$ is an irreducible representation of $S_0 G_{\x,0+}$.
\end{theorem}

This means in particular that we may apply Theorem \ref{thm:SG+ vreg} to $H_c^*(X_r \cap \bS_r \bG_r^+, \overline \QQ_\ell)[\theta]$ and Theorem \ref{thm:SG vreg} to $R_{\bbS_r}^{\bG_r}(\theta)$ to any toral character $\theta \from S \to \overline \QQ_\ell^\times$ of depth $r\in\R_{>0}$ whenever the assumption (\ref{ineq}) is satisfied.
Recall that the assumption (\ref{ineq}) is needed in order to guarantee that there is no character of $\bbS(\FF_q)$ which restricts to some constant $c \neq 1$ on all of $\bbS(\FF_q)_{\vreg}$ (Lemma \ref{lem:greater-than-2}). 
In \cite[Section 2.7]{MR1235293}, Henniart proves that the stronger condition $|\bbS(\FF_q)|/|\bbS(\F_{q})_{\nvreg}| > 2n$ holds unless $(q,n) \in \{(2,2), (2,4), (2,6), (3,2)\}$. If $q = 3$ and $n = 2$, then $|\bbS(\FF_q)|/|\bbS(\F_{q})_{\nvreg}| = 8/2 = 4 > 2$, so we may remove this case from the list. We will now argue that in the remaining three exceptional cases, there is no character of $\bbS(\FF_q)$ which restricts to $c \neq 1$ on all of $\bbS(\FF_q)_{\vreg}$. We have $\bbS(\FF_q) \cong \FF_{q^n}^\times$ which is cyclic of order $q^n - 1$; let $\zeta$ be a generator of $\bbS(\FF_q)$. If $q$ is even, then $\zeta$ has odd order, and so $\zeta^2$ must also be a generator of $\bbS(\FF_q)$ and must be in $\bbS(\FF_q)_{\vreg}$. It follows that if a character of $\bbS(\FF_q)$ restricts to a constant on $\bbS(\FF_q)_{\vreg}$, then the character must be the trivial representation. In summary, in our present setting of $\bfG$ being an inner form of $\GL_n$, Theorems \ref{thm:SG+ vreg} and \ref{thm:SG vreg} hold for all odd $q$ (i.e., without assumption (\ref{ineq})).

\begin{theorem}\label{thm:GLn compare}
Let $\theta$ be a toral character of $S$ of depth $r > 0$.
We put $\phi \colonequals \theta \varepsilon^{\ram}[\theta]$.
Then we have
\begin{align*}
R_{\bS_r}^{\bG_r}(\theta) &\cong (-1)^{r(\bfS, \phi)} \Ind_{K^{d}}^{SG_{\x,0}}(\rho_d' \otimes \phi_d), \\
H_c^*(X_r \cap \bS_r \bG_r^+, \overline \QQ_\ell)[\theta] &\cong (-1)^{r(\bfS, \phi)} \Ind_{K^{d}}^{SG_{\x,0+}}(\rho_d' \otimes \phi_d), \\
R_{\bS_r}^{\bG_r}(\theta) &\cong \Ind_{SG_{\x,0+}}^{SG_{\x,0}}\big(H_c^*(X_r \cap \bS_r \bG_r^+, \overline \QQ_\ell)[\theta]\big).
\end{align*}
\end{theorem}

The twist $\varepsilon^{\ram}[\theta]$ is subtle, and we will discuss it later in this subsection. Before we do that, we note that Theorem \ref{thm:GLn compare} in particular yields an \textit{algebraic} alternative to the geometric approach of \cite{CI_loopGLn} to establish the irreducibility of the $\theta$-eigenspace of loop Deligne--Lusztig varieties:

\begin{cor}\label{cor:GLn sc}
Let $\theta$ be a toral character of $S$ of depth $r > 0$.
Then $\cInd_{SG_{\x,0}}^{G}(R_{\bS_r}^{\bG_r}(\theta))$ is (up to sign) irreducible supercuspidal. In particular, the assumption $p>n$ in \cite[Theorem A]{CI_loopGLn} can be relaxed to $p >2$.
\end{cor}

The $p>2$ assumption likely cannot be relaxed due to our approach of comparing with Yu's construction, which uses from the start that $p$ is odd. However, given Henniart's work \cite{MR1235293}, it seems likely that one can obtain the above result for $p=2$ by considering G\'erardin's construction \cite{MR546594} in place of Yu's in this setting.

To finish this subsection, we would like to illustrate the subtle nature of $\varepsilon^{\ram}[\theta]$ and the sign $(-1)^{r(\bfS, \phi)}$ by further specializing the present setting to the case $\bfG = \GL_n$. In this setting, each twisted Levi $\bfG^i$ is $\Res_{E_i/F} \GL_{n_i}$, where $E_i$ is the unramified degree-$n/n_i$ extension of $F$.
If $\gamma \in S_{0,\vreg}$ is such that the image of $\gamma$ in $S_{0:0+} \cong \FF_{q^n}^\times$ generates all of $S_{0:0+}$, then $\varepsilon^{\ram}[\theta]$ is the unique tamely ramified character of $S$ such that $\varepsilon^{\ram}[\theta]|_{Z_{\bfG}}=\mathbbm{1}$ and
\begin{equation*}
\varepsilon^{\ram}[\theta](\gamma) = \prod_{i=0}^{d-1} (-1)^{(\floor{n_{i+1}/2} - \floor{n_i/2})(r_i-1)}.
\end{equation*}
Recalling the definition of $\varepsilon^{\ram}[\theta]$ from Definition \ref{def:varepsilon ram}, the above formula comes from the calculation that:
\begin{align*}
|(\Gamma_F \times \{\pm 1\}) \backslash (R_{\x, r/2}^{\bfG} \cap R(\bfG, \bfS)^{\sym})| &= 
\begin{cases}
\floor{(n-1)/2} & \text{if $r$ is even}, \\
0 & \text{otherwise},
\end{cases} \\
|\Gamma_F \backslash (R_{\x,r/2}^{\bfG} \cap R(\bfG, \bfS)_{\sym, \ur})| 
&=
\begin{cases}
1 & \text{if $r$ is even and $n$ is even}, \\
0 & \text{otherwise}.
\end{cases}
\end{align*}

The calculation of the sign $(-1)^{r(\bfS, \phi)}$ is slightly less delicate.
We first note that $r(\bfS, \phi)$ equals $r(\bfS, \theta)$ since $r(\bfS, \phi)$ depends only on the positive-depth part $\phi|_{S_{0+}}$ by definition (Proposition \ref{prop:rho_d'}) and we have $\phi|_{S_{0+}}=\theta|_{S_{0+}}$.
We can calculate that 
\begin{equation*}
|\Gamma_F \backslash R_{\x, r/2}^{\bfG}| = \begin{cases} n-1 & \text{if $r$ is even,} \\ 0 & \text{otherwise.} \end{cases}
\end{equation*}
and therefore 
\begin{equation*}
r(\bfS, \phi)=r(\bfS, \theta) \equiv \sum_{i=0}^{d-1} (n_{i+1} - n_i)(r_i-1) \pmod 2.
\end{equation*}
It is known (\cite[Theorem 6.1.1]{CI_DrinfeldStrat}) that there is a unique and explicitly described cohomological degree $r_\theta$ such that $H_c^{r_\theta}(X_r \cap \bS_r \bG_r^+, \overline \QQ_\ell)[\theta] \neq 0$. 
Moreover, if $\bfG^0(\bfS, \theta) = \bfS$, then Theorem \ref{thm:GLn compare} guarantees that $r_\theta \equiv r(\bfS, \phi)$ modulo $2$. We can see this explicitly as well: 
\begin{align*}
r(\bfS, \phi) 
&= \sum_{i=0}^{d-1} (n_{i+1} - n_i)(r_i - 1) \\
&\equiv (r_d + 1)(n_d) - n_0(r_0 + 1) + \sum_{i=1}^d n_i(r_{i-1} - r_i) \pmod{2}, 
\end{align*}
and if $n_0 = 1$ (i.e., if $\bfG^0(\bfS, \theta) = \bfS$), then this last formula is almost exactly equal to $r_\theta$ in \cite[Corollary 6.1.2]{CI_DrinfeldStrat}, the only difference being summands with even values. From the rearranging of the above sum, it seems that the algebraically calculated number $r(\bfS, \phi)$ arises from induction on the twisted Levis $\bfG^i$ while the geometrically calculated number $r_\theta$ arises from induction on the depths $r_i$.

These calculations can be made similarly for $\bfG$ being any inner form of $\GL_n$, keeping in mind the further calculations that if $\bfG$ is a nonsplit form of $\GL_n$, the point $\bar{\x} \in \cB^{\red}(\bfG, F)$ corresponding to $\bfS \hookrightarrow \bfG$ is not superspecial in the sense of \cite[Definition 3.4.8]{MR4013740}.

\subsection{Discussion of small residual characteristic}\label{sec:small p}

We illustrate some of the subtleties that arise when considering small $p$ in a particular example. We consider the example that $F$ has residual characteristic $p = 2$ and $\bfG = \SL_2$. 
In this case, Yu's construction is not available since it requires the oddness of $p$ to utilize the theory of Weil representations for Heisenberg groups over $\F_{p}$ (\cite[Sections 10--]{MR1824988}).
When $p=2$, the notion of Yu-datum still makes sense, even if Yu's construction cannot be used to build a supercuspidal from such data.
On the other hand, when $p=2$ and $\bfG=\SL_{2}$, there does not even exist a toral Yu-datum!
To see this, we let $\bfS$ be the unramified elliptic maximal torus of $\bfG$; note that $\bfS$ splits over the degree-$2$ unramified extension of $F$ which we denote by $E$. In this setting, one can compute that $S_{r:r+} = \ker(E_{r:r+}^\times \to F_{r:r+}^\times) = F_{r:r+}^\times$. In particular, every character $\theta$ of $S_{r:r+}$ is stabilized by all of $W_G(\bfS) = \Gal(E/F)$. A closely related observation is that $W_G(\bfS)$ acts trivially on the reduction-modulo-$2$ of the character lattice. 
This implies that the condition (GE2) (see \cite[Sections 8]{MR1824988}), which is one of the two conditions for that a character $\theta$ of $S$ is $\bfG$-generic of depth $r$, is never satisfied.

Now let us move to the geometric side. Let $\theta \from S \to \overline \QQ_\ell^\times$ be a character of depth $r$. In this setting, it is automatic that $\theta|_{S_{r:r+}}$ is nontrivial and $\theta|_{\Nr_{E/F}(\alpha^\vee(E_{r}^\times))} \neq \mathbbm{1}$ for all $\alpha \in R(\bfS, \bfG)$. By Theorem \ref{thm:mackey}, 
\begin{equation*}
\bigl\langle R_{\bbS_r, \bU_r}^{\bG_r}(\theta), R_{\bbS_r, \bU_r}^{\bG_r}(\theta) \bigr\rangle_{\SL_2(\cO_F/\mfp_{F}^{r+1})} = 2.
\end{equation*}
It necessarily follows that $R_{\bbS_r, \bU_r}^{\bG_r}(\theta)$ consists of two non-isomorphic irreducible representations $\tau_1, \tau_2$. It seems possible to guess that $\cInd_{\SL_2(\cO_F)}^{\SL_2(F)}(\tau_1)$ and $\cInd_{\SL_2(\cO_F)}^{\SL_2(F)}(\tau_2)$ are non-isomorphic irreducible (and therefore supercuspidal) representations of $\SL_2(F)$. 

Our expectation is that in general the geometric representations $R_{\bbS_r, \bU_r}^{\bG_r}(\theta)$ realize representations which \textit{do not} appear in Yu's construction.
In the above setting, we expect that $R_{\bbS_r, \bU_r}^{\bG_r}(\theta)$ should be a representation constructed by G\'erardin \cite{MR0396859} \cite{MR385015}. In some cases, this is a theorem of Chen--Stasinski \cite{MR3703469}; their technique works for all $p$, and specializing their result to the setting of $\SL_2$ and $p=2$, Chen and Stasinski prove that if $r$ is odd, then $R_{\bbS_r, \bU_r}^{\bG_r}(\theta)$ coincides with the representation corresponding to the character $\theta|_{S_0}$ in G\'erardin's construction. It seems reasonable to us that for general depths, the representation $R_{\bbS_r, \bU_r}^{\bG_r}(\theta)$ should also correspond to one of G\'erardin's representations $\theta\varepsilon|_{S_0}$, where $\varepsilon$ is a twisting character which should be independent of $\theta$ and $p$. (For large $p$, Theorem \ref{thm:SG vreg} can be applied to compare G\'erardin's construction with Yu's construction.)

\section{Geometric $L$-packets of toral supercuspidal representations}\label{sec:Lpacket}

In this section, we examine the implications of our results within the context of the local Langlands correspondence. We write $W_F$ for the Weil group, $I_F$ the inertia group, and $P_F$ the wild inertia group. 
We let $I_{F}^{r}$ denote the $r$-th upper ramification filtration of $I_{F}$ for $r\in\R_{\geq0}$.
Let $\widehat \bfG$ be the Langlands dual group for $\bfG$ and write ${}^L G = \widehat \bfG \rtimes W_F$. Following Kaletha \cite{MR4013740}, we assume $F$ has characteristic $0$ in this section. 

We first review Kaletha's construction of $0$-toral supercuspidal $L$-packets. As in \cite[Definition 6.1.1]{MR4013740}, define:

\begin{definition}
A \textit{0-toral supercuspidal parameter of generic depth $r>0$} is a discrete Langlands parameter $\varphi \from W_F \to {}^L G$ satisfying the following conditions:
\begin{enumerate}[label=(\arabic*)]
\item
The centralizer of $\varphi(I_F^r)$ in $\widehat \bfG$ is a maximal torus and contains (the projection from ${}^{L}G$ to $\widehat{\bfG}$ of) $\varphi(P_F)$.
\item
$\varphi(I_F^{r+})$ is trivial, i.e., $\varphi(\sigma)=1\rtimes\sigma$ for any $\sigma\in I_{F}^{r+}$.
\end{enumerate}
\end{definition}

In the present paper, our focus is on Howe-unramified supercuspidal representations. If a $0$-toral supercuspidal parameter $\varphi$ of generic depth $r>0$ is such that the centralizer of $\varphi(I_F^r)$ in $\widehat \bfG$ corresponds to an unramified torus in $\bfG$ \cite[Section 5.1]{MR4013740}, then we additionally say that $\varphi$ is \textit{Howe-unramified}. 

Following \cite[Definition 5.2.4, Section 6.1]{MR4013740}, we define a \textit{Howe-unramified 0-toral supercuspidal $L$-packet datum of depth $r$} to be a tuple $(\bfS, \jhat, \chi, \theta)$ consisting of 
\begin{enumerate}[label=\textbullet]
\item
$\bfS$ an unramified torus of dimension equal to the absolute rank of $\bfG$, defined over $F$ with anisotropic quotient $\bfS/\bfZ_{\bfG}$
\item
$\jhat \from \widehat{\bfS} \to \widehat \bfG$ is an embedding of complex reductive groups whose $\widehat \bfG$-conjugacy class is $\Gamma_F$-stable
\item
$\chi$ is a minimally ramified $\chi$-data for $R(\bfS, \bfG)$
\item
$\theta \from S \to \bbC^\times$ is a $0$-toral character of depth $r$.
\end{enumerate}
Here, $\chi$ is canonical because $\bfS$ is unramified \cite[Definition 4.6.1]{MR4013740}, so we will omit it from the tuple. By \cite[Proposition 6.1.2]{MR4013740}, there is a bijection 
\begin{equation*}
(\bfS, \jhat, \theta) \mapsto \varphi_{(\bfS, \jhat, \theta)}
\end{equation*}
between isomorphism classes of Howe-unramified $0$-toral supercuspidal $L$-packet data of depth $r$ and $\widehat \bfG$-conjugacy classes of Howe-unramified $0$-toral supercuspidal parameters of generic depth $r$. Explicitly, the Langlands parameter is the composition
\begin{equation*}
\varphi_{(\bfS, \jhat, \theta)} \from W_F \overset{\varphi_\theta}{\longrightarrow} {}^L S \overset{{}^L j}{\longrightarrow} {}^L G
\end{equation*}
where $\varphi_{\theta}\colon W_{F}\rightarrow {}^{L}S$ is the $L$-parameter of $\theta$ and ${}^L j \from {}^L S \to {}^L G$ is the Langlands--Shelstad extension of $\jhat$ determined by the canonical $\chi$-data. A \textit{Howe-unramified $0$-toral supercuspidal datum of depth $r$} \cite[Definition 5.3.2]{MR4013740} is a tuple $(\bfS, \jhat, \theta, j)$ where $(\bfS, \jhat, \theta)$ is a Howe-unramified $0$-toral supercuspidal $L$-packet datum of depth $r$ and $j \from \bfS \hookrightarrow \bfG$ is an $F$-rational embedding admissible for $\jhat$. 
We crucially observe that the set of all Howe-unramified $0$-toral supercuspidal datum of depth $r$ corresponding to a fixed $(\bfS, \jhat, \theta)$ is indexed by $G$-conjugacy classes within the stable conjugacy class of a(ny) $F$-rational embedding $j \from \bfS \hookrightarrow \bfG$ admissible for $\jhat$.

For such an embedding $j \from \bfS \hookrightarrow \bfG$, let $\x_j \in \cB(\bfG, F)$ be a point in the apartment of $j\bfS\colonequals j(\bfS)$. We let $\bS_{j,r} \subset \bG_{j,r}$ be the group schemes defined over $\FF_q$ associated to $j$ as in Section \ref{sec:parahoricDL} and we assume the inequality \eqref{ineq} in Section \ref{subsec:vreg}.
By combining our geometric comparison theorem (Theorem \ref{thm:comparison SG}) with Kaletha's construction of toral supercuspidal $L$-packets, we obtain the following result:

\begin{theorem}\label{thm:geom L packets}
Let $(\bfS, \jhat, \theta)$ be a Howe-unramified $0$-toral supercuspidal $L$-packet datum of depth $r$. Then the corresponding Langlands parameter $\varphi_{(\bfS, \jhat, \theta)}$ has $L$-packet 
\begin{equation*}
\bigl\{\cInd_{jS \cdot G_{\x_j,0}}^G(|R_{\bbS_{j,r}}^{\bbG_{j,r}}(j\theta)|)\bigr\}_{j \in \mathcal{J}^{\bfG}_{G}},
\end{equation*}
where 
\begin{itemize}
\item
$j\theta\colonequals \theta\circ j^{-1}$, 
\item
$|R_{\bbS_{j,r}}^{\bbG_{j,r}}(j\theta)|$ denotes $(-1)^{r(j\bfS,j\theta)}R_{\bbS_{j,r}}^{\bbG_{j,r}}(j\theta)$, and
\item
$\mathcal{J}^{\bfG}_{G}$ is a set of representatives of the $G$-conjugacy classes of the stable conjugacy class within $F$-rational embeddings $\bfS \hookrightarrow \bfG$ admissible for $\jhat$.
\end{itemize}
\end{theorem}

\begin{proof}[Proof of Theorem \ref{thm:geom L packets}]
Kaletha's construction of regular supercuspidal $L$-packets assigns to $\varphi_{(\bfS, \jhat, \theta)}$ the finite set of supercuspidal representations of $G$ given by $\{\pi_{(j\bfS,j\theta')}\}_{j\in\mathcal{J}^{\bfG}_{G}}$ where $j\theta'$ is a particular twist of the character $j\theta$ of $jS\colonequals j\bfS(F)$. 
In \cite[Step 3 in Section 5.3]{MR4013740}, this twist is written as $j\theta' = (\theta \cdot \zeta_S^{-1})\circ j^{-1} \cdot \epsilon_{f, \ram} \cdot \epsilon^{\ram}$; Kaletha's $\epsilon^{\ram}$ is the character $\varepsilon^{\ram}[j\theta]$ with respect to $(j\bfS\subset\bfG,j\theta)$ in Section \ref{sec:AS}, and because we assume $\bfS$ is unramified, the other twists $\zeta_S$ and $\epsilon_{f,\ram}$ are trivial. Therefore by Theorem \ref{thm:comparison SG}, $\pi_{(j\bfS, j\theta')}$ is isomorphic to the geometrically arising supercuspidal representation $\cInd_{jS \cdot G_{\x_j,0}}^G(|R_{\bS_{j,r}}^{\bG_{j,r}}(j\theta)|)$.
\end{proof}

The fact that the geometric description does not need to be separately twisted by $\varepsilon^{\ram}[j\theta]$ has nontrivial significance. Kaletha's $L$-packets $\{\pi_{(j\bfS, j\theta')}\}_{j\in\mathcal{J}^{\bfG}_{G}}$ are \textit{not} the same as $\{\pi_{(j\bfS, j\theta)}\}_{j\in\mathcal{J}^{\bfG}_{G}}$; these packets can be genuinely different (see \cite[Example 5.5]{MR3849622} for an explicit example). 
In fact, it is known that although $\{\pi_{(j\bfS, j\theta)}\}_{j\in\mathcal{J}^{\bfG}_{G}}$ appears more canonical from the perspective of Yu's construction and Kaletha's generalization of the Howe factorization, these sets of supercuspidals are \textit{not} (!) stable and therefore cannot be $L$-packets.

Let us elaborate on stability here. 
In general, it is expected that the local Langlands correspondence not only associates an $L$-packet $\Pi_{\varphi}^{\bfG}$ to each $L$-parameter $\varphi$, but also gives a parametrization of members of $\Pi_{\varphi}^{\bfG}$ in terms of a certain finite group $\mathcal{S}_{\varphi}$ determined by $\varphi$; members of $\Pi_{\varphi}^{\bfG}$ are expected to be labelled by irreducible characters of $\mathcal{S}_{\varphi}$.
If we let $\langle\pi,-\rangle$ denote the irreducible character of $\mathcal{S}_{\varphi}$ corresponding to $\pi\in\Pi_{\varphi}^{\bfG}$, then the stability of the $L$-packet $\Pi_{\varphi}^{\bfG}$ asserts that the linear combination of Harish-Chandra characters
\[
\sum_{\pi\in\Pi_{\varphi}^{\bfG}}\langle\pi,1\rangle \cdot \Theta_{\pi}
\]
is stable, i.e., constant on every stable conjugacy class of strongly regular semisimple elements (see \cite[Section 6]{MR1021499} for a general formulation of the stability).
When $\varphi$ is a $0$-toral supercuspidal parameter (or, more generally, regular supercuspidal parameter), the associated group $\mathcal{S}_{\varphi}$ is abelian (see \cite[Section 5.3]{MR4013740}).
Hence, for such an $L$-parameter, the stability simply asserts that the sum $\sum_{\pi\in\Pi_{\varphi}^{\bfG}}\Theta_{\pi}$ should be stable.
Although the stability alone cannot characterize the local Langlands correspondence, it has an important role as a touchstone in verifying the validity of the construction of the $L$-packets.

Given this, 
is it naturally pressing to ask:
\begin{quote}
Does the correspondence $(j\colon\bfS \hookrightarrow \bfG, \theta) \mapsto \pi_{(j\bfS, j\theta)}$ map stable conjugacy classes to sets of supercuspidals with stable character sums?
\end{quote}
In the setting of Howe-unramified $0$-toral supercuspidal representations, this question was posed by Reeder \cite{MR2427973}, who proved that this correspondence maps stable conjugacy classes to sets of supercuspidals with constant formal degree. 
DeBacker--Spice \cite{MR3849622} proved that the answer is in fact no and defined a twisting character $\varepsilon^{\ram}$. 
Under the additional assumption that $F$ has characteristic zero with sufficiently large residual characteristic, DeBacker--Spice proved that the twisted correspondence $(j\colon\bfS \hookrightarrow \bfG, \theta) \mapsto \pi_{(j\bfS, j\theta\cdot \varepsilon^{\ram}[j\theta])}$ does in fact preserve stability (\cite[Theorem 5.10]{MR3849622}). 
Kaletha \cite{MR4013740} defined twisting characters in the more general setting of regular supercuspidal representations and proved the associated stability preservation assertion for $0$-toral supercuspidal representations under the same assumptions on $F$ as in DeBacker--Spice (\cite[Theorem 6.3.2]{MR4013740}).
This fact is one of strong evidences for the validity of Kaletha's construction of the local Langlands correspondence.

The content of the next theorem is that if we replace the correspondence $(j\colon\bfS \hookrightarrow \bfG, \theta) \mapsto \pi_{(j\bfS, j\theta)}$ with the geometric construction $(j\colon\bfS \hookrightarrow \bfG, \theta) \mapsto \cInd_{jS \cdot G_{\x_j,0}}^G(|R_{\bbS_{j,r}}^{\bbG_{j,r}}(j\theta)|)$, then we do not need to separately define a twisting character. 
Theorem \ref{thm:stability} is a corollary of results we have already established in this paper (and \cite[Theorem 5.10]{MR3849622} or \cite[Theorem 6.3.2]{MR4013740}), but we would like to emphasize and repeat the following point mentioned in the introduction: the geometry seems to innately know about automorphic side of the local Langlands correspondence.

\begin{theorem}\label{thm:stability}
Assume additionally that $F$ has characteristic zero with residual characteristic $p \geq (2+e)n$ where $e$ is the ramification degree of $F$ over $\bbQ_p$ and $n$ is the dimension of the smallest faithful rational representation of $\bfG$. 
The correspondence
\begin{equation*}
(j\colon\bfS \hookrightarrow \bfG, \theta) \mapsto \cInd_{jS \cdot G_{\x_j,0}}^G(|R_{\bbS_{j,r}}^{\bbG_{j,r}}(j\theta)|)
\end{equation*}
preserves stability for $0$-toral characters $\theta$.
\end{theorem}

\section{Regular supercuspidal representations characterized by $S_{\vreg}$}\label{sec:vreg characterization}

In this section, we prove that certain regular supercuspidal representations are determined by their Harish-Chandra character on unramified very regular elements. 
These results can be viewed as versions of the comparison results of Section \ref{sec:comparison} in the setting that the group $SG_{\x,0}$ is replaced by the $p$-adic group $G$; on the other hand, neither result implies the other logically (see Remark \ref{rem:Gx0 vs G}). 
The advantage to obtaining a characterization result at the level of $G$ is that it allows one to characterize members of certain $L$-packets by their Harish-Chandra character on a very small collection of elements of $G$, as one does for real groups \cite{MR1011897}, resolving a hole mentioned by Kaletha around \cite[(5.3.3)]{MR4013740}. 
Such a characterization problem was also mentioned several years earlier by Adler--Spice \cite[Section 0.3]{MR2543925}, motivated by Henniart's characterization of certain supercuspidal representations of $\GL_n$ \cite{MR1235293,MR1263525}. 

The class of regular supercuspidals for which we establish this characterization are those which correspond to tame elliptic regular pairs $(\bfS, \phi)$ where $\bfS$ is unramified and $\bfG^0(\bfS, \phi) = \bfS$---that is, precisely the class of Howe-unramified toral supercuspidal representations. We will additionally assume in this section that $\bfS$ is such that the inequality \eqref{ineq} introduced in Section \ref{subsec:vreg} is satisfied.

We prove the following theorem in Section \ref{subsec:unram pi vreg}. 

\begin{theorem}\label{thm:unram pi vreg}
Let $\bfS$ be an elliptic unramified maximal torus of $\bfG$ and let $\theta \from S \to \bbC^\times$ be a toral character. Then there is a unique regular supercuspidal representation $\pi$ of $G$ such that for every $\gamma \in S_{\vreg}$,
\begin{equation}\label{e:HC vreg}
\Theta_\pi(\gamma) = c \cdot \sum_{w \in W_G(\bfS)} \theta({}^w \gamma)
\end{equation}
for a nonzero constant $c \in \bbC$.
Furthermore, $\pi \cong \pi_{(\bfS, \phi)}$ with $\phi:=\theta\cdot\varepsilon^{\ram}[\theta]$ and we must have $c = (-1)^{r(\bbG_{\x}^0) - r(\bbS) + r(\bfS, \phi)}$, where $\x\in\mcB(\bfG,F)$ is a point associated to $\bfS$ and the exponent of $(-1)$ is as in Proposition \ref{prop:rho_d'}.
\end{theorem}

Theorem \ref{thm:unram pi vreg} allows us to formulate the construction of Kaletha's $L$-packets in the following way. Let $j \from \bfS \hookrightarrow \bfG$ be an unramified elliptic maximal torus defined over $F$ and let $\theta \from S \to \bbC^\times$ be a toral character.
For any $F$-rational embedding $j' \from \bfS \hookrightarrow \bfG$ stably conjugate to $j$, let $\pi_{j'}$ be \textit{the} regular supercuspidal representation of $G$ with Harish-Chandra character
\begin{equation*}
\Theta_{\pi_{j'}}(\gamma) = c \cdot \sum_{w \in W_G(j'\bfS)} j'\theta({}^w \gamma), \qquad \text{for $\gamma \in j' S_{\vreg}$,}
\end{equation*}
where $j'\theta:=\theta\circ j'^{-1}$ and $c$ is some nonzero constant. Then the $L$-packet Kaletha constructs \cite[Section 5.3]{MR4013740} is precisely the collection of all $\pi_{j'}$; moreover, the Langlands parameter corresponding to this $L$-packet is the homomorphism $\varphi_{(\bfS, \widehat j, \theta)} \from W_F \to {}^L G$ recalled in Section \ref{sec:Lpacket}. The contribution of Theorem \ref{thm:unram pi vreg} in this context is exactly the italicized \textit{the} above; that is, that members of certain $L$-packets are characterized by their Harish-Chandra character on unramified very regular elements.

\begin{remark}\label{rem:shallow vs vreg}
Kaletha actually asks for something slightly different---that regular supercuspidal representations are characterized by their characters on \textit{shallow} elements (see around \cite[(5.3.3)]{MR4013740}). Not all shallow elements are unramified very regular (because shallow elements can have connected centralizer being equal to a torus which is not unramified), and not all unramified very regular elements are shallow (because shallow elements necessarily have order coprime to $p$). A key point here is that if $\gamma$ is unramified very regular, then any element of $\gamma (T_\gamma)_{0+}$ is also unramified very regular.
But this is not the case for shallow elements! Pushing further in this direction, it is in fact possible to find two nonisomorphic regular supercuspidal representations with identical character values on shallow elements.
Indeed, if we take two non-$G$-conjugate tame elliptic regular pairs $(\bfS,\theta)$ and $(\bfS,\theta')$ such that the regular generic reduced cuspidal $\bfG$-data associated to them have the same depth zero part (but have different positive depth part), then we cannot distinguish the regular supercuspidal representations $\pi_{(\bfS,\theta)}$ and $\pi_{(\bfS,\theta')}$ by their characters at shallow elements.
So, in order for Kaletha's desired characterization to hold in general (i.e., outside the Howe-unramified setting), one must replace ``shallow'' by a generalized notion of very regularity.
\end{remark}

\begin{remark}\label{rem:Gx0 vs G}
We note that Theorem \ref{thm:unram pi vreg} is not strong enough to obtain Theorem \ref{thm:geom L packets} without the work of Sections \ref{sec:ell unram vreg}, \ref{sec:comparison}. The point here is without Sections \ref{sec:ell unram vreg},\ref{sec:comparison}, we would not know the irreducibility nor the supercuspidality of the induced representation $\cInd_{S \cdot G_{\x,0}}^G(|R_{\bbS_r}^{\bbG_r}(\theta)|)$.

On the flip side, the results of Section \ref{sec:ell unram vreg} are not enough to obtain Theorem \ref{thm:geom L packets} because the results of Section \ref{sec:ell unram vreg} characterize representations at the level of parahoric subgroups. 
\end{remark}

\subsection{Character formula on unramified very regular elements}
In order to prove Theorem \ref{thm:unram pi vreg} we will first need a character formula for regular supercuspidal representations on unramified very regular elements. Such a formula is not contained in the work of Adler--Spice \cite[Theorem 6.4]{MR2543925} because their formula requires the compactness assumption that $\bfG^{d-1}/\bfZ_\bfG$ is $F$-anisotropic, and this hypothesis is not satisfied by every regular supercuspidal representation. In \cite[Section 4.4]{MR4013740}, Kaletha establishes a character formula for shallow elements without this compactness condition. Given the comments in Remark \ref{rem:shallow vs vreg}, neither Kaletha's nor Adler--Spice's formulas suffice for us.

In this section, we prove:

\begin{proposition}\label{prop:char at vreg}
Let $(\bfS', \phi)$ be a tame elliptic regular pair of $\bfG$ and let $\x' \in \cB(\bfG, F)$ be a point associated to $\bfS' \hookrightarrow \bfG$.
Let $\bfS$ be an unramified elliptic maximal torus of $\bfG$.
When $\bfS'$ is not $G$-conjugate to $\bfS$, we have $\Theta_{\pi_{(\bfS', \phi)}}(\gamma)=0$ for any unramified very regular element $\gamma \in S_{\vreg}$.
When $\bfS'$ is $G$-conjugate to $\bfS$ (in this case, we assume that $\bfS$ itself equals $\bfS'$), for any unramified very regular element $\gamma \in S_{\vreg}$, we have
\[
\Theta_{\pi_{(\bfS', \phi)}}(\gamma)
=
(-1)^{r(\bbG_{\x}^0) - r(\bbS) + r(\bfS, \phi)} \sum_{w \in W_G(\bfS)} \varepsilon^{\ram}[\phi]({}^w \gamma) \cdot \phi({}^w \gamma).
\]
\end{proposition}

Before we prove the main result of this subsection (Proposition \ref{prop:char at vreg}) which we will use to prove Theorem \ref{thm:unram pi vreg} in the next section, let us fix some notation. From now until the end of the paper, we will use the following notation.
Let $(\bfS', \phi)$ be a tame elliptic regular pair with an associated point $\x'\in\mathcal{B}(\bfG,F)$ and let $(\vec{\bfG}, \pi_{-1}, \vec{\phi})$ be a corresponding regular generic reduced cuspidal $\bfG$-datum (Section \ref{subsec:Kaletha-TER}).
We caution that at this point $\bfS'$ might not be unramified nor $\phi$ might be toral.
We let $\cc \rho_i'$, $\tilde \phi_i$ denote the ``intermediate'' representations arising in Yu's construction of the supercuspidal representation associated to $(\vec{\bfG}, \pi_{-1}, \vec{\phi})$ (see Sections \ref{subsec:Yu}, \ref{subsec:stab-parah} for recollections). We remind the reader that $\cc \rho_i'$ is a representation of $\cc K^i = SG_{\x',0}^0(G^0, \ldots, G^i)_{\x',(s_0,\ldots, s_{i-1})}$ and $\tilde \phi_i$ is a representation of $G_{\bar{\x}'}^i \ltimes J^{i+1}$ where $J^{i+1} = (G^i, G^{i+1})_{\x',(r_i,s_i)}$.
Furthermore, we fix an unramified elliptic maximal torus $\bfS$ of $\bfG$.

\begin{proof}[Proof of Proposition \ref{prop:char at vreg}]

We first note that the central character of $\pi_{(\bfS',\phi)}$ is given by $\phi|_{Z_{\bfG}}$ (\cite[Fact 3.7.11]{MR4013740}) and $\varepsilon^{\ram}[\phi]|_{Z_{\bfG}}=\mathbbm{1}$ by definition (Definition \ref{def:varepsilon ram}).
Moreover, since $\bfS$ is unramified, we have $S=S_{0}Z_{\bfG}$.
Thus it is enough to treat the case where $\gamma$ belongs to $S_{0,\vreg}$.

Recall (Sections \ref{subsec:Yu}, \ref{subsec:stab-parah}) that we have
\begin{equation*}
\pi_{(\bfS', \phi)} = \cInd_{K_{\sigma}}^{G}(\sigma \otimes \phi_d) = \cInd_{K_\sigma}^G(\sigma) \otimes \phi_d
\end{equation*}
where $K_\sigma = K_{\sigma_d} = G_{\bar{\x}'}^{d-1} G_{\x',0+}$ and $\sigma = \sigma_d = \Ind_{K^d}^{K_{\sigma_d}}(\rho_d').$ Let $\Theta_\sigma$ denote the character of $\sigma$ and let $\dot{\Theta}_\sigma$ denote the extension by zero of $\Theta_\sigma$ to a function of $G$.

\begin{claim}\label{claim:1}
For any unramified very regular element $\gamma \in S_{\vreg}$,
\begin{equation*}
\Theta_{\pi_{(\bfS', \phi)}}(\gamma)
=
\phi_{d}(\gamma)\sum_{g \in \cc K^{d} \backslash G_{\bar{\x}'}}  \dot{\Theta}_{\cc\rho'_{d}}({}^{g}\gamma).
\end{equation*}
\end{claim}

We first argue that the function $g \mapsto \dot{\Theta}_\sigma({}^g \gamma)$ on $G/Z_\bfG$ is supported on a single right coset $(G_{\bar{\x}'}/Z_\bfG) \cdot g_\sigma$ in $G/Z_\bfG$ (and is in particular compactly supported).
For this we follow an argument of Kaletha \cite[\S4.4]{MR4013740} (who did it under the assumption that $\gamma$ is shallow) adapted to our situation ($\gamma$ is an unramified very regular element of $S$).
By definition, we have that for any $\alpha \in R(\bfS, \bfG)$, we have $\alpha(\gamma) \not\equiv 1 \pmod{\mathfrak{p}}$, and so by \cite[Section 3.6]{MR546588}, the set of fixed points of $\gamma$ in $\cB^{\red}(\bfG, F^{\ur})$ is $\cA^{\red}(\bfT_\gamma, F^{\ur})=\cA^{\red}(\bfS, F^{\ur})$.
Hence the fixed points of $\gamma$ in the rational building $\cB^{\red}(\bfG,F)$ consists of a single point $\bar{\x}$.
The same holds for ${}^g \gamma$ in place of $\gamma$, with fixed point $g \bar{\x}$.
Obviously ${}^g \gamma$ is an element of the stabilizer subgroup $G_{\bar{\x}'}$ of $\bar{\x}'$ if and only if $\bar{\x}' = g \bar{\x}$.
In particular, unless $\bar{\x}' = g \bar{\x}$,  we have ${}^g \gamma \notin K_\sigma$ and $\dot{\Theta}_\sigma({}^g \gamma) = 0$.
Hence $g \mapsto \dot{\Theta}_\sigma({}^g \gamma)$ is supported on $\{g \in G \mid g \bar{\x} = \bar{\x}'\}$, which is either empty (in which we take $g_\sigma$ arbitrarily) or forms a single right coset of $G_{\bar{\x}'}/Z_\bfG$ in $G/Z_\bfG$ (in which case $g_\sigma$ is uniquely determined).

In the following, by replacing $(\bfS,\gamma)$ with $({}^{g_{\sigma}}\bfS,{}^{g_{\sigma}}\gamma)$, we assume that $g_{\sigma}=1$.
Note that then the point $\bar{\x}$ associated to $\bfS$ is necessarily equal to $\bar{\x}'$.

By the Harish-Chandra integral character formula, we have
\[\label{HC}
\Theta_{\pi_{(\bfS', \phi)}}(\gamma)
= 
\frac{\deg\pi_{(\bfS', \phi)}}{\dim\sigma} \phi_d(\gamma)
\int_{G/Z_\bfG} \int_{\mathcal{K}} \dot{\Theta}_\sigma({}^{gc} \gamma) \, dc \, dg,
\tag{HC}
\]
where 
\begin{itemize}
\item
$\deg\pi_{(\bfS',\phi)}$ is the formal degree of the supercuspidal representation $\pi_{(\bfS',\phi)}$ with respect to a fixed Haar measure of $G/Z_{\bfG}$,
\item
$\mathcal{K}$ is a(ny) compact open subgroup of $G$, 
\item
$dc$ is the Haar measure of $\mathcal{K}$ normalized so that $\mathcal{K}$ has measure $1$.
\end{itemize}

Since the function $g \mapsto \dot{\Theta}_\sigma({}^{gc} \gamma)$ is compactly supported, by a standard argument via Fubini's theorem, we may exchange the two integrals and get rid of the integral over $\mathcal{K}$:
\[
\int_{G/Z_\bfG} \int_{\mathcal{K}} \dot{\Theta}_\sigma({}^{gc} \gamma) \, dc \, dg
=
\int_{G/Z_\bfG} \dot{\Theta}_\sigma({}^{g} \gamma) \, dg.
\]
Since the support of the function $g\mapsto\dot{\Theta}({}^{g}\gamma)$ is contained in $G_{\bar{\x}'}/Z_{\bfG}$ (note that now $g_{\sigma}=1$), we can compute this integral as follows:
\begin{align*}
\int_{G/Z_\bfG} \dot{\Theta}_\sigma({}^{g} \gamma) \, dg
&=
\sum_{g' \in K_\sigma \backslash G_{\bar{\x}'}} \int_{K_\sigma g'/Z_\bfG} \dot{\Theta}_\sigma({}^{g} \gamma) \, dg \\
&=\sum_{g' \in K_\sigma \backslash G_{\bar{\x}'}} \meas(K_\sigma g'/Z_\bfG) \cdot \dot{\Theta}_\sigma({}^{g'} \gamma) \\
&=\meas(K_\sigma/Z_\bfG)\sum_{g \in K_\sigma \backslash G_{\bar{\x}'}}  \dot{\Theta}_\sigma({}^{g} \gamma).
\end{align*}
As the irreducible supercuspidal representation $\pi_{(\bfS',\phi)}$ is obtained by the compact induction of $\sigma$ from $K_{\sigma}$ to $G$, we have
\[
\deg\pi_{(\bfS',\phi)}=\dim\sigma\cdot\meas(K_\sigma/Z_\bfG)^{-1}
\]
(see, e.g., \cite[Theorem A.14]{MR1423022}).
Thus the formula (\ref{HC}) is simplified to 
\[
\Theta_{\pi_{(\bfS', \phi)}}(\gamma)
=
\sum_{g \in K_\sigma \backslash G_{\bar{\x}'}}  \dot{\Theta}_\sigma({}^{g}\gamma).
\]
Since we have $\sigma=\Ind_{\cc K^{d}}^{K_{\sigma}}\cc\rho'_{d}\otimes\phi_{d}$, the Frobenius formula implies Claim \ref{claim:1}.

\begin{claim}\label{claim:2}
If there is an unramified very regular element $\gamma \in S_{\vreg}$ such that $\Theta_{\pi_{(\bfS',\phi)}}(\gamma) \neq 0$, then $\bfS'$ is necessarily $G$-conjugate to the unramified elliptic maximal torus $\bfS$.
\end{claim}

If $\Theta_{\pi_{(\bfS', \phi)}}(\gamma)$ is not zero, there exists $g\in G_{\bar{\x}'}$ such that $\dot\Theta_{\cc\rho'_{d}}({}^{g}\gamma)$ is not zero by Claim \ref{claim:1}. 
By replacing $\gamma$ with ${}^{g}\gamma$, we may assume that $\Theta_{\cc\rho'_{d}}(\gamma)\neq0$.
By definition, $\cc \rho_d'$ is the representation of $\cc K^d = \cc K^{d-1} J^d$ descended from the $\cc K^{d-1} \ltimes J^d$-representation $(\tilde \phi_{d-1}|_{\cc K^{d-1} \ltimes J^d}) \otimes ((\cc \rho_{d-1}' \otimes \phi_{d-1}|_{\cc K^{d-1}}) \ltimes \mathbbm{1})$.
Hence, as we have $\gamma\in S_{0}\subset G^{0}_{\x,0}=G^{0}_{\x',0}\subset \cc K^{d-1}$, we get
\[
\Theta_{\cc \rho_d'}(\gamma) = \Theta_{\tilde \phi_{d-1}}(\gamma \ltimes 1) \cdot \Theta_{\cc \rho_{d-1}'}(\gamma) \cdot \phi_{d-1}(\gamma).
\]
Then the second factor on the right-hand side can be computed in the same way:
\[
\Theta_{\cc \rho_{d-1}'}(\gamma) = \Theta_{\tilde \phi_{d-2}}(\gamma \ltimes 1) \cdot \Theta_{\cc \rho_{d-2}'}(\gamma) \cdot \phi_{d-2}(\gamma).
\]
From this, we see that we must have $\Theta_{\cc\rho_{0}'}(\gamma)\neq0$.
If we take a topological Jordan decomposition (or a normal $(0+)$-approximation) $\gamma=\gamma_{0}\gamma_{+}$ with a topologically semisimple element $\gamma_{0}\in G^{0}_{\x',0}$ and a topologically unipotent element $\gamma_{+}\in G^{0}_{\x',0}$, then we have
\[
\Theta_{\cc\rho_{0}'}(\gamma)
=
\Theta_{\kappa_{(\bfS',\phi_{-1})}}(\gamma)
=
\Theta_{\kappa_{(\bfS',\phi_{-1})}}(\gamma_{0}).
\]
By Lemma \ref{lem:Kaletha}, this does not vanish only if $\gamma_{0}$ is $S'G^{0}_{\x',0}$-conjugate to an element of $S'$.
We take $g\in S'G^{0}_{\x',0}$ such that ${}^{g}\gamma_{0}\in S'$.
Since $\gamma$ is unramified very regular, the topologically semisimple part $\gamma_{0}$ is regular semisimple in $\bfG^{0}$.
Thus the connected centralizer $\bfT_{{}^{g}\gamma_{0}}={}^{g}\bfT_{\gamma_{0}}$ of ${}^{g}\gamma_{0}\in S'$ is equal to $\bfS'$.
On the other hand, by noting that $\gamma_{0}$ is semisimple and commutes with $\gamma_{+}$, $\gamma_{0}$ is contained in the connected centralizer $\bfT_{\gamma}$ of $\gamma$.
This implies that we have $\bfT_{\gamma_{0}}=\bfT_{\gamma}$ by the regularity of $\gamma_{0}$.
As $\bfT_{\gamma}$ is nothing but $\bfS$, we get $\bfS=\bfT_{\gamma}=\bfT_{\gamma_{0}}={}^{g^{-1}}\bfS'$, which establishes Claim \ref{claim:2}.

\mbox{}

We are now ready to finish the proof of Proposition \ref{prop:char at vreg}. From now on, we assume that $\bfS'$ is $G$-conjugate to $\bfS$.
Let us compute the character 
\[
\Theta_{\pi_{(\bfS', \phi)}}(\gamma)
=
\sum_{g \in \cc K^{d} \backslash G_{\bar{\x}'}} \dot{\Theta}_{\cc\rho'_{d}}({}^{g}\gamma).
\]
This can be done in a similar way to the proof of Proposition \ref{prop:AS-vreg}.
By taking $G$-conjugation, we may assume that $\bfS'=\bfS$.
In the following, we omit $\prime$ from the notation; we simply write $\bfS$ and $\x$.
By the argument in the previous paragraph, we see that
\[
\sum_{g \in \cc K^{d} \backslash G_{\bar{\x}}} \dot{\Theta}_{\cc\rho'_{d}}({}^{g}\gamma)
=
\sum_{\begin{subarray}{c}g \in \cc K^{d} \backslash G_{\bar{\x}} \\ {}^{g}\gamma\in S\end{subarray}} \Theta_{\cc\rho'_{d}}({}^{g}\gamma).
\]
The index set of the sum on the right-hand side can be furthermore rewritten as $\cc K^{d}\cap N_{G_{\bar{\x}}}(\bfS) \backslash N_{G_{\bar{\x}}}(\bfS)$.
Then we can check that $\cc K^{d}\cap N_{G_{\bar{\x}}}(\bfS)$ equals $N_{SG^{0}_{\x,0}}(\bfS)$ in the same way as the proof of Proposition \ref{prop:AS-vreg}.
Here note that $N_{SG^{0}_{\x,0}}(\bfS)=SN_{G^{0}_{\x,0}}(\bfS)$ and $N_{G_{\bar{\x}}}(\bfS)=N_{G}(\bfS)$ (the latter equality holds since if an element of $G$ normalizes $\bfS$, then $g$ stabilizes the point $\bar{\x}$ associated to $\bfS$).
Hence, by Proposition \ref{prop:rho_d'}, we have 
\begin{align*}
\Theta_{\pi_{(\bfS, \phi)}}(\gamma)
&=
\sum_{g\in SN_{G^{0}_{\x,0}}(\bfS) \backslash N_{G}(\bfS)}
(-1)^{r(\bbG^{0}_{\x})-r(\bbS)+r(\bfS,\phi)}
\sum_{w\in W_{G^{0}_{\x,0}}(\bfS)}
\varepsilon^{\ram}[\phi]({}^{wg}\gamma)\cdot\phi({}^{wg}\gamma)\\
&=
(-1)^{r(\bbG^{0}_{\x})-r(\bbS)+r(\bfS,\phi)}
\sum_{w\in W_{G}(\bfS)}
\varepsilon^{\ram}[\phi]({}^{w}\gamma)\cdot\phi({}^{w}\gamma). \qedhere
\end{align*}
\end{proof}

\subsection{Proof of Theorem \ref{thm:unram pi vreg}}\label{subsec:unram pi vreg}

We use the notation fixed in the previous subsection. To prove Theorem \ref{thm:unram pi vreg} we will use Proposition \ref{prop:char at vreg} together with the following lemma:

\begin{lemma}\label{lem:vreg nonzero}
Let $\theta \from S \to \bbC^\times$ be a smooth character such that $\theta|_{S_{0+}}$ has trivial $W_G(\bfS)$-stabilizer. Then there exists an element $\gamma \in S_{\vreg}$ such that 
\begin{equation*}
\sum_{w \in W_G(\bfS)} \theta({}^w \gamma) \neq 0.
\end{equation*}
\end{lemma}

\begin{proof}
The proof is very similar to that of Lemma \ref{lem:toral Henn}. We will actually prove the stronger statement that for any $\gamma \in S_{\vreg}$, there exists an element $x \in S_{0+}$ such that 
\begin{equation}\label{e:gamma x}
\sum_{w \in W_G(\bfS)} \theta({}^w (\gamma x)) \neq 0.
\end{equation}
To this end, fix $\gamma \in S_{\vreg}$ and first note that $\gamma S_{0+} \subset S_{\vreg}$. Now assume for a contradiction that \eqref{e:gamma x} is false for all $x \in S_{0+}$. This implies that we have the following identity of characters of $S_{0+}$:
\begin{equation*}
\theta|_{S_{0+}} = - \theta(\gamma)^{-1} \sum_{1 \neq w \in W_G(\bfS)} \theta({}^w \gamma) \cdot (\theta{}^w|_{S_{0+}}).
\end{equation*}
But now we have
\begin{equation*}
1 = \langle \theta|_{S_{0+}}, \theta|_{S_{0+}} \rangle = -\theta(\gamma)^{-1} \sum_{1 \neq w \in W_G(\bfS)} \theta({}^w \gamma) \cdot \langle \theta|_{S_{0+}}, \theta{}^w|_{S_{0+}} \rangle,
\end{equation*}
which is a contradiction since each summand on the right-hand side is zero by assumption.
\end{proof}

\begin{proof}[Proof of Theorem \ref{thm:unram pi vreg}]

First, the regular supercuspidal representation associated to the pair $(\bfS,\phi)$ indeed satisfies the required condition by Proposition \ref{prop:char at vreg}.
Thus our task is to show uniqueness.

Let $\pi$ be a regular supercuspidal representation of $G$ and let $(\bfS', \phi')$ be a corresponding tame elliptic regular pair (see Proposition \ref{prop:TER-pair}).
Recall that $(\bfS', \phi')$ is unique up to $G$-conjugacy. 
We assume that the representation $\pi$ satisfies the equality (\ref{e:HC vreg}) for a toral character $\theta$ of an unramified elliptic maximal torus $\bfS$ of $\bfG$.
Because $\theta|_{S_{0+}}$ has trivial $W_G(\bfS)$-stabilizer by Lemma \ref{lem:theta-stab}, Lemma \ref{lem:vreg nonzero} implies that there is an element $\gamma \in S_{\vreg}$ such that $\Theta_\pi(\gamma) \neq 0$.
By Proposition \ref{prop:char at vreg}, this implies that $\bfS$ and $\bfS'$ must be $G$-conjugate. In particular, this means we may assume $\bfS' = \bfS$, so that we now have $\pi \cong \pi_{(\bfS, \phi')}$ for some character $\phi'$ of $S$.

By Proposition \ref{prop:char at vreg}, for all $\gamma \in S_{\vreg}$, we have
\begin{equation}\label{e:char at vreg}
\Theta_{\pi_{(\bfS, \phi')}}(\gamma)
= (-1)^{r(\bbG_{\x}^0) - r(\bbS) + r(\bfS, \phi')} \sum_{w \in W_G(\bfS)} \varepsilon^{\ram}[\phi']({}^w \gamma) \cdot \phi'({}^w \gamma).
\end{equation}
It is now in relating $\phi'$ and $\theta$ that we will invoke the remaining assumption \eqref{ineq}. With this assumption, Lemma \ref{lem:toral Henn} holds after replacing every instance of $W_{G_{\x,0}}(\bfS)$ by $W_G(\bfS)$. 
Therefore we must have that $\varepsilon^{\ram}[\phi'] \cdot \phi' = \theta^w$ for some $w \in W_G(\bfS)$ and $c=(-1)^{r(\bbG_{\x}^0) - r(\bbS) + r(\bfS, \phi')}$.
As $\varepsilon^{\ram}[\phi']$ is tamely ramified, we get $\phi'|_{S_{0+}}=\theta^{w}|_{S_{0+}}$.
Since $\varepsilon^{\ram}[\phi']$ (resp.\ $\varepsilon^{\ram}[\theta^{w}]$) is determined by $\phi'|_{S_{0+}}$ (resp.\ $\theta^{w}|_{S_{0+}}$), we have $\varepsilon^{\ram}[\phi']=\varepsilon^{\ram}[\theta^{w}]$.
By noting that $\varepsilon^{\ram}$ is a sign character, we finally conclude that $\phi'=\theta^{w}\varepsilon^{\ram}[\theta^{w}] \,(=(\theta\varepsilon^{\ram}[\theta])^{w})$.
This implies that $\pi_{(\bfS,\phi')}\cong\pi_{(\bfS,\theta\varepsilon^{\ram}[\theta])}$.
\end{proof}

\end{document}